\documentclass[a4paper,reqno,12pt]{article}
\usepackage{amssymb}
\usepackage{amsfonts}
\usepackage{amsmath}
\usepackage{graphicx}
\usepackage{hyperref}

\newtheorem{theorem}{Theorem}[section]
\newtheorem{proposition}[theorem]{Proposition}
\newtheorem{lemma}[theorem]{Lemma}
\newtheorem{corollary}[theorem]{Corollary}
\newtheorem{remark}[theorem]{Remark}
\newtheorem{example}[theorem]{Example}
\newenvironment{claime}[1][Claim]{\smallskip\noindent\textsc{#1}\it}{\rm\smallskip}
\newtheorem{definition}[theorem]{Definition}

\newenvironment{case}[1][Case]{\textsl{#1}}{\smallskip}
\newenvironment{acknowledgement}{\smallskip{\sc Acknowledgments.}\rm}{\smallskip}
\renewcommand{\theequation}{\thesection.\arabic{equation}}

\input{tcilatex}

\def\RM{\rm}

\renewenvironment{proof}[1][Proof]{\textbf{#1.} }{\ \rule{0.5em}{0.5em}}

\def\limfunc#1{\mathop{\mathrm{#1}}}
\def\func#1{\mathop{\mathrm{#1}}\nolimits}

\def\tbigcup{\bigcup}

\def\dbigoplus{\displaystyle\bigoplus}

\def\enddoc{\end{document}}

\def\Qlb#1{#1}

\def\Qcb#1{#1} 

\def\FRAME#1#2#3#4#5#6#7#8
{
 \begin{figure}[ptbh]
 \begin{center}
 \includegraphics[height=#3]{#7}
 \caption{#5}
 \label{#6}
 \end{center}
 \end{figure}
}

\begin{document}

\title{Negative eigenvalues of two-dimensional Schr\"{o}dinger operators}
\author{Alexander Grigor'yan\thanks{%
Research partially supported by SFB 701 of the German Research Council (DFG)}
\\
Fakult\"{a}t f\"{u}r Mathematik\\
Universit\"{a}t Bielefeld\\
Postfach 100131\\
33501 Bielefeld, Germany \and Nikolai Nadirashvili \\
CNRS, LATP \\
Centre de Math\'{e}matiques et Informatique\\
Universit\'{e} Aix-Marseille\\
13453 Marseille, France }
\date{April 2012}
\maketitle

\begin{abstract}
We prove a certain upper bound for the number of negative eigenvalues of the
Schr\"{o}dinger operator $H=-\Delta -V$ in $\mathbb{R}^{2}.$
\end{abstract}

\tableofcontents

\section{Introduction}

\label{SecIntro} \setcounter{equation}{0}

\subsection{Main statement}

Given a non-negative $L_{loc}^{1}$ function $V\left( x\right) $ on $\mathbb{R%
}^{n}$, consider the Schr\"{o}dinger type operator%
\begin{equation*}
H_{V}=-\Delta -V
\end{equation*}%
where $\Delta =\sum_{k=1}^{n}\frac{\partial ^{2}}{\partial x_{k}^{2}}$ is
the classical Laplace operator. More precisely, $H_{V}$ is defined as a form
sum of $-\Delta $ and $-V$, so that, under certain assumptions about $V$,
the operator $H_{V}$ is self-adjoint in $L^{2}\left( \mathbb{R}^{n}\right) $%
. Denote by $\func{Neg}\left( V,\mathbb{R}^{n}\right) $ the number of
non-positive eigenvalues of $H_{V}$ counted with multiplicity, assuming that
its spectrum in $(-\infty ,0]$ is discrete.

For the operator $H_{V}$ in $\mathbb{R}^{n}$ with $n\geq 3$ a celebrated
inequality of Cwikel-Lieb-Rozenblum says that%
\begin{equation}
\func{Neg}\left( V,\mathbb{R}^{n}\right) \leq C_{n}\int_{\mathbb{R}%
^{n}}V\left( x\right) ^{n/2}dx.  \label{CLR}
\end{equation}%
This estimate was proved independently by the above named authors in
1972-1977 in \cite{Cwikel}, \cite{Lieb}, and \cite{Rozen}, respectively%
\footnote{%
See also \cite{GrigYauHigh}, \cite{LevSolom}, \cite{LiYauE}, \cite{LiebE}
for further developments.}.

The estimate (\ref{CLR}) is not valid in $\mathbb{R}^{2}$ as one can see on
simple examples. On the contrary, in $\mathbb{R}^{2}$ a similar lower bound
holds:%
\begin{equation}
\func{Neg}\left( V,\mathbb{R}^{2}\right) \geq c\int_{\mathbb{R}^{2}}V\left(
x\right) dx  \label{Neglow}
\end{equation}%
that was proved in \cite{GrigNetYau}.

Our main result -- Theorem \ref{Tmainp} below, provides an upper bound for $%
\func{Neg}\left( V,\mathbb{R}^{2}\right) .$ To state it, let us introduce
some notation. For any $n\in \mathbb{Z}$ define the annuli $U_{n}$ and $%
W_{n} $ in $\mathbb{R}^{2}$ by%
\begin{equation}
U_{n}=\left\{ 
\begin{array}{ll}
\{e^{2^{n-1}}<\left\vert x\right\vert <e^{2^{n}}\}, & n\geq 1, \\ 
\{e^{-1}<\left\vert x\right\vert <e\}, & n=0, \\ 
\{e^{-2^{\left\vert n\right\vert }}<\left\vert x\right\vert
<e^{-2^{\left\vert n\right\vert -1}}\}, & n\leq -1,%
\end{array}%
\right.  \label{Un}
\end{equation}%
and 
\begin{equation}
W_{n}=\left\{ x\in \mathbb{R}^{2}:e^{n}<\left\vert x\right\vert
<e^{n+1}\right\} .  \label{Wn}
\end{equation}%
Given a potential (=a non-negative $L_{loc}^{1}$-function$\mathbb{)}$ $%
V\left( x\right) $ on $\mathbb{R}^{2}$ and $p>1$, define for any $n\in 
\mathbb{Z}$ the following quantities:%
\begin{equation}
A_{n}\left( V\right) =\int_{U_{n}}V\left( x\right) \left( 1+\left\vert \ln
\left\vert x\right\vert \right\vert \right) dx  \label{An}
\end{equation}%
and%
\begin{equation}
B_{n}\left( V\right) =\left( \int_{W_{n}}V^{p}\left( x\right) \left\vert
x\right\vert ^{2\left( p-1\right) }dx\right) ^{1/p}.  \label{Bn}
\end{equation}%
We will write for simplicity $A_{n}$ and $B_{n}$ for $A_{n}\left( V\right) $
and $B_{n}\left( V\right) $, respectively, if it is clear from the context
to which potential $V$ this refers.

\begin{theorem}
\label{Tmainp}\label{TWDini}For any non-negative function $V\in
L_{loc}^{1}\left( \mathbb{R}^{2}\right) $ and $p>1$, we have%
\begin{equation}
\func{Neg}\left( V,\mathbb{R}^{2}\right) \leq 1+C\sum_{\left\{ n\in \mathbb{Z%
}:A_{n}>c\right\} }\sqrt{A_{n}}+C\sum_{\left\{ n\in \mathbb{Z}%
:B_{n}>c\right\} }B_{n},  \label{NegIn}
\end{equation}%
where $C,c$ are some positive constants depending only on $p$.
\end{theorem}

The additive term $1$ in (\ref{NegIn}) reflects a special feature of $%
\mathbb{R}^{2}$: for any non-trivial potential $V$, the spectrum of $H_{V}$
has a negative part, no matter how small are the sums in (\ref{NegIn}). In $%
\mathbb{R}^{n}$ with $n\geq 3$, $\func{Neg}\left( V,\mathbb{R}^{n}\right) $
can be $0$ provided the integral in (\ref{CLR}) is small enough.

In fact, the quantity $\func{Neg}\left( V,\mathbb{R}^{2}\right) $ is
understood in a more general manner using the Morse index of an appropriate
energy form, rather than the operator $H_{V}$ directly (see Section \ref%
{SecGen}) so that $\func{Neg}\left( V,\mathbb{R}^{2}\right) $ always makes
sense.

\subsection{Discussion and historical remarks}

A simpler (and coarser) version of (\ref{NegIn}) is%
\begin{equation}
\func{Neg}\left( V,\mathbb{R}^{2}\right) \leq 1+C\int_{\mathbb{R}%
^{2}}V\left( x\right) \left( 1+\left\vert \ln \left\vert x\right\vert
\right\vert \right) dx+C\sum_{n\in \mathbb{Z}}B_{n}.  \label{Zn}
\end{equation}%
Indeed, if $A_{n}>c$ then $\sqrt{A_{n}}\leq c^{-1/2}A_{n}$ so that the first
sum in (\ref{NegIn}) can be replaced by $\sum_{n\in \mathbb{Z}}A_{n}$ thus
yielding (\ref{Zn}).

The estimate (\ref{Zn}) was first proved by Solomyak \cite{Solom}. In fact,
he proved a better estimate than (\ref{Zn}): 
\begin{equation}
\func{Neg}\left( V,\mathbb{R}^{2}\right) \leq 1+C\left\Vert A\right\Vert
_{1,\infty }+C\sum_{n\in \mathbb{Z}}B_{n},  \label{Solom}
\end{equation}%
where $A$ denoted the whole sequence $\left\{ A_{n}\right\} _{n\in \mathbb{Z}%
}$ and $\left\Vert A\right\Vert _{1,\infty }$ is the weak $l^{1}$-norm (the
Lorentz norm) defined by%
\begin{equation*}
\left\Vert A\right\Vert _{1,\infty }=\sup_{s>0}s\#\left\{ n:A_{n}>s\right\} .
\end{equation*}%
Clearly, $\left\Vert A\right\Vert _{1,\infty }\leq \left\Vert A\right\Vert
_{1}$ so that (\ref{Solom}) is better than (\ref{Zn}). The estimate (\ref%
{Solom}) was so far the best known\footnote{%
In fact, the estimate of \cite{Solom} is slightly sharper than (\ref{Solom})
because $B_{n}$ are defined in \cite{Solom} using not the $L^{p}$-norm but a
certain Orlicz norm. Further improvement of the term $B_{n}$ can be found in 
\cite{LapSol}. However, our main concern are the terms $A_{n}$ reflecting
the global geometry of $\mathbb{R}^{2}$.} upper bound for $\func{Neg}\left(
V,\mathbb{R}^{2}\right) $.

However, (\ref{Solom}) also follows from our estimate (\ref{NegIn}). Indeed,
it is easy to verify that%
\begin{equation*}
\left\Vert A\right\Vert _{1,\infty }\leq \sup_{s>0}s^{1/2}\sum_{\left\{
A_{n}>s\right\} }\sqrt{A_{n}}\leq 4\left\Vert A\right\Vert _{1,\infty }.
\end{equation*}%
In particular, we have%
\begin{equation*}
\sum_{\left\{ A_{n}>c\right\} }\sqrt{A_{n}}\leq 4c^{-1/2}\left\Vert
A\right\Vert _{1,\infty },
\end{equation*}%
so that (\ref{NegIn}) implies (\ref{Solom}). As we will see below, our
estimate (\ref{NegIn}) provides for certain potentials strictly better
results than (\ref{Solom}).

In the case when $V\left( x\right) $ is a radial function, that is, $V\left(
x\right) =V\left( \left\vert x\right\vert \right) $, the following estimate
was proved by Chadan, Khuri, Martin and Wu \cite{CKMW}, \cite{KMW}:%
\begin{equation}
\func{Neg}\left( V,\mathbb{R}^{2}\right) \leq 1+\int_{\mathbb{R}^{2}}V\left(
x\right) \left( 1+\left\vert \ln \left\vert x\right\vert \right\vert \right)
dx.  \label{KMW}
\end{equation}%
Although this estimate is sharper than (\ref{Zn}), we will see that our main
estimate (\ref{NegIn}) gives for certain radial potentials strictly better
results than (\ref{KMW}).

Laptev and Solomyak \cite{LapSolrad} improved (\ref{Zn}) for general
potentials by modifying the definition of $B_{n}$ so that all the terms $%
B_{n}$ vanish for radial potentials thus yielding (\ref{KMW}).

Another known estimate for $\func{Neg}\left( V,\mathbb{R}^{2}\right) $ is
due to Molchanov and Vainberg \cite{MolchVain}:%
\begin{equation}
\func{Neg}\left( V,\mathbb{R}^{2}\right) \leq 1+C\int_{\mathbb{R}%
^{2}}V\left( x\right) \ln \left\langle x\right\rangle dx+C\int_{\mathbb{R}%
^{2}}V\left( x\right) \ln \left( 2+V\left( x\right) \left\langle
x\right\rangle ^{2}\right) dx,  \label{MV}
\end{equation}%
where $\left\langle x\right\rangle =e+\left\vert x\right\vert $. However,
due to the logarithmic term in the second integral, this estimate never
leads to the semi-classical asymptotic%
\begin{equation}
\func{Neg}\left( \alpha V,\mathbb{R}^{2}\right) =O\left( \alpha \right) \ \
\ \text{as }\alpha \rightarrow \infty .  \label{a1}
\end{equation}%
that is expected to be true for \textquotedblleft nice\textquotedblright\
potentials. Note that the estimates (\ref{Zn}), (\ref{Solom}) and (\ref{MV})
are linear in $\alpha $ and, hence, imply (\ref{a1}) whenever the right hand
sides are finite.

Our main estimate (\ref{NegIn}) uses two types of quantities: $\sqrt{A_{n}}$
and $B_{n}.$ While $B_{n}\left( \alpha V\right) $ is linear in $a$, the term 
$\sqrt{A_{n}\left( \alpha V\right) }$ is \emph{sublinear }in $\alpha $,
which allows to obtain some interesting effects as $\alpha \rightarrow
\infty $ (see Section \ref{SecEx}).

Another novelty in (\ref{NegIn}) is the restriction of the both sums in (\ref%
{NegIn}) to the values $A_{n}>c$ and $B_{n}>c$, respectively. It follows
that if $A_{n}\rightarrow 0$ and $B_{n}\rightarrow 0$ then the both sums in (%
\ref{NegIn}) and, hence, $\func{Neg}\left( V,\mathbb{R}^{2}\right) $ are
finite, which does not follow from any of the previously known results. For
example, this is the case for a potential $V$ such that 
\begin{equation*}
V\left( x\right) =o\left( \frac{1}{\left\vert x\right\vert ^{2}\ln
^{2}\left\vert x\right\vert }\right) \ \ \text{as }x\rightarrow \infty .
\end{equation*}%
We discuss this and many other examples in Section \ref{SecEx}.

The nature of the terms $\sqrt{A_{n}}$ and $B_{n}$ in (\ref{NegIn}) can be
explained as follows. Different parts of the potential $V$ contribute
differently to $\func{Neg}\left( V,\mathbb{R}^{2}\right) $. The high values
of $V$ concentrated on relatively small areas contribute to $\func{Neg}%
\left( V,\mathbb{R}^{2}\right) $ via the terms $B_{n}$, while the low values
of $V$ scattered over large areas, contribute via the terms $\sqrt{A_{n}}.$
Since we integrate $V$ over long annuli, the long range effect of $V$
becomes similar to that of an one-dimensional potential. In $\mathbb{R}^{1}$
one expects 
\begin{equation*}
\func{Neg}\left( \alpha V,\mathbb{R}^{1}\right) =O\left( \sqrt{\alpha }%
\right) \ \ \text{as }\alpha \rightarrow \infty ,
\end{equation*}%
which explains the appearance of the square root in (\ref{NegIn}).

An exhaustive account of upper bounds in one-dimensional case can be found
in \cite{BLSone}, \cite{NaimSol}, \cite{RozSolSmall}. By the way, the
following estimate was proved by Naimark and Solomyak \cite{NaimSol}:%
\begin{equation}
\func{Neg}\left( V,\mathbb{R}_{+}^{1}\right) \leq 1+C\sum_{n=0}^{\infty }%
\sqrt{a_{n}},  \label{sol1}
\end{equation}%
where 
\begin{equation*}
a_{n}=\int_{I_{n}}V\left( x\right) \left( 1+\left\vert x\right\vert \right)
dx
\end{equation*}%
and $I_{n}=\left[ 2^{n-1},2^{n}\right] $ if $n>0~$and $I_{0}=[0,1].$
Clearly, the sum $\sum \sqrt{a_{n}}$ here resembles $\sum \sqrt{A_{n}}$ in (%
\ref{NegIn}), which is not a coincidence. In fact, our method allows to
improve (\ref{sol1}) by restricting the sum to $\left\{ n:a_{n}>c\right\} $.

Let us state two consequences of Theorem \ref{Tmainp}.

\begin{corollary}
\label{Coral}If%
\begin{equation}
\int_{\mathbb{R}^{2}}V\left( x\right) \left( 1+\left\vert \ln \left\vert
x\right\vert \right\vert \right) dx+\sum_{n\in \mathbb{Z}}B_{n}\left(
V\right) <\infty  \label{An+Bn}
\end{equation}%
then%
\begin{equation}
\func{Neg}\left( \alpha V,\mathbb{R}^{2}\right) \leq C\alpha \sum_{n\in 
\mathbb{Z}}B_{n}\left( V\right) +o\left( \alpha \right) \ \ \text{as\ }%
\alpha \rightarrow \infty .  \label{NVA}
\end{equation}
\end{corollary}

\begin{corollary}
\label{CorW}Assume that $\mathcal{W}\left( r\right) $ is a positive monotone
increasing function on $(0,+\infty )$ that satisfies the following Dini type
condition both at $0$ and at $\infty $:%
\begin{equation}
\int_{0}^{\infty }\frac{r\left\vert \ln r\right\vert ^{\frac{p}{p-1}}dr}{%
\mathcal{W}\left( r\right) ^{\frac{1}{p-1}}}<\infty .  \label{W4}
\end{equation}%
Then 
\begin{equation}
\func{Neg}\left( V,\mathbb{R}^{2}\right) \leq 1+C\left( \int_{\mathbb{R}%
^{2}}V^{p}\left( x\right) \mathcal{W}\left( \left\vert x\right\vert \right)
dx\right) ^{1/p},  \label{NegW}
\end{equation}%
where the constant $C$ depends on $p$ and $\mathcal{W}$.
\end{corollary}

Here is an example of a weight function $\mathcal{W}\left( r\right) $ that
satisfies (\ref{W4}):%
\begin{equation}
\mathcal{W}\left( r\right) =r^{2\left( p-1\right) }\langle \ln r\rangle
^{2p-1}\ln ^{p-1+\varepsilon }\langle \ln r\rangle ,  \label{Psidef}
\end{equation}%
where $\varepsilon >0$. In particular, for $p=2$, (\ref{NegW}) becomes%
\begin{equation}
\func{Neg}\left( V,\mathbb{R}^{2}\right) \leq 1+C\left( \int_{\mathbb{R}%
^{2}}V^{2}\left( x\right) \left\vert x\right\vert ^{2}\langle \ln \left\vert
x\right\vert \rangle ^{3}\ln ^{1+\varepsilon }\langle \ln \left\vert
x\right\vert \rangle dx\right) ^{1/2}.  \label{Neg2}
\end{equation}

Let us emphasize once again that none of the above mentioned estimates (\ref%
{Zn}), (\ref{Solom}), (\ref{KMW}), (\ref{MV}), (\ref{NegW}) matches the full
strength of our main estimate (\ref{NegIn}) even for radial potentials as
will be seen on examples below.

\subsection{Outline of the paper}

\label{here copy(3)}

Our method of the proof of Theorem \ref{Tmainp} is significantly different
from other existing methods of estimating $\func{Neg}\left( V,\mathbb{R}%
^{n}\right) $ and uses the advantages of $\mathbb{R}^{2}$ such as the
presence of a large class of conformal mappings preserving the Dirichlet
integral. Let us briefly describe the structure of paper that matches the
flowchart of the proof.

In Section \ref{SecEx} we give examples of application of Theorem \ref%
{Tmainp}.

In Section \ref{SecGen} we define for any open set $\Omega \subset \mathbb{R}%
^{2}$ the quantity $\func{Neg}\left( V,\Omega \right) $ as the Morse index
of the quadratic form 
\begin{equation*}
\mathcal{E}_{V,\Omega }\left( u\right) =\int_{\Omega }\left\vert \nabla
u\right\vert ^{2}dx-\int_{\Omega }Vu^{2}dx,
\end{equation*}%
and prove various properties of the former including subadditivity with
respect to partitioning and the behavior under conformal and bilipschitz
mappings. For bounded domains $\Omega $ with smooth boundary, $\func{Neg}%
\left( V,\Omega \right) $ coincides with the number of non-positive
eigenvalues of the Neumann problem for $-\Delta -V$ in $\Omega $.

The main result of Section \ref{SecBdd} is Lemma \ref{LemMain} that provides
the following estimate for a unit square $Q$:%
\begin{equation}
\func{Neg}\left( {V,Q}\right) \leq 1+C\left\Vert V\right\Vert _{L^{p}\left(
Q\right) }.  \label{L1}
\end{equation}%
The proof involves a careful partitioning of $Q$ into tiles $\Omega
_{1},...,\Omega _{N}$ with small enough $\left\Vert V\right\Vert
_{L^{p}\left( \Omega _{n}\right) }$ so that $\func{Neg}\left( V,\Omega
_{n}\right) =1$. The main difficulty is to control the number $N$ of the
tiles, which yields then (\ref{L1}). While the number of those $\Omega _{n}$
where $\left\Vert V\right\Vert _{L^{p}\left( \Omega _{n}\right) }$ is large
enough can be controlled via $\left\Vert V\right\Vert _{L^{p}\left( Q\right)
}$, the tiles $\Omega _{n}$ with small values of $\left\Vert V\right\Vert
_{L^{p}\left( \Omega _{n}\right) }$ are controlled inductively using special
features of the partitioning.

This argument is reminiscent of the Calderon-Zygmund decomposition (cf. \cite%
{BirSolTr}, \cite{Fefferman}, \cite{MelRoz}, \cite{RozenVUZ}), but is
simpler because we do not restrict the shape of the tiles to squares.

The estimate (\ref{L1}) leads in the end to the terms $B_{n}$ in (\ref{NegIn}%
) reflecting the local properties of the potential.

In Section \ref{SecR2one} we make the first step towards the global
properties of $V.$ Our starting point is the Green function $g\left(
x,y\right) $ of the operator $H_{0}=-\Delta +V_{0}$ where $V_{0}\in
C_{0}^{\infty }\left( \mathbb{R}^{2}\right) $ is a fixed potential for which 
$\func{Neg}\left( V_{0},\mathbb{R}^{2}\right) =1.$ We use the following
estimate of $g\left( x,y\right) $ that was proved in \cite{GrigW}:%
\begin{equation*}
g\left( x,y\right) \simeq \ln \left\langle x\right\rangle \wedge \ln
\left\langle y\right\rangle +\ln _{+}\frac{1}{\left\vert x-y\right\vert }.
\end{equation*}%
Considering the integral operator%
\begin{equation*}
G_{V}f\left( x\right) =\int_{\mathbb{R}^{2}}g\left( x,y\right) f\left(
y\right) V\left( y\right) dy
\end{equation*}%
acting in $L^{2}\left( Vdx\right) $, we show first that%
\begin{equation*}
\left\Vert G_{V}\right\Vert \leq \frac{1}{2}\Rightarrow \func{Neg}\left( V,%
\mathbb{R}^{2}\right) =1
\end{equation*}%
(Corollary \ref{CorR2}). Hence, to characterize the potentials $V$ with $%
\func{Neg}\left( V,\mathbb{R}^{2}\right) =1$ it suffices to estimate the
norm of $G_{V}$. Using the conformal mapping $z\mapsto \ln z$, we translate
the problem to a simpler integral operator $\Gamma _{V}$ acting in a strip%
\begin{equation*}
S=\left\{ \left( x_{1},x_{2}\right) \in \mathbb{R}^{2}:x_{1}\in \mathbb{R},\
0<x_{2}<\pi \right\} .
\end{equation*}

In Section \ref{SecNorm} we estimate the norm of a certain integral operator
in $S$ using a weighted Hardy inequality (Lemma \ref{Leman}).

In Section \ref{SecW} we obtain an estimate of $\left\Vert \Gamma
_{V}\right\Vert $ (Lemma \ref{LemGaV}) that leads to conditions for $\func{%
Neg}\left( V,S\right) =1$ (Proposition \ref{Pab}). A number of further
steps, involving a careful partitioning of the strip into rectangles, is
needed to obtain an upper bound for $\func{Neg}\left( V,S\right) \ $that is
stated in Theorem \ref{TStrip} and that is interesting on its own right.

In the final Section \ref{SecProof} we translate the estimate for $\func{Neg}%
\left( V,S\right) $ into that for $\func{Neg}\left( V,\mathbb{R}^{2}\right) $
thus finishing the proof of Theorem \ref{Tmainp}.

\begin{acknowledgement}
The first named author thanks Stanislav Molchanov for bringing this problem
to his attention and for fruitful discussions. The authors are indebted to
Ari Laptev, Grigori Rozenblum, and Michail Solomyak for useful remarks that
led to significant improvement of the results. They also thank Eugene
Shargorodsky for interesting comments.

This work was partially done during the visits of the second named author to
University of Bielefeld and of the first named author to Chinese University
of Hong Kong. The support of SFB 701 of the German Research Council and of a
visiting grant of CUHK is gratefully acknowledged.
\end{acknowledgement}

\section{Examples}

\label{SecEx}\setcounter{equation}{0}Let $V$ be a potential in $\mathbb{R}%
^{2}$, and let us use the abbreviation $\func{Neg}\left( V\right) \equiv 
\func{Neg}\left( V,\mathbb{R}^{2}\right) .$ We write $f\simeq g$ if the
ratio $\frac{f}{g}$ is bounded between two positive constants.

\begin{case}[1. ]
Assume that, for all $x\in \mathbb{R}^{2}$, 
\begin{equation*}
V\left( x\right) \leq \frac{\alpha }{\left\vert x\right\vert ^{2}}
\end{equation*}%
for a small enough positive constant $\alpha $. Then, for all $n\in \mathbb{Z%
}$,%
\begin{equation*}
B_{n}\leq \alpha \left( \int_{e^{n}}^{e^{n+1}}\frac{1}{r^{2p}}r^{2\left(
p-1\right) }2\pi rdx\right) ^{1/p}\simeq \alpha
\end{equation*}%
so that $B_{n}<c$ and the last sum in (\ref{NegIn}) is void, whence we
obtain 
\begin{eqnarray}
\func{Neg}\left( V\right) &\leq &1+C\sum_{\left\{ n:A_{n}>c\right\} }\sqrt{%
A_{n}}  \label{Nln1} \\
&\leq &1+C\int_{\mathbb{R}^{2}}V\left( x\right) \left( 1+\left\vert \ln
\left\vert x\right\vert \right\vert \right) dx.  \label{Nln}
\end{eqnarray}%
The estimate (\ref{Nln}) in this case follows also from (\ref{MV}).
\end{case}

\begin{case}[2. ]
Consider a potential%
\begin{equation*}
V\left( x\right) =\frac{1}{\left\vert x\right\vert ^{2}\left( 1+\ln
^{2}\left\vert x\right\vert \right) },
\end{equation*}%
As in the first example, $B_{n}\simeq 1$, while $A_{n}$ can be computed as
follows: for $n\geq 1$ 
\begin{equation}
A_{n}=\int_{e^{2^{n-1}}}^{e^{2^{n}}}\frac{1}{r^{2}\left( 1+\ln ^{2}r\right) }%
\left( 1+\ln r\right) 2\pi rdr\simeq 1,  \label{An=al}
\end{equation}%
and the same estimate holds for $n\leq 0$. Hence, if $\alpha >0$ is small
enough then $A_{n}\left( \alpha V\right) $ and $B_{n}\left( \alpha V\right) $
are smaller than $c$ for all $n$, and the both sums in (\ref{NegIn}) are
void. It follows that 
\begin{equation*}
\func{Neg}\left( \alpha V\right) =1.
\end{equation*}%
This result cannot be obtained by any of the previously known estimates.
Indeed, in the estimates (\ref{KMW}) and (\ref{MV}) the integral $\int_{%
\mathbb{R}^{2}}V\left( x\right) \left( 1+\left\vert \ln \left\vert
x\right\vert \right\vert \right) dx$ diverges, and in the estimate (\ref%
{Solom}) of Solomyak one has $\left\Vert A\right\Vert _{1,\infty }=\infty $.
As will be shown below, if $\alpha >1/4$ then $\func{Neg}\left( \alpha
V\right) =\infty $. Hence, $\func{Neg}\left( \alpha V\right) $ exhibits a
non-linear behavior with respect to the parameter $\alpha $, which cannot be
captured by linear estimates.
\end{case}

\begin{case}[3. ]
Assume that $V\left( x\right) $ is locally bounded and 
\begin{equation}
V\left( x\right) =o\left( \frac{1}{\left\vert x\right\vert ^{2}\ln
^{2}\left\vert x\right\vert }\right) \ \ \text{as }x\rightarrow \infty .
\label{V22}
\end{equation}%
Similarly to the previous example, we see that $A_{n}\left( V\right)
\rightarrow 0$ and $B_{n}\left( V\right) \rightarrow 0$ as $n\rightarrow
\infty $, which implies that the both sums in (\ref{NegIn}) are finite and,
hence, 
\begin{equation*}
\func{Neg}\left( V\right) <\infty .
\end{equation*}%
This result is also new.
\end{case}

\begin{case}[4. ]
Choose $q>0$ and set 
\begin{equation}
V\left( x\right) =\frac{1}{\left\vert x\right\vert ^{2}\ln ^{2}\left\vert
x\right\vert \left( \ln \ln \left\vert x\right\vert \right) ^{q}}\ \ \text{%
for }\left\vert x\right\vert >e^{2}  \label{V22q}
\end{equation}%
and $V\left( x\right) =0$ for $\left\vert x\right\vert \leq e^{2}.$ Then $%
A_{n}=0$ for $n\leq 1$, while for $n\geq 2$ we obtain%
\begin{equation*}
A_{n}\left( V\right) =\int_{e^{2^{n-1}}}^{e^{2^{n}}}\frac{\left( 1+\ln
r\right) 2\pi rdr}{r^{2}\ln ^{2}r\left( \ln \ln r\right) ^{q}}\simeq \frac{1%
}{n^{q}}.
\end{equation*}%
Similarly, we have for $n\geq 2$%
\begin{equation*}
B_{n}\left( V\right) =\left( \int_{e^{n}}^{e^{n+1}}\frac{r^{2\left(
p-1\right) }2\pi rdr}{\left[ r^{2}\ln ^{2}r\left( \ln \ln r\right) ^{q}%
\right] ^{p}}\right) ^{1/p}\simeq \frac{1}{n^{2}\ln ^{q}n}.
\end{equation*}%
Let $\alpha $ be a large real parameter. Then 
\begin{equation}
A_{n}\left( \alpha V\right) \simeq \frac{\alpha }{n^{q}},  \label{Anq}
\end{equation}%
and the condition $A_{n}\left( \alpha V\right) >c$ is satisfied for $n\leq
C\alpha ^{1/q},$ whence we obtain%
\begin{equation*}
\sum_{\left\{ A_{n}\left( \alpha V\right) >c\right\} }\sqrt{A_{n}\left(
\alpha V\right) }\leq C\sum_{n=1}^{\lceil C\alpha ^{1/q}\rceil }\sqrt{\frac{%
\alpha }{n^{q}}}\simeq C\sqrt{\alpha }\left( \alpha ^{1/q}\right)
^{1-q/2}=C\alpha ^{1/q}.
\end{equation*}%
It is clear that $\sum_{n}B_{n}\left( \alpha V\right) \simeq \alpha $.
Hence, we obtain from (\ref{NegIn})%
\begin{equation*}
\func{Neg}\left( \alpha V\right) \leq C\left( \alpha ^{1/q}+\alpha \right) .
\end{equation*}%
If $q\geq 1$ then the leading term here is $\alpha $. Combining this with (%
\ref{Neglow}), we obtain 
\begin{equation*}
\func{Neg}\left( \alpha V\right) \simeq \alpha \ \ \text{as\ }\alpha
\rightarrow \infty .
\end{equation*}%
If $q<1$ then the leading term is $\alpha ^{1/q},$ and we obtain 
\begin{equation*}
\func{Neg}\left( \alpha V\right) \leq C\alpha ^{1/q}.
\end{equation*}%
Birman and Laptev \cite{BirLap} proved that, in this case, indeed, 
\begin{equation*}
\func{Neg}\left( \alpha V\right) \sim \func{const}\alpha ^{1/q}\ \ \text{as\ 
}\alpha \rightarrow \infty .
\end{equation*}%
In the case $q<1$ we have $\left\Vert A\right\Vert _{1,\infty }=\infty $,
and neither of the estimates (\ref{Zn}), (\ref{KMW}), (\ref{Solom}), (\ref%
{MV}), (\ref{NegW}) yields even the finiteness of $\func{Neg}\left( \alpha
V\right) ,$ leaving alone the correct rate of growth in $\alpha $.
\end{case}

\begin{case}[5. ]
Let us study the behavior of $\func{Neg}\left( \alpha V\right) $ as $\alpha
\rightarrow \infty $ for a potential $V$ such that%
\begin{equation}
\int_{\mathbb{R}^{2}}V\left( x\right) \left( 1+\left\vert \ln \left\vert
x\right\vert \right\vert \right) dx+\sum_{n\in \mathbb{Z}}B_{n}\left(
V\right) <\infty .  \label{A+B}
\end{equation}%
By Corollary \ref{Coral} and (\ref{Neglow}), we obtain 
\begin{equation}
c\alpha \int_{\mathbb{R}^{2}}Vdx\leq \func{Neg}\left( \alpha V\right) \leq
C\alpha \sum_{n\in \mathbb{Z}}B_{n}\left( V\right) +o\left( \alpha \right)
,\ \ \alpha \rightarrow \infty ,  \label{aa}
\end{equation}%
in particular, $\func{Neg}\left( \alpha V\right) \simeq \alpha $. If $V$
satisfies in addition the following condition: 
\begin{equation}
\sup_{W_{n}}V\simeq \inf_{W_{n}}V,  \label{K}
\end{equation}%
for all $n\in \mathbb{Z}$, then 
\begin{equation*}
B_{n}\left( V\right) \simeq \int_{W_{n}}Vdx,
\end{equation*}%
and (\ref{aa}) implies that%
\begin{equation}
\func{Neg}\left( \alpha V\right) \simeq \alpha \int_{\mathbb{R}^{2}}V\left(
x\right) dx\ \ \text{as }\alpha \rightarrow \infty .  \label{Nai}
\end{equation}%
For example, (\ref{A+B}) and, hence, (\ref{Nai}) are satisfied for the
potential (\ref{V22q}) with $q>1$. The exact asymptotic for $\func{Neg}%
\left( \alpha V\right) $ as $\alpha \rightarrow \infty $ was obtained by
Birman and Laptev \cite{BirLap}.
\end{case}

\begin{case}[6.]
Set $R=e^{2^{m}}$ where $m$ is a large integer and consider the following
potential on $\mathbb{R}^{2}$ 
\begin{equation*}
V\left( x\right) =\left\{ 
\begin{array}{ll}
\frac{\alpha }{\left\vert x\right\vert ^{2}\ln ^{2}\left\vert x\right\vert },
& \text{if }e<\left\vert x\right\vert <R, \\ 
0, & \text{otherwise,}%
\end{array}%
\right.
\end{equation*}%
where $\alpha >\frac{1}{4}$. Computing $A_{n}\left( V\right) $ as in (\ref%
{An=al}) we obtain $A_{n}\left( V\right) \simeq \alpha $ for any $1\leq
n\leq m$, and $A_{n}=0$ otherwise, whence it follows that%
\begin{equation*}
\sum_{n\in \mathbb{Z}}\sqrt{A_{n}\left( V\right) }\simeq \sqrt{\alpha }%
m\simeq \sqrt{\alpha }\ln \ln R.
\end{equation*}%
Similarly, we have, for $1\leq n<2^{m},$%
\begin{equation*}
B_{n}\left( V\right) =\left( \int_{e^{n}}^{e^{n+1}}\left[ \frac{\alpha }{%
r^{2}\ln ^{2}r}\right] ^{p}r^{2\left( p-1\right) }2\pi rdr\right)
^{1/p}\simeq \frac{\alpha }{n^{2}},
\end{equation*}%
and $B_{n}\left( V\right) =0$ otherwise, whence%
\begin{equation*}
\sum_{n\in \mathbb{Z}}B_{n}\left( V\right) \simeq \sum_{n=1}^{2^{m}-1}\frac{%
\alpha }{n^{2}}\simeq \alpha .
\end{equation*}%
By (\ref{NegIn}) we obtain%
\begin{equation}
\func{Neg}\left( V\right) \leq C\sqrt{\alpha }\ln \ln R+C\alpha .
\label{Nara}
\end{equation}%
Let us remark that none of the previously known general estimates for $\func{%
Neg}\left( V,\mathbb{R}^{2}\right) $ yields (\ref{Nara}). For example, both (%
\ref{Solom}) and (\ref{KMW}) give in this case a weaker estimate%
\begin{equation*}
\func{Neg}\left( V\right) \leq C\alpha \ln \ln R.
\end{equation*}%
Obviously, (\ref{Nara}) requires a full strength of (\ref{NegIn}).

Let us estimate $\func{Neg}\left( V\right) $ from below to show the
sharpness of (\ref{Nara}) with respect to the parameters $\alpha ,R$.
Consider the function%
\begin{equation*}
f\left( x\right) =\sqrt{\ln \left\vert x\right\vert }\sin \left( \sqrt{%
\alpha -\frac{1}{4}}\ln \ln \left\vert x\right\vert \right)
\end{equation*}%
that satisfies in the region $\Omega =\left\{ e<\left\vert x\right\vert
<R\right\} $ the differential equation $\Delta f+V\left( x\right) f=0.$ For
any positive integer $k$, function $f$ does not change sign in the rings 
\begin{equation*}
\Omega _{k}:=\left\{ x\in \mathbb{R}^{2}:\pi k<\sqrt{\alpha -\frac{1}{4}}\ln
\ln \left\vert x\right\vert <\pi \left( k+1\right) \right\}
\end{equation*}%
and vanishes on $\partial \Omega _{k}$ as long as $\Omega _{k}\subset \Omega
.$ Since $\mathcal{E}_{V,\Omega _{k}}\left( f\right) =0$, using $f|_{\Omega
_{k}}$ as test functions for the energy functional, we obtain $\func{Neg}%
\left( V\right) \geq N$ where $N$ is the number of the rings $\Omega _{k}$
inside $\Omega .$ Assuming that $\alpha >>\frac{1}{4}$, we see that $N\simeq 
\sqrt{\alpha }\ln \ln R$, whence it follows that%
\begin{equation*}
\func{Neg}\left( V\right) \geq c\sqrt{\alpha }\ln \ln R.
\end{equation*}%
On the other hand, (\ref{Neglow}) yields $\func{Neg}\left( V\right) \geq
c\alpha $. Combining these two estimates, we obtain the lower bound%
\begin{equation*}
\func{Neg}\left( V\right) \geq c\left( \sqrt{\alpha }\ln \ln R+\alpha
\right) ,
\end{equation*}%
that matches the upper bound (\ref{Nara}).
\end{case}

\begin{case}[7.]
\label{Exp1}This example is of a different nature. Let us show that no
estimate of the type%
\begin{equation*}
\func{Neg}\left( V\right) \leq \func{const}+\int_{\mathbb{R}^{2}}V\left(
x\right) \mathcal{W}\left( x\right) dx
\end{equation*}%
can be true, provided a weight function $\mathcal{W}$ is bounded in a
neighborhood of at least one point. Indeed, assume without loss of
generality that $\mathcal{W}\left( x\right) \leq C$ for $\left\vert
x\right\vert <\varepsilon .$ We will construct a potential $V$ supported in $%
\left\{ \left\vert x\right\vert <\varepsilon \right\} $ such that $\int_{%
\mathbb{R}^{2}}Vdx<\infty $ while $\func{Neg}\left( V\right) =\infty .$

It will be easier to construct $V$ as a measure but then it can be routinely
approximated by a $L_{loc}^{1}$-function. For any $r>0$, let $S_{r}$ be the
circle $\left\{ \left\vert x\right\vert =r\right\} .$ We will use the
measure $\delta _{S_{r}}$ supported on $S_{r}$. Given two sequences $\left\{
a_{n}\right\} $ and $\left\{ b_{n}\right\} $ of reals such that $%
0<a_{n}<b_{n}$, consider the measures%
\begin{equation*}
V_{n}=\frac{1}{a_{n}\ln \frac{b_{n}}{a_{n}}}\delta _{S_{a_{n}}}
\end{equation*}%
and test functions%
\begin{equation}
\varphi _{n}\left( x\right) =\left\{ 
\begin{array}{ll}
1, & \left\vert x\right\vert <a_{n}, \\ 
\frac{\ln \frac{b_{n}}{\left\vert x\right\vert }}{\ln \frac{b_{n}}{a_{n}}},
& a_{n}\leq \left\vert x\right\vert \leq b_{n}, \\ 
0, & \left\vert x\right\vert >b_{n}.%
\end{array}%
\right.  \label{fiab}
\end{equation}%
An easy computation shows that%
\begin{equation}
\int_{\mathbb{R}^{2}}\left\vert \nabla \varphi _{n}\right\vert ^{2}dx=\frac{%
2\pi }{\ln \frac{b_{n}}{a_{n}}}  \label{Efi}
\end{equation}%
and%
\begin{equation*}
\int_{\mathbb{R}^{2}}\varphi _{n}^{2}V_{n}dx=\int_{\mathbb{R}^{2}}V_{n}dx=%
\frac{2\pi }{\ln \frac{b_{n}}{a_{n}}},
\end{equation*}%
whence it follows that $\mathcal{E}_{Vn}\left( \varphi _{n}\right) =0.$

Let us now specify $a_{n}=4^{-n^{3}}$ and $b_{n}=2^{-n^{3}}.$ Consider also
the following sequence of points in $\mathbb{R}^{2}$: $y_{n}=\left(
4^{-n},0\right) $. Then all disks $D_{b_{n}}\left( y_{n}\right) $ with large
enough $n$ are disjoint and 
\begin{equation}
\sum_{n=1}^{\infty }\frac{2\pi }{\ln \frac{b_{n}}{a_{n}}}<\infty .  \label{<}
\end{equation}%
Consider the generalized function 
\begin{equation}
V=\sum_{n=N}^{\infty }V\left( \cdot -y_{n}\right) .  \label{-yn}
\end{equation}%
The functions $\psi _{n}=$ $\varphi _{n}\left( \cdot -y_{n}\right) $ have
disjoint supports and satisfy $\mathcal{E}_{V}\left( \psi _{n}\right) =0$
for all $n\geq N$, whence it follows that $\func{Neg}\left( V\right) =\infty
.$ On the other hand, by (\ref{<}) we have%
\begin{equation*}
\int_{\mathbb{R}^{2}}Vdx<\infty .
\end{equation*}%
By taking $N$ large enough, one can make $\int_{\mathbb{R}^{2}}Vdx$
arbitrarily small and $\limfunc{supp}V$ to be located in an arbitrarily
small neighborhood of the origin, while still having $\func{Neg}\left(
V\right) =\infty .$
\end{case}

\section{Generalities of counting functions}

\setcounter{equation}{0}\label{SecGen}

\subsection{Index of quadratic forms}

Let $\Omega \subset \mathbb{R}^{2}$ be an arbitrary open set. By a \emph{%
potential} in $\Omega \subset \mathbb{R}^{n}$ we mean always a non-negative
function from $L_{loc}^{1}\left( \Omega \right) $. Given a potential $V$ in $%
\Omega ,$ define the energy form%
\begin{equation}
\mathcal{E}_{V,\Omega }\left( f\right) =\int_{\Omega }\left\vert \nabla
f\right\vert ^{2}dx-\int_{\Omega }Vf^{2}dx  \label{Edef}
\end{equation}%
in the domain%
\begin{equation}
\mathcal{F}_{V,\Omega }=\left\{ f\in L_{loc}^{2}\left( \Omega \right)
:\int_{\Omega }\left\vert \nabla f\right\vert ^{2}dx<\infty ,\ \
\int_{\Omega }Vf^{2}dx<\infty \right\} .  \label{Fdef}
\end{equation}%
Clearly, $\mathcal{F}_{V,\Omega }$ is a linear space. Note that a more
conventional choice for the ambient space for $\mathcal{F}_{V,\Omega }$
would be $L^{2}\left( \Omega \right) $, but for us a larger space $%
L_{loc}^{2}\left( \Omega \right) $ will be more convenient.

Set%
\begin{equation}
\func{Neg}\left( V,\Omega \right) :=\sup \left\{ \dim \mathcal{V}:\mathcal{V}%
\prec \mathcal{F}_{V,\Omega }:\mathcal{E}_{V,\Omega }\left( f\right) \leq 0%
\text{ for all }f\in \mathcal{V}\right\} ,  \label{negdef}
\end{equation}%
where $\mathcal{V}\prec \mathcal{F}_{V,\Omega }$ means that $\mathcal{V}$ is
a linear subspace of $\mathcal{F}_{V,\Omega }$, and the supremum of $\dim 
\mathcal{V}$ is taken over all subspaces $\mathcal{V}$ such that $\mathcal{E}%
_{V,\Omega }\leq 0$ on $\mathcal{V}$. In other words, $\func{Neg}\left(
V,\Omega \right) $ is the Morse index of the quadratic form $\mathcal{E}%
_{V,\Omega }$ in $\mathcal{F}_{V,\Omega }$. Observe that one can restrict in
(\ref{negdef}) the class of subspaces $\mathcal{V}$ to those of finite
dimension without changing the value of the right hand side.

Note that $\func{Neg}\left( V,\Omega \right) \geq 1$ for any potential $V$.
Indeed, if $V\in L^{1}\left( \Omega \right) $ then $1\in \mathcal{F}_{\Omega
}$ and $\mathcal{E}_{V,\Omega }\left( 1\right) \leq 0,$ which implies that $%
\func{Neg}\left( V,\Omega \right) \geq 1.$ If $V\notin L^{1}\left( \Omega
\right) $, then consider for any positive integer $n$ a function $%
f_{n}\left( x\right) =\frac{1}{n}\left( n-\left\vert x\right\vert \right)
_{+}.$ This function belongs to $\mathcal{F}_{V,\Omega }$ as it has a
compact support, $0\leq f_{n}\leq 1$, and $\int_{\Omega }\left\vert \nabla
f_{n}\right\vert ^{2}dx\leq \pi .$ Since $f_{n}\uparrow 1$ as $n\rightarrow
\infty $, it follows that%
\begin{equation*}
\int_{\Omega }Vf_{n}^{2}dx\rightarrow \int_{\Omega }Vdx=\infty .
\end{equation*}%
Hence, for large enough $n$, we obtain $\mathcal{E}_{V,\Omega }\left(
f_{n}\right) <0$ and, hence, $\func{Neg}\left( V,\Omega \right) \geq 1$.

If $\Omega =\mathbb{R}^{n}$ then we use the abbreviations 
\begin{equation*}
\mathcal{E}_{V}\equiv \mathcal{E}_{V,\mathbb{R}^{n}},\ \ \mathcal{F}%
_{V}\equiv \mathcal{F}_{V,\mathbb{R}^{n}},\ \ \func{Neg}\left( V\right)
\equiv \func{Neg}\left( V,\mathbb{R}^{n}\right) .
\end{equation*}%
The operator%
\begin{equation*}
H_{V}=-\Delta -V
\end{equation*}%
is defined as a self-adjoint operator in $L^{2}\left( \mathbb{R}^{n}\right) $
using the following standard procedure. Firstly, observe that the classical
Dirichlet integral 
\begin{equation*}
\mathcal{E}\left( u\right) =\int_{\mathbb{R}^{n}}\left\vert \nabla
u\right\vert ^{2}dx
\end{equation*}%
with the domain $W^{1,2}\left( \mathbb{R}^{2}\right) $ is a closed form in $%
L^{2}\left( \mathbb{R}^{2}\right) $, and the quadratic form%
\begin{equation*}
u\mapsto \int_{\mathbb{R}^{n}}Vu^{2}dx
\end{equation*}%
associated with the multiplication operator $u\mapsto Vu$, is closed with
the domain $L^{2}\left( dx\right) \cap L^{2}\left( Vdx\right) .$ Clearly,
the form $\mathcal{E}_{V}$ is well-defined in the domain%
\begin{equation*}
\mathcal{D}_{V}=W^{1,2}\cap L^{2}\left( Vdx\right)
\end{equation*}%
that is a subspace of $\mathcal{F}_{V}$. Under certain assumptions about $V$%
, the form $\left( \mathcal{E}_{V},\mathcal{D}_{V}\right) $ is closed in $%
L^{2}$ (and, in fact, $\mathcal{D}_{V}=W^{1,2}$). Consequently, its
generator, denoted by $H_{V}$, is a self-adjoint, semi-bounded below
operator in $L^{2}$, whose domain is a subspace of $\mathcal{D}_{V}$.

For any self-adjoint operator $A$, denote by $\func{Neg}\left( A\right) $
the rank of the operator $\mathbf{1}_{(-\infty ,0]}\left( A\right) $, that
is,%
\begin{equation*}
\func{Neg}\left( A\right) =\dim \func{Im}\mathbf{1}_{(-\infty ,0]}\left(
A\right) .
\end{equation*}
If the spectrum of $A$ below $0$ is discrete then $\func{Neg}\left( A\right) 
$ coincides with the number of non-positive eigenvalues of $A$ counted with
multiplicities.

\begin{lemma}
\label{LemED}If the form $\left( \mathcal{E}_{V},\mathcal{D}_{V}\right) $ is
closed and, hence, $H_{V}$ is well-defined, then%
\begin{equation}
\func{Neg}\left( H_{V}\right) \leq \func{Neg}\left( V\right) .
\label{supdim}
\end{equation}
\end{lemma}

\begin{proof}
It is well-known that 
\begin{equation*}
\func{Neg}\left( H_{V}\right) =\sup \left\{ \dim \mathcal{V}:\mathcal{V}%
\prec \mathcal{D}_{V}\text{ and }\mathcal{E}_{V}\left( f\right) \leq
0\;\forall f\in \mathcal{V}\right\}
\end{equation*}%
\textbf{\ } (cf. \textbf{\ }\cite[Lemma 2.7]{GrigNetYau}). Since $\mathcal{D}%
_{V}\subset \mathcal{F}_{V}$, (\ref{supdim}) holds by monotonicity argument.
\end{proof}

Theorem \ref{Tmainp} states the upper bound for $\func{Neg}\left( V\right) $%
, which implies then by Lemma \ref{LemED} the same bound for $\func{Neg}%
\left( H_{V}\right) $ whenever $H_{V}$ is well-defined. If this method were
applied in $\mathbb{R}^{n}$ with $n\geq 3$ then the resulting estimate would
not have been satisfactory, because $\func{Neg}\left( H_{V}\right) $ can be $%
0$ (as follows, for example, from (\ref{CLR})), whereas $\func{Neg}\left(
V\right) \geq 1$ for all potentials $V$ as it was remarked above. However,
our aim is $\mathbb{R}^{2}$, where $\func{Neg}\left( H_{V}\right) \geq 1$
for any non-zero potential $V$, so that we do not loose $1$ in the estimate.

In the rest of this section we prove some general properties of $\func{Neg}%
\left( V,\Omega \right) $ that will be used in the next sections. For a
bounded domain $\Omega $ with smooth boundary, the form $\mathcal{E}%
_{V,\Omega }$ can be associated with the operator $\Delta +V$ in $\Omega $
with the Neumann boundary condition on $\partial \Omega $. In this case $%
\func{Neg}\left( V,\Omega \right) $ is equal to the number of non-positive
eigenvalues of the Neumann problem in $\Omega $ for the operator $\Delta +V$%
. This understanding helps the intuition, but technically we never need to
use the operator $\Delta +V$. Nor the closability of the form $\mathcal{E}%
_{V,\Omega }$ is needed, except for Lemma \ref{LemED}.

\begin{lemma}
\label{LemL}Let $\Omega ,\widetilde{\Omega }$ be open subsets of $\mathbb{R}%
^{2}$ and $V$ and $\widetilde{V}$ be potentials in $\Omega $ and $\widetilde{%
\Omega },$ respectively. Let $\mathcal{L}:\mathcal{F}_{V,\Omega }\rightarrow 
\mathcal{F}_{\widetilde{V},\widetilde{\Omega }}$ be a linear injective
mapping.

\begin{itemize}
\item[$\left( a\right) $] If $\mathcal{E}_{V,\Omega }\left( u\right) \leq 0$
implies $\mathcal{E}_{\widetilde{V},\widetilde{\Omega }}\left( \widetilde{u}%
\right) \leq 0$ for $\widetilde{u}=\mathcal{L}\left( u\right) $ then 
\begin{equation}
\func{Neg}\left( V,\Omega \right) \leq \func{Neg}(\widetilde{V},\widetilde{%
\Omega }).  \label{N<N}
\end{equation}

\item[$\left( b\right) $] Assume that there are positive constants $%
c_{1},c_{2}$, such that, for any $u\in \mathcal{F}_{V,\Omega }$, the
function $\widetilde{u}=\mathcal{L}\left( u\right) $ satisfies%
\begin{equation}
\int_{\widetilde{\Omega }}\left\vert \nabla \widetilde{u}\right\vert
^{2}dx\leq c_{1}\int_{\Omega }\left\vert \nabla u\right\vert ^{2}dx
\label{uti1}
\end{equation}%
and%
\begin{equation}
\int_{\widetilde{\Omega }}\widetilde{V}\widetilde{u}^{2}dx\geq
c_{2}\int_{\Omega }Vu^{2}dx.  \label{uti2}
\end{equation}%
Then 
\begin{equation}
\func{Neg}\left( V,\Omega \right) \leq \func{Neg}(\frac{c_{1}}{c_{2}}%
\widetilde{V},\widetilde{\Omega }).  \label{Neg<=}
\end{equation}
\end{itemize}
\end{lemma}

\begin{proof}
$\left( a\right) $ Let $\mathcal{V}$ be a finitely dimensional linear
subspace of $\mathcal{F}_{\Omega }$ where $\mathcal{E}_{V,\Omega }\leq 0$.
Then $\widetilde{\mathcal{V}}:=\mathcal{L}\left( \mathcal{V}\right) $ is a
linear subspace of $\mathcal{F}_{\widetilde{V},\widetilde{\Omega }}$ of the
same dimension. For any $\widetilde{u}\in \widetilde{\mathcal{V}}$ we have $%
\mathcal{E}_{\widetilde{V},\widetilde{\Omega }}\left( \widetilde{u}\right)
\leq 0,$ which implies $\dim \widetilde{\mathcal{V}}\leq \func{Neg}(%
\widetilde{V},\widetilde{\Omega })$. Since $\dim \mathcal{V}=\dim \widetilde{%
\mathcal{V}}$, we have also $\dim \mathcal{V}\leq \func{Neg}(\widetilde{V},%
\widetilde{\Omega })$, whence (\ref{N<N}) follows.

$\left( b\right) $ If $\mathcal{E}_{V,\Omega }\left( u\right) \leq 0$ then%
\begin{eqnarray*}
\mathcal{E}_{\frac{c_{1}}{c_{2}}\widetilde{V},\widetilde{\Omega }}\left( 
\widetilde{u}\right) &=&\int_{\widetilde{\Omega }}\left\vert \nabla 
\widetilde{u}\right\vert ^{2}dx-\frac{c_{1}}{c_{2}}\int_{\widetilde{\Omega }}%
\widetilde{V}\widetilde{u}^{2}dx \\
&\leq &c_{1}\int_{\Omega }\left\vert \nabla u\right\vert ^{2}dx-c_{1}\int
Vu^{2}dx=c_{1}\mathcal{E}_{V,\Omega }\left( u\right) \leq 0.
\end{eqnarray*}%
Applying part $\left( a\right) $ with $\frac{c_{1}}{c_{2}}\widetilde{V}$
instead of $\widetilde{V}$, we obtain (\ref{Neg<=}).
\end{proof}

\begin{lemma}
\label{Lem-K}Let $\Omega $ be any open subset of $\mathbb{R}^{2},$ and $K$
be a closed subset of $\mathbb{R}^{n}$ of measure $0.$ Set $\Omega ^{\prime
}=\Omega \setminus K$. Then we have%
\begin{equation}
\func{Neg}\left( V,\Omega \right) \leq \func{Neg}\left( V,\Omega ^{\prime
}\right) .  \label{=0}
\end{equation}
\end{lemma}

\begin{proof}
Every function $u\in \mathcal{F}_{V,\Omega }$ can be considered as an
element of $\mathcal{F}_{V,\Omega ^{\prime }}$ simply by restricting $u$ to $%
\Omega ^{\prime }$. Since the difference $\Omega \setminus \Omega ^{\prime }$
has measure $0$, we have $\mathcal{E}_{V,\Omega }\left( u\right) =\mathcal{E}%
_{V,\Omega ^{\prime }}\left( u\right) .$ Then Lemma \ref{LemL}$\left(
a\right) $ implies (\ref{=0}).
\end{proof}

\begin{definition}
\RM We say that a (finite or infinite) sequence $\left\{ \Omega _{k}\right\} 
$ of non-empty open sets $\Omega _{k}\subset \mathbb{R}^{2}$ is a \emph{%
partition} of an open set $\Omega \subset \mathbb{R}^{n}$ if all the sets $%
\Omega _{k}$ are disjoint, $\Omega _{k}\subset \Omega $, and $\overline{%
\Omega }\setminus \tbigcup_{k}\Omega _{k}$ has measure $0$ (cf. Fig. \ref%
{pic22}).
\end{definition}

\FRAME{ftbphFU}{3.2534in}{1.8619in}{0pt}{\Qcb{A partition of $\Omega $}}{%
\Qlb{pic22}}{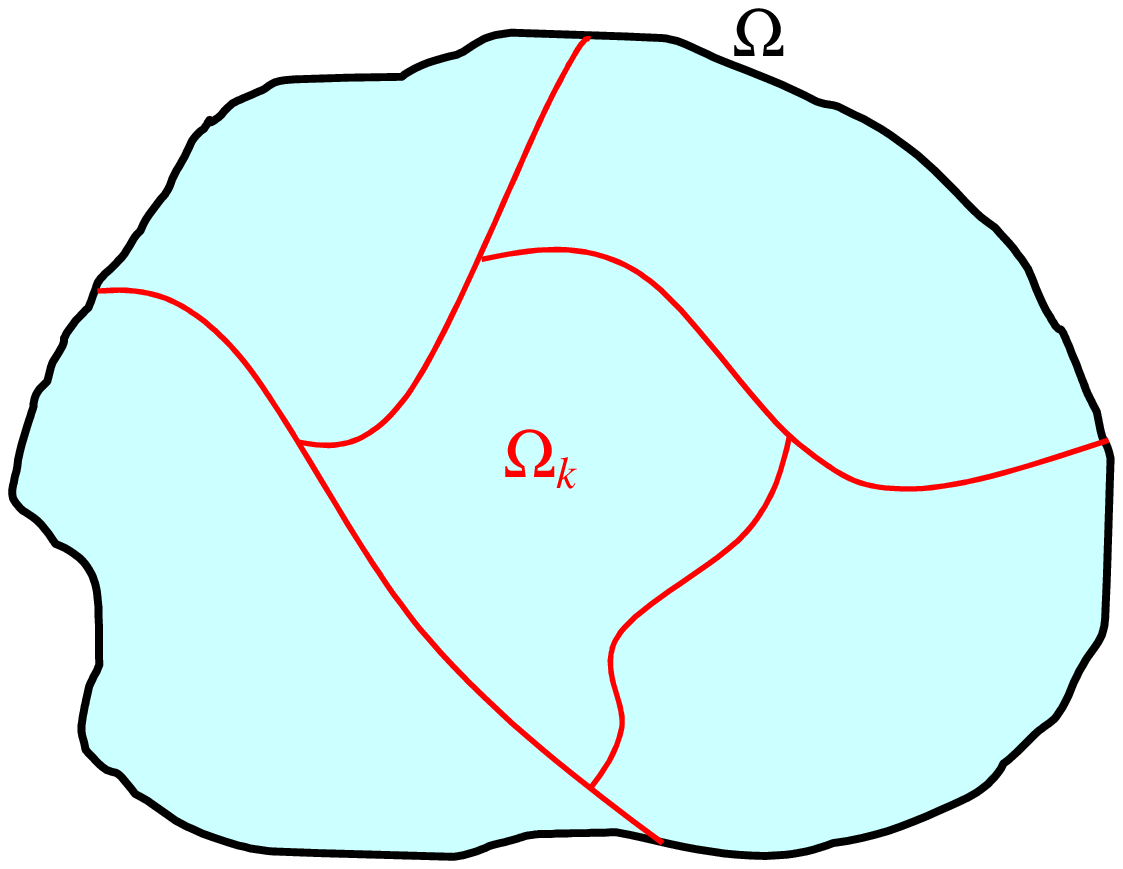}{\special{language "Scientific Word";type
"GRAPHIC";maintain-aspect-ratio TRUE;display "USEDEF";valid_file "F";width
3.2534in;height 1.8619in;depth 0pt;original-width 6.3027in;original-height
3.5864in;cropleft "0";croptop "1";cropright "1";cropbottom "0";filename
'pic22.eps';file-properties "XNPEU";}}

\begin{lemma}
\label{LemSub}If $\left\{ \Omega _{k}\right\} $ is a partition of $\Omega ,$
then%
\begin{equation}
\func{Neg}\left( V,\Omega \right) \leq \sum_{k}\func{Neg}\left( V,\Omega
_{k}\right) .  \label{n+n}
\end{equation}
\end{lemma}

\begin{proof}
Set $\Omega ^{\prime }=\tbigcup_{k}\Omega _{k}$ and $K=\overline{\Omega }%
\setminus \Omega ^{\prime }$. Since $K$ is closed, $K$ has measure $0$, and $%
\Omega ^{\prime }=\Omega \setminus K$, we obtain by Lemma \ref{Lem-K} that%
\begin{equation*}
\func{Neg}\left( V,\Omega \right) \leq \func{Neg}(V,\Omega ^{\prime }).
\end{equation*}%
Next, we claim that%
\begin{equation}
\func{Neg}(V,\Omega ^{\prime })\leq \sum_{k}\func{Neg}\left( V,\Omega
_{k}\right) .  \label{n+n'}
\end{equation}%
If the sum in (\ref{n+n'}) is infinite then there is nothing to prove.
Assume that this sum is finite. Since $\func{Neg}\left( V,\Omega _{k}\right)
\geq 1$, the number of elements in the partition $\left\{ \Omega
_{k}\right\} $ must be finite, which will be assumed in the sequel. Denote
for simplicity $\mathcal{F}^{\prime }=\mathcal{F}_{V,\Omega ^{\prime }}$, $%
\mathcal{E}^{\prime }=\mathcal{E}_{V,\Omega ^{\prime }}$, $\mathcal{F}_{k}=%
\mathcal{F}_{V,\Omega _{k}}$ and $\mathcal{E}_{k}=\mathcal{E}_{V,\Omega
_{k}}.$

For any $f\in \mathcal{F}^{\prime }$ and index $k$, set $f_{k}=f|_{\Omega
_{k}}$ so that $f_{k}\in \mathcal{F}_{k}.$ Clearly, we have ~$%
f=\sum_{k}f_{k} $ and%
\begin{equation}
\mathcal{E}^{\prime }\left( f\right) =\sum_{k}\mathcal{E}_{k}\left(
f_{k}\right) .  \label{sumfk}
\end{equation}%
Hence, $\mathcal{F}^{\prime }$ can be identified as a subspace of the direct
sum $\mathcal{F}=\dbigoplus \mathcal{F}_{k}$, and $\mathcal{E}^{\prime }$
can be extended from $\mathcal{F}^{\prime }$ to $\mathcal{F}$ by (\ref{sumfk}%
), as the direct sum of all $\mathcal{E}_{k}$.

Let $\mathcal{V}$ be a finite dimensional subspace of $\mathcal{F}^{\prime }$
(or even of $\mathcal{F}$) where $\mathcal{E}^{\prime }\leq 0.$ Restricting
as above the functions from $\mathcal{V}$ to $\Omega _{k}$, we obtain a
finite dimensional subspace $\mathcal{V}_{k}$ of $\mathcal{F}_{k}$. Set $%
\mathcal{U}=\dbigoplus \mathcal{V}_{k},$ so that $\mathcal{V\prec U}\prec 
\mathcal{F}.$ The quadratic form $\mathcal{E}_{k}$ is diagonalizable on the
finite dimensional space $\mathcal{V}_{k}$, and the number $N_{k}$ of the
non-positive terms in the signature of $\mathcal{E}_{k}$ on $\mathcal{V}_{k}$
is clearly bounded by $\func{Neg}\left( V,\Omega _{k}\right) $. Hence,
denoting by $N$ the number of the non-positive terms in the signature of $%
\mathcal{E}^{\prime }$ on $\mathcal{U}$, we obtain%
\begin{equation*}
N=\sum_{k}N_{k}\leq \sum_{k}\func{Neg}\left( V,\Omega _{k}\right) .
\end{equation*}%
If $\dim \mathcal{V}>N$ then $\mathcal{V}$ intersects the subspace of $%
\mathcal{U}$ where $\mathcal{E}^{\prime }$ is positive definite, which
contradicts the assumption that $\mathcal{E}^{\prime }\leq 0$ on $\mathcal{V}
$. Therefore, $\dim \mathcal{V}\leq N$, whence (\ref{n+n'}) follows.
\end{proof}

\begin{lemma}
\label{LemV1V2}If $V_{1},V_{2}$ are two potentials in $\Omega $ then%
\begin{equation}
\func{Neg}\left( V_{1}+V_{2},\Omega \right) \leq \func{Neg}\left(
2V_{1},\Omega \right) +\func{Neg}\left( 2V_{2},\Omega \right) .
\label{NegV1+V2}
\end{equation}
\end{lemma}

\begin{proof}
Let us write for simplicity $\mathcal{E}_{V,\Omega }\equiv \mathcal{E}_{V}$
and $\mathcal{F}_{V,\Omega }\equiv \mathcal{F}_{V}$. Set $V=V_{1}+V_{2}$ and
observe that by (\ref{Fdef})%
\begin{equation*}
\mathcal{F}_{V}=\mathcal{F}_{V_{1}}\cap \mathcal{F}_{V_{2}}
\end{equation*}%
and by (\ref{Edef}) 
\begin{equation}
2\mathcal{E}_{V}=\mathcal{E}_{2V_{1}}+\mathcal{E}_{2V_{2}}\ \ \text{on }%
\mathcal{F}_{V}.  \label{EV1V2}
\end{equation}%
Assume that (\ref{NegV1+V2}) is not true. Then there exists a
finite-dimensional subspace $\mathcal{V}$ of $\mathcal{F}_{V}$ where $%
\mathcal{E}_{V}\leq 0$ and such that%
\begin{equation}
\dim \mathcal{V}>\func{Neg}\left( 2V_{1}\right) +\func{Neg}\left(
2V_{2}\right) .  \label{dimV>}
\end{equation}%
Set $N=\dim \mathcal{V}$ and denote by $N_{i}$, $i=1,2$, the maximal
dimension of a subspace of $\mathcal{V}$ where $\mathcal{E}_{2V_{i}}\leq 0.$
Then there exists a subspace $\mathcal{P}_{i}$ of $\mathcal{V}$ of dimension 
$N-N_{i}$ where $\mathcal{E}_{2V_{i}}\geq 0.$ The intersection $\mathcal{P}%
_{1}\cap \mathcal{P}_{2}$ has dimension at least 
\begin{equation*}
\left( N-N_{1}+N-N_{2}\right) -N=N-\left( N_{1}+N_{2}\right) >0,
\end{equation*}%
where the positivity holds by (\ref{dimV>}). By (\ref{EV1V2}) the form $%
\mathcal{E}_{V}$ is non-negative on $\mathcal{P}_{1}\cap \mathcal{P}_{2}$,
which contradicts the assumption that $\mathcal{E}_{V}\leq 0$ on $\mathcal{V}
$.
\end{proof}

\subsection{Transformation of potentials and weights}

Given a $2\times 2$ matrix $A=\left( a_{ij}\right) $, denote by $\left\Vert
A\right\Vert $ the norm of $A$ as an linear operator in $\mathbb{R}^{2}$
with the Euclidean norm. Denote also%
\begin{equation*}
\left\Vert A\right\Vert _{2}:=\sqrt{%
a_{11}^{2}+a_{12}^{2}+a_{21}^{2}+a_{22}^{2}}.
\end{equation*}%
It is easy to see that%
\begin{equation*}
\frac{1}{\sqrt{2}}\left\Vert A\right\Vert _{2}\leq \left\Vert A\right\Vert
\leq \left\Vert A\right\Vert _{2}
\end{equation*}%
Assuming further that $A$ is non-singular, define the quantities%
\begin{equation*}
M\left( A\right) :=\frac{\left\Vert A\right\Vert ^{2}}{\det A}\text{\ \ \
and\ \ \ }M_{2}\left( A\right) :=\frac{\left\Vert A\right\Vert _{2}^{2}}{%
\det A}
\end{equation*}%
For example, if $A$ is a conformal matrix, that is, $\left( 
\begin{array}{cc}
\alpha & \beta \\ 
-\beta & \alpha%
\end{array}%
\right) $ or $\left( 
\begin{array}{cc}
\alpha & \beta \\ 
\beta & -\alpha%
\end{array}%
\right) ,$ then 
\begin{equation*}
\det A=\alpha ^{2}+\beta ^{2}=\left\Vert A\right\Vert ^{2},
\end{equation*}%
whence $M\left( A\right) =1.$

For a general non-singular matrix $A$, the following identity holds:%
\begin{equation}
M_{2}\left( A\right) =M_{2}\left( A^{-1}\right) .  \label{M2=M2}
\end{equation}%
Indeed, denoting $a=\det A$, we obtain%
\begin{equation*}
A^{-1}=\frac{1}{a}\left( 
\begin{array}{cc}
a_{22} & -a_{12} \\ 
a_{21} & a_{11}%
\end{array}%
\right) ,
\end{equation*}%
whence $\left\Vert A^{-1}\right\Vert _{2}^{2}=\frac{1}{a^{2}}\left\Vert
A\right\Vert _{2}^{2}$, which implies (\ref{M2=M2}). Consequently, we obtain
that, for any non-singular matrix $A$,%
\begin{equation}
\frac{1}{2}M\left( A\right) \leq M\left( A^{-1}\right) \leq 2M\left(
A\right) .  \label{MAA-1}
\end{equation}

Let $\Omega $ and $\widetilde{\Omega }$ be two open subsets of $\mathbb{R}%
^{2}$ and $\Phi :\widetilde{\Omega }\rightarrow \Omega $ be a $C^{1}$%
-diffeomorphism. Denote by $\Phi ^{\prime }$ its Jacobi matrix and by $%
J_{\Phi }$ - its Jacobian, that it $J_{\Phi }=\det \Phi ^{\prime }.$ Set%
\begin{equation*}
M_{\Phi }:=\sup_{x\in \widetilde{\Omega }}M\left( \Phi ^{\prime }\left(
x\right) \right) =\sup_{x\in \widetilde{\Omega }}\frac{\left\Vert \Phi
^{\prime }\left( x\right) \right\Vert ^{2}}{\left\vert J_{\Phi }\left(
x\right) \right\vert }.
\end{equation*}

We will use two types of mappings $\Phi :$ bilipschitz and conformal. If $%
\Phi $ is conformal then we have $M_{\Phi }=1$. Moreover, if $\Phi $ is
holomorphic then 
\begin{equation}
J_{\Phi }\left( z\right) =\left\vert \Phi ^{\prime }\left( z\right)
\right\vert ^{2},  \label{J}
\end{equation}%
where now $\Phi ^{\prime }=\frac{d\Phi }{dz}$ is a complex derivative in $%
z\in \mathbb{C}$.

If $\Phi $ is bilipschitz and with bilipschitz constant $L$ then an easy
calculation shows that $\left\Vert \Phi ^{\prime }\left( x\right)
\right\Vert ^{2}\leq 4L^{2}$ and that both $\left\vert J_{\Phi }\right\vert $
and $\left\vert J_{\Phi ^{-1}}\right\vert $ are bounded by $2L^{2}$ whence $%
M_{\Phi }\leq 8L^{4}$.

By (\ref{MAA-1}), we always have%
\begin{equation}
\frac{1}{2}M_{\Phi }\leq M_{\Phi ^{-1}}\leq 2M_{\Phi }  \label{MFi2}
\end{equation}

The next lemma establishes the behavior of $\func{Neg}\left( V,\Omega
\right) $ and certain integrals over $\Omega $ under transformations of $%
\Omega $. By a weight function on $\Omega $ we mean any non-negative
function from $L_{loc}^{1}\left( \Omega \right) .$

\begin{lemma}
\label{LemVW}\label{Lemn<n}Let $\Omega ,\widetilde{\Omega }$ be two open
subsets of $\mathbb{R}^{2}$ and 
\begin{equation*}
\Psi :\Omega \rightarrow \widetilde{\Omega }
\end{equation*}%
be a $C^{1}$ diffeomorphism with a finite $M_{\Psi }$. Set $\Phi =\Psi
^{-1}. $

\begin{itemize}
\item[$\left( a\right) $] For any potential $V$ on $\Omega $, define a $\Psi 
$-push-forward potential $\widetilde{V}$ on $\widetilde{\Omega }$ by 
\begin{equation}
\widetilde{V}\left( y\right) =M_{\Phi }\left\vert J_{\Phi }\left( y\right)
\right\vert V\left( \Phi \left( y\right) \right) .  \label{Vti}
\end{equation}%
Then 
\begin{equation}
\func{Neg}\left( V,\Omega \right) \leq \func{Neg}({\widetilde{V},\widetilde{%
\Omega })}.  \label{n<n}
\end{equation}

\item[$\left( b\right) $] For any $p\geq 1$ and any weight function $W$ on $%
\Omega $, define a $\Psi $-push-forward weight function $\widetilde{W}$ on $%
\widetilde{\Omega }$ by 
\begin{equation}
\widetilde{W}\left( y\right) =M_{\Phi }^{-p}\left\vert J_{\Phi }\left(
y\right) \right\vert ^{1-p}W\left( \Phi \left( y\right) \right) .  \label{WW}
\end{equation}%
Then we the following identity holds%
\begin{equation}
\int_{\Omega }V\left( x\right) ^{p}W\left( x\right) dx=\int_{\widetilde{%
\Omega }}\widetilde{V}\left( y\right) ^{p}\widetilde{W}\left( y\right) dy
\label{VW}
\end{equation}
\end{itemize}
\end{lemma}

As one sees from (\ref{Vti}) and (\ref{WW}), the rules of change of a
potential and a weight function under a mapping $\Psi $ are different.

\begin{proof}
$\left( a\right) $ Let $\mathcal{V}$ be a subspace of $\mathcal{F}_{V,\Omega
}$ as in (\ref{negdef}). Define $\widetilde{\mathcal{V}}$ as the pullback of 
$\mathcal{V}$ under the mapping $\Phi $, that is, any function $\widetilde{f}%
\in \widetilde{\mathcal{V}}$ has the form 
\begin{equation*}
\widetilde{f}\left( y\right) =f\left( \Phi \left( y\right) \right)
\end{equation*}%
for some $f\in \mathcal{V}$. Let us show that $\widetilde{f}\in \mathcal{F}_{%
\widetilde{V},\widetilde{\Omega }}.$ That $\widetilde{f}\in L_{loc}^{2}(%
\widetilde{\Omega })$ is obvious. Using the change $y=\Psi \left( x\right) $
(or $x=\Phi \left( y\right) $), we obtain 
\begin{eqnarray}
\int_{\widetilde{\Omega }}\left\vert \widetilde{f}\left( y\right)
\right\vert ^{2}\widetilde{V}\left( y\right) dy &=&\int_{\Omega }\left\vert 
\widetilde{f}\left( y\right) \right\vert ^{2}\widetilde{V}\left( y\right)
\left\vert J_{\Psi }\left( x\right) \right\vert dx  \notag \\
&=&\int_{\Omega }\left\vert f\left( x\right) \right\vert ^{2}M_{\Phi
}V\left( x\right) \left\vert J_{\Phi }\left( y\right) \right\vert \left\vert
J_{\Phi }\left( y\right) \right\vert ^{-1}dx  \notag \\
&=&M_{\Phi }\int_{\Omega }\left\vert f\left( x\right) \right\vert
^{2}V\left( x\right) dx  \label{3}
\end{eqnarray}%
and%
\begin{eqnarray}
\int_{\widetilde{\Omega }}\left\vert \nabla \widetilde{f}\left( y\right)
\right\vert ^{2}dy &=&\int_{\widetilde{\Omega }}\left\vert \left( \nabla
f\right) \left( \Phi \left( y\right) \right) \cdot \Phi ^{\prime }\left(
y\right) \right\vert ^{2}dy  \notag \\
&\leq &\int_{\widetilde{\Omega }}\left\Vert \Phi ^{\prime }\left( y\right)
\right\Vert ^{2}\left\vert \nabla f\right\vert ^{2}\left( \Phi \left(
y\right) \right) dy  \notag \\
&\leq &M_{\Phi }\int_{\widetilde{\Omega }}\left\vert J_{\Phi }\left(
y\right) \right\vert \left\vert \nabla f\right\vert ^{2}\left( \Phi \left(
y\right) \right) dy  \notag \\
&=&M_{\Phi }\int_{\Omega }\left\vert \nabla f\right\vert ^{2}\left( x\right)
dx.  \label{4}
\end{eqnarray}%
It follows from (\ref{3}) and (\ref{4}) that $\widetilde{f}\in \mathcal{F}_{%
\widetilde{V},\widetilde{\Omega }}$ and $\mathcal{E}_{\widetilde{V},%
\widetilde{\Omega }}(\widetilde{f})\leq M_{\Phi }\mathcal{E}_{V,\Omega
}\left( f\right) $. Applying Lemma \ref{LemL} to the mapping $f\mapsto 
\widetilde{f}$, we obtain (\ref{n<n}).

$\left( b\right) $ Using the same change in integral, we obtain%
\begin{eqnarray*}
\int_{\widetilde{\Omega }}\widetilde{V}\left( y\right) ^{p}\widetilde{W}%
\left( y\right) dy &=&\int_{\Omega }\widetilde{V}\left( \Psi \left( x\right)
\right) ^{p}\widetilde{W}\left( \Psi \left( x\right) \right) \left\vert
J_{\Psi }\left( x\right) \right\vert dx \\
&=&\int_{\Omega }\left( M_{\Phi }V\left( x\right) \left\vert J_{\Phi }\left(
y\right) \right\vert \right) ^{p}M_{\Phi }^{-p}\left\vert J_{\Phi }\left(
y\right) \right\vert ^{1-p}W\left( x\right) \left\vert J_{\Phi }\left(
y\right) \right\vert ^{-1}dx \\
&=&\int_{\Omega }V\left( x\right) ^{p}W\left( x\right) dx.
\end{eqnarray*}
\end{proof}

\begin{remark}
\RM\label{Remcon}If $\Psi $ is conformal then it follows from Lemma \ref%
{LemVW} that 
\begin{equation*}
\func{Neg}\left( V,\Omega \right) =\func{Neg}(\widetilde{V},\widetilde{%
\Omega }),
\end{equation*}%
where 
\begin{equation*}
\widetilde{V}\left( y\right) =\left\vert J_{\Phi }\left( y\right)
\right\vert V\left( \Phi \left( y\right) \right) .
\end{equation*}%
Furthermore, if $\Psi $ is holomorphic then the formulas (\ref{Vti}) and (%
\ref{WW}) can be simplified as follows: 
\begin{equation*}
\widetilde{V}\left( z\right) =\left\vert \Phi ^{\prime }\left( z\right)
\right\vert ^{2}V\left( \Phi \left( z\right) \right) 
\end{equation*}%
and%
\begin{equation*}
\widetilde{W}\left( z\right) =\frac{W\left( \Phi \left( z\right) \right) }{%
\left\vert \Phi ^{\prime }\left( z\right) \right\vert ^{2\left( p-1\right) }}%
,
\end{equation*}%
where $\Phi ^{\prime }$ is a $\mathbb{C}$-derivative.
\end{remark}

\subsection{Bounded test functions}

\label{SecNegb}Consider the following modification of the space $\mathcal{F}%
_{V,\Omega }$:%
\begin{equation}
\mathcal{F}_{V,\Omega }^{b}=\left\{ f\in L^{\infty }\left( \Omega \right)
:\int_{\Omega }\left\vert \nabla f\right\vert ^{2}dx<\infty ,\ \
\int_{\Omega }Vf^{2}dx<\infty \right\}   \label{Fbdef}
\end{equation}%
and of the counting function:%
\begin{equation}
\func{Neg}^{b}\left( V,\Omega \right) :=\sup \left\{ \dim \mathcal{V}:%
\mathcal{V}\prec \mathcal{F}_{V,\Omega }^{b}:\mathcal{E}_{V,\Omega }\left(
f\right) \leq 0\text{ for all }f\in \mathcal{V}\right\} .  \label{negbdef}
\end{equation}%
In short, we restrict consideration to the class of bounded test functions.
By monotonicity we have%
\begin{equation*}
\func{Neg}^{b}\left( V,\Omega \right) \leq \func{Neg}\left( V,\Omega \right)
.
\end{equation*}

The following claim will be used in Section \ref{SecOneStrip}.

\begin{lemma}
\label{Lem2V}Let $\Omega $ be a connected domain in $\mathbb{R}^{2}$ such
that $\func{Neg}^{b}\left( 2V,\Omega \right) =1$.  Then $\func{Neg}\left(
V,\Omega \right) =1.$
\end{lemma}

\begin{proof}
Assume that $\func{Neg}\left( V,\Omega \right) >1.$ Then there exists a
two-dimensional subspace $\mathcal{V}$ of $\mathcal{F}_{V,\Omega }$ such
that $\mathcal{E}_{V,\Omega }\leq 0$ on $\mathcal{V}$. Consider the
following two functions on $\mathcal{V}$:%
\begin{equation}
X\left( f\right) =\int_{\Omega }\left\vert \nabla f_{+}\right\vert
^{2}dx-2\int_{\Omega }Vf_{+}^{2}dx  \label{Xf}
\end{equation}%
and%
\begin{equation}
Y\left( f\right) =\int_{\Omega }\left\vert \nabla f_{-}\right\vert
^{2}dx-2\int_{\Omega }Vf_{-}^{2}dx,  \label{Yf}
\end{equation}%
where $f_{\pm }=\frac{1}{2}\left( \left\vert f\right\vert \pm f\right) $ are
the positive and negative parts of $f$.  Clearly, we have%
\begin{equation*}
X\left( f\right) +Y\left( f\right) =\int_{\Omega }\left\vert \nabla
f\right\vert ^{2}dx-2\int_{\Omega }Vf^{2}dx=\mathcal{E}_{2V,\Omega }\left(
f\right) \leq 0.
\end{equation*}%
Let us show that in fact a strict inequality holds for all $f\in \mathcal{V}%
\setminus \left\{ 0\right\} $:%
\begin{equation}
X\left( f\right) +Y\left( f\right) <0.  \label{X+Y<0}
\end{equation}%
Indeed, if this is not true, that is,%
\begin{equation}
\int_{\Omega }\left\vert \nabla f\right\vert ^{2}dx\geq 2\int_{\Omega
}Vf^{2}dx,  \label{22}
\end{equation}%
then combining with%
\begin{equation*}
2\int_{\Omega }\left\vert \nabla f\right\vert ^{2}dx\leq 2\int_{\Omega
}Vf^{2}dx,
\end{equation*}%
we obtain $\int_{\Omega }\left\vert \nabla f\right\vert ^{2}dx=0$ and,
hence, $f=\func{const}$ in $\Omega .$ Then (\ref{22}) implies $V=0$ in $%
\Omega $, which is not possible by the assumption $\func{Neg}\left( V,\Omega
\right) >1$. This proves (\ref{X+Y<0}).

A second observation that we need is the identities%
\begin{equation}
X\left( -f\right) =Y\left( f\right) \ \ \text{and\ \ \ }Y\left( -f\right)
=X\left( f\right) ,  \label{X-Y}
\end{equation}%
that follow immediately from the definitions (\ref{Xf}), (\ref{Yf}).

Now consider a mapping $F:\mathcal{V}\rightarrow \mathbb{R}^{2}$ given by%
\begin{equation*}
F\left( f\right) =\left( X\left( f\right) ,Y\left( f\right) \right) .
\end{equation*}%
Let $T$ be the unit circle in $\mathcal{V}$ (with respect some arbitrary
norm in $\mathcal{V}$). Then the image $F\left( T\right) $ is a compact
connected subset of $\mathbb{R}^{2}$ that by (\ref{X+Y<0}) lies in the
half-plane $\left\{ x+y<0\right\} $, and by  (\ref{X-Y}) is symmetric in the
diagonal $x=y.$ It follows that there is a point in $F\left( T\right) $ that
lies on the diagonal $x=y$, that is, there is a function $f\in \mathcal{V}%
\setminus \left\{ 0\right\} $ such that%
\begin{equation*}
X\left( f\right) =Y\left( f\right) <0.
\end{equation*}%
This can be rewritten in the form%
\begin{equation*}
\mathcal{E}_{2V,\Omega }\left( f_{+}\right) =\mathcal{E}_{2V,\Omega }\left(
f_{-}\right) <0.
\end{equation*}%
Since%
\begin{equation*}
\mathcal{E}_{2V,\Omega }\left( f\wedge n\right) \rightarrow \mathcal{E}%
_{2V,\Omega }\left( f\right) \ \ \ \text{as }n\rightarrow +\infty ,\ 
\end{equation*}%
it follows that there is large enough $n$ such that%
\begin{equation*}
\mathcal{E}_{2V,\Omega }\left( f_{+}\wedge n\right) <0,\ \ \mathcal{E}%
_{2V,\Omega }\left( f_{-}\wedge n\right) <0.
\end{equation*}%
The functions $f_{+}\wedge n$ and $f_{-}\wedge n$ are bounded and have
\textquotedblleft almost\textquotedblright\ disjoint supports. It follows
that $\mathcal{E}_{2V,\Omega }\left( f\right) \leq 0$ holds for all linear
combinations $f$ of these two functions. Hence, we obtain a two dimensional
subspace of $\mathcal{F}_{2V,\Omega }^{b}$ where $\mathcal{E}_{2V,\Omega
}\leq 0$, which implies $\func{Neg}^{b}\left( 2V,\Omega \right) \geq 2$.
This contradiction finishes the proof. 
\end{proof}

\section{$L^{p}$-estimate in bounded domains}

\label{SecBdd}\setcounter{equation}{0}In this section we obtain upper bound
for $\func{Neg}\left( V,\Omega \right) $ for certain bounded domains $\Omega
\subset \mathbb{R}^{2}$.

\subsection{Extension of functions from $\mathcal{F}_{V,\Omega }$}

\label{SecL2loc}Here we consider auxiliary techniques for extending
functions from $\mathcal{F}_{V,\Omega }$ to larger domains. Denote by $%
D_{r}\left( x\right) $ an open disk in $\mathbb{R}^{2}$ of radius $r$
centered at $x$.

\begin{lemma}
\label{LemL2loc}Let $\Omega $ be a domain in $\mathbb{R}^{2}$ with piecewise
smooth boundary. Then $\mathcal{F}_{V,\Omega }\subset L_{loc}^{2}\left( 
\overline{\Omega }\right) $, where $\overline{\Omega }$ is the closure of $%
\Omega .$ If in addition $\Omega $ is bounded then $\mathcal{F}_{V,\Omega
}\subset L^{2}\left( \Omega \right) .$
\end{lemma}

\begin{proof}
Fix a point $x\in \partial \Omega $ and consider the domain $U=\Omega \cap
D_{r}\left( x\right) $ where $r>0$ is sufficiently small. It suffices to
verify that 
\begin{equation}
f\in L_{loc}^{2}\left( \Omega \right) ,\ \ \nabla f\in L^{2}\left( \Omega
\right) \ \Rightarrow \ f\in L^{2}\left( U\right) .  \label{fff}
\end{equation}%
Choose a little disk $K$ inside $U.$ For any function $f\in
W_{loc}^{1,2}\left( U\right) $ we have the following Poincar\'{e} type
inequality:\label{rem: reference to this Poincare type inequality}%
\begin{equation}
\int_{U}f^{2}dx\leq C\int_{U}\left\vert \nabla f\right\vert
^{2}dx+C\int_{K}f^{2}dx  \label{ukf}
\end{equation}%
where $C=C\left( K,U\right) .$ Since the right hand side of (\ref{ukf}) is
finite by hypotheses, it follows that $f\in L^{2}\left( U\right) $, which
was to be proved.
\end{proof}

Lemma \ref{LemL2loc} can be used to extend functions from $\mathcal{F}%
_{V,\Omega }$ to $\mathcal{F}_{V,\Omega ^{\prime }}$ where $\Omega ^{\prime
} $ is a larger domain. Any potential $V$ in a domain $\Omega $ can be
extended to a larger domain $\Omega ^{\prime }\ $by setting $V=0$ outside $%
\Omega $. We will refer to such an extension as a trivial one.

Let us give two examples, which will be frequently used in the next
sections. In all cases we assume that $V$ is trivially extended from $\Omega 
$ to $\Omega ^{\prime }.$

\begin{example}
\RM\label{ExL}Let $\Omega $ be a rectangle and let $L$ be one of its sides.
Merging $\Omega $ with its image under the axial symmetry around $L$, we
obtain a larger rectangle $\Omega ^{\prime }$. Any function $f$ on $\Omega $
can be extended to $\Omega ^{\prime }$ using push-forward under the axial
symmetry. We claim that if $f\in \mathcal{F}_{V,\Omega }$ then the extended
function $f$ belongs to $\mathcal{F}_{V,\Omega ^{\prime }}.$ By Lemma \ref%
{LemL2loc} we have $f\in L^{2}\left( \Omega \right) $ and, hence, $f\in
W^{1,2}\left( \Omega \right) $. It is well-known that if a $W^{1,2}$
function extends by axial symmetry then the resulting function is again from 
$W^{1,2}$, which implies that $f\in \mathcal{F}_{V,\Omega ^{\prime }}.$
\end{example}

\begin{example}
\RM\label{ExC}Let $\Omega $ be a sector of a disk $D_{r}\left( x_{0}\right) $
and let $C$ be a circular part of $\partial U$. Let us merge $\Omega $ with
its image under the inversion in $C$ and denote the resulting wedge by $%
\Omega ^{\prime }$. Extend any function $f$ from $\Omega $ to $\Omega
^{\prime }$ using push-forward under the inversion. Let us show that if $%
f\in \mathcal{F}_{V,\Omega }$ then the extended function $f$ belongs to $%
\mathcal{F}_{V,\Omega ^{\prime }}.$ Set $U=\Omega \setminus \overline{%
D_{\varepsilon }\left( x_{0}\right) }$ with some $\varepsilon >0$ so that $U$
is away from the center of inversion. Let $U^{\prime }$ be obtained by
merging $U$ with its image under inversion. By Lemma \ref{LemL2loc}, any
function $f\in \mathcal{F}_{V,\Omega }$ belongs to $L^{2}\left( U\right) $
and, hence, to $W^{1,2}\left( U\right) .$ Since $U^{\prime }$ is bounded,
the extended function $f$ belongs also to $W^{1,2}\left( U^{\prime }\right) $%
, which implies that $f\in W_{loc}^{1,2}\left( \Omega ^{\prime }\right) .$
By the conformal invariance of the Dirichlet integral we have%
\begin{equation*}
\int_{\Omega ^{\prime }\setminus \Omega }\left\vert \nabla f\right\vert
^{2}dx=\int_{\Omega }\left\vert \nabla f\right\vert ^{2}dx,
\end{equation*}%
which implies that $\int_{\Omega ^{\prime }}\left\vert \nabla f\right\vert
^{2}dx<\infty $ and, hence, $f\in \mathcal{F}_{V,\Omega ^{\prime }}.$

\label{Sechalf}Let $H_{+}=\left\{ \left( x_{1},x_{2}\right) \in \mathbb{R}%
^{2}:x_{2}>0\right\} $ be an upper half-plane.
\end{example}

\begin{lemma}
\label{LemH+}For any potential $V$ in $H_{+}$, we have%
\begin{equation}
\func{Neg}\left( V,H_{+}\right) \leq \func{Neg}\left( 2V,\mathbb{R}%
^{2}\right) ,  \label{NegH+}
\end{equation}%
assuming that $V$ is trivially extended from $H_{+}$ to $\mathbb{R}^{2}$.
\end{lemma}

\begin{proof}
Any function $f\in \mathcal{F}_{V,H_{+}}$ can be extended to a function $f$
on $\mathbb{R}^{2}$ by the axial symmetry around the axis $x_{1}$. Since by
Lemma \ref{LemL2loc} $f\in L^{2}\left( U\right) $ for any bounded open
subset $U$ of $H_{+}$, in particular, for any rectangle $U$ attached to $%
\partial H_{+}$, we obtain as in Example \ref{ExL} that $f\in
W_{loc}^{1,2}\left( \mathbb{R}^{2}\right) $. Since also 
\begin{equation*}
\int_{\mathbb{R}^{2}}\left\vert \nabla f\right\vert
^{2}dx=2\int_{H_{+}}\left\vert \nabla f\right\vert ^{2}dx
\end{equation*}%
and 
\begin{equation*}
\int_{\mathbb{R}^{2}}Vf^{2}dx=\int_{H_{+}}Vf^{2}dx,
\end{equation*}%
we see that $f\in \mathcal{F}_{V,\mathbb{R}^{2}}$, and the estimate (\ref%
{NegH+}) follows by Lemma \ref{LemL}.
\end{proof}

Let $D_{r}=D_{r}\left( 0\right) $ be an open disk of radius $r$ centered at
the origin.

\begin{lemma}
\label{Lem2D}For any potential $V$ in a disk $D_{r}$, 
\begin{equation}
\func{Neg}\left( V,\mathbb{R}^{2}\right) \leq \func{Neg}\left(
2V,D_{r}\right) ,  \label{NegDr}
\end{equation}%
assuming that $V$ is trivially extended from $D$ to $\mathbb{R}^{2}$.
\end{lemma}

\begin{proof}
Any function $f\in \mathcal{F}_{V,D_{r}}$ can be extended to a function $%
f\in \mathcal{F}_{V,\mathbb{R}^{2}}$ using inversion in the circle $\left\{
\left\vert x\right\vert =r\right\} $ as in Example \ref{ExC}. Then we have%
\begin{equation*}
\int_{\mathbb{R}^{2}}\left\vert \nabla f\right\vert
^{2}dx=2\int_{D_{r}}\left\vert \nabla f\right\vert ^{2}dx
\end{equation*}%
and 
\begin{equation*}
\int_{\mathbb{R}^{2}}Vf^{2}dx=\int_{D_{r}}Vf^{2}dx,
\end{equation*}%
which implies (\ref{NegDr}) by Lemma \ref{LemL}.
\end{proof}

A more complicated result analogous to Lemmas \ref{LemH+} and \ref{Lem2D}
will be considered in Section \ref{SecRec}.

\subsection{One negative eigenvalue in a disc}

Let $D=\left\{ \left\vert x\right\vert <1\right\} $ be the open unit disk in 
$\mathbb{R}^{2}$.

\begin{lemma}
\label{LemNeg=1}For any $p>1$ there is $\varepsilon >0$ such that, for any
potential $V$ in $D$, 
\begin{equation}
\left\Vert V\right\Vert _{L^{p}\left( D\right) }\leq \varepsilon \
\Rightarrow \ \func{Neg}\left( V,D\right) =1.  \label{eD}
\end{equation}
\end{lemma}

\begin{proof}
Extend $V$ to entire $\mathbb{R}^{2}$ by setting $V\left( x\right) =0$ for
all $\left\vert x\right\vert \geq 1$. Given a function $u\in \mathcal{F}%
_{V,D}$, extend $u$ to the entire $\mathbb{R}^{2}$ using the inversion $\Phi
\left( x\right) =\frac{x}{\left\vert x\right\vert ^{2}}$: for any $%
\left\vert x\right\vert >1$, set $u\left( x\right) =u\left( \Phi \left(
x\right) \right) $. As in Example \ref{ExC}, we have $u\in \mathcal{F}_{V,%
\mathbb{R}^{2}}$. By the conformal invariance of the Dirichlet integral, we
have%
\begin{equation}
\int_{\mathbb{R}^{2}}\left\vert \nabla u\right\vert
^{2}dx=2\int_{D}\left\vert \nabla u\right\vert ^{2}dx.  \label{2D}
\end{equation}%
Choose a cutoff function $\varphi $ such that $\varphi |_{D_{2}}\equiv 1,$ $%
\varphi |_{\mathbb{R}^{2}\setminus D_{3}}=0$ and $\varphi =\varphi \left(
\left\vert x\right\vert \right) $ is linear in $\left\vert x\right\vert $ in 
$D_{3}\setminus D_{2}$, and define a function $u^{\ast }$ by%
\begin{equation*}
u^{\ast }=u\varphi .
\end{equation*}%
Then $u^{\ast }\in W^{1,2}\left( \mathbb{R}^{2}\right) $ and $u^{\ast }$
vanishes outside $D_{3}$. Next, we prove some estimates for the function $%
u^{\ast }.$

\begin{claime}[Claim 1.]
We have%
\begin{equation}
\int_{D_{3}}\left\vert \nabla u^{\ast }\right\vert ^{2}dx\leq
4\int_{D}\left\vert \nabla u\right\vert ^{2}dx+162\int_{D}u^{2}dx.
\label{u*}
\end{equation}
\end{claime}

Indeed, since $\nabla u^{\ast }=\varphi \nabla u+u\nabla \varphi $, we have%
\begin{eqnarray*}
\int_{D_{3}}\left\vert \nabla u^{\ast }\right\vert ^{2}dx &\leq
&2\int_{D_{3}}\varphi ^{2}\left\vert \nabla u\right\vert
^{2}dx+2\int_{D_{3}}u^{2}\left\vert \nabla \varphi \right\vert ^{2}dx \\
&\leq &2\int_{\mathbb{R}^{2}}\left\vert \nabla u\right\vert
^{2}dx+2\int_{D_{3}\setminus D_{2}}u^{2}dx,
\end{eqnarray*}%
where we have used that $\left\vert \nabla \varphi \right\vert =1$ in $%
D_{3}\setminus D_{2}$ and $\nabla \varphi =0$ otherwise. Next, use the
change $y=\Phi \left( x\right) $ to map $D_{3}\setminus D_{2}$ to $%
D_{1/2}\setminus D_{1/3}$. Since $\left\vert J_{\Phi ^{-1}}\left( y\right)
\right\vert =$ $\frac{1}{\left\vert y\right\vert ^{4}}$, we obtain 
\begin{equation*}
\int_{D_{3}\setminus D_{2}}u^{2}\left( x\right) dx=\int_{D_{1/2}\setminus
D_{1/3}}u^{2}\left( y\right) \frac{1}{\left\vert y\right\vert ^{4}}dy\leq
3^{4}\int_{D}u^{2}dy.
\end{equation*}%
Combining the above estimates and using also (\ref{2D}), we obtain (\ref{u*}%
).

\begin{claime}[Claim 2.]
If $u\bot 1$ in $L^{2}\left( D\right) $ and $\mathcal{E}_{V,D}\left(
u\right) \leq 0$ then 
\begin{equation}
\int_{D_{4}}\left\vert \nabla u^{\ast }\right\vert ^{2}dx\leq
C\int_{D}Vu^{2}dx,  \label{uV}
\end{equation}%
with some absolute constant $C$.
\end{claime}

Indeed, the assumption $u\bot 1$ implies by the Poincar\'{e} inequality%
\begin{equation*}
\int_{D}u^{2}dx\leq c\int_{D}\left\vert \nabla u\right\vert ^{2}dx,
\end{equation*}%
which together with (\ref{u*}) yields%
\begin{equation*}
\int_{D_{4}}\left\vert \nabla u^{\ast }\right\vert ^{2}dx\leq \left(
4+162c\right) \int_{D}\left\vert \nabla u\right\vert ^{2}dx.
\end{equation*}%
Combining this with the hypothesis $\mathcal{E}_{V}\left( u\right) \leq 0,$
that is, 
\begin{equation}
\int_{D}\left\vert \nabla u\right\vert ^{2}dx\leq \int_{D}Vu^{2}dx,
\label{Vu1}
\end{equation}%
we obtain (\ref{uV}).

Now we prove the implication (\ref{eD}). Applying the H\"{o}lder inequality
to the right hand side of (\ref{uV}), we obtain%
\begin{eqnarray}
\int_{D}Vu^{2}dx &\leq &\left( \int_{D}V^{p}dx\right) ^{1/p}\left(
\int_{D}\left\vert u\right\vert ^{\frac{2p}{p-1}}dx\right) ^{1-1/p}  \notag
\\
&\leq &\left( \int_{D}V^{p}dx\right) ^{1/p}\left( \int_{D_{4}}\left\vert
u^{\ast }\right\vert ^{\frac{2p}{p-1}}dx\right) ^{1-1/p}.  \label{Vu2}
\end{eqnarray}%
Next, let us use Sobolev inequality for Lipschitz functions $f$ supported in 
$\overline{D_{3}}$:%
\begin{equation*}
\left( \int_{D_{3}}\left\vert f\right\vert ^{\alpha }dx\right) ^{1/\alpha
}\leq C\int_{D_{3}}\left\vert \nabla f\right\vert dx
\end{equation*}%
where $\alpha \in (1,2)$ is arbitrary and $C=C\left( \alpha \right) .$
Replacing $f$ by $f^{\beta }$ (where $\beta >1$), we obtain%
\begin{eqnarray*}
\left( \int_{D_{3}}\left\vert f\right\vert ^{\alpha \beta }dx\right)
^{1/\alpha } &\leq &C\int_{D_{3}}\left\vert \nabla f\right\vert \left\vert
f\right\vert ^{\beta -1}dx \\
&\leq &C\left( \int_{D_{3}}\left\vert \nabla f\right\vert ^{2}dx\right)
^{1/2}\left( \int_{D_{3}}\left\vert f\right\vert ^{2\left( \beta -1\right)
}dx\right) ^{1/2}.
\end{eqnarray*}%
Choosing $\beta $ to satisfy the identity $\alpha \beta =2\left( \beta
-1\right) $, that is, $\beta =\frac{2}{2-\alpha }$, we obtain%
\begin{equation}
\left( \int_{D_{3}}\left\vert f\right\vert ^{\frac{2\alpha }{2-\alpha }%
}dx\right) ^{\frac{2-\alpha }{\alpha }}\leq C\int_{D_{3}}\left\vert \nabla
f\right\vert ^{2}dx.  \label{fSob}
\end{equation}%
This inequality extends routinely to $W^{1,2}$ functions $f$ supported in $%
\overline{D_{3}}$. Applying (\ref{fSob}) with for $f=u^{\ast }$ with $\alpha
=\frac{2p}{2p-1}$ we obtain 
\begin{equation*}
\left( \int_{D_{3}}\left\vert u^{\ast }\right\vert ^{\frac{2p}{p-1}%
}dx\right) ^{1-1/p}\leq C\int_{D_{3}}\left\vert \nabla u^{\ast }\right\vert
^{2}dx,
\end{equation*}%
which together with (\ref{uV}), (\ref{Vu2}) yields%
\begin{equation}
\int_{D_{3}}\left\vert \nabla u^{\ast }\right\vert ^{2}dx\leq C\left(
\int_{D}V^{p}dx\right) ^{1/p}\int_{D_{3}}\left\vert \nabla u^{\ast
}\right\vert ^{2}dx.  \label{CDu}
\end{equation}%
Assuming that%
\begin{equation}
\left\Vert V\right\Vert _{L^{p}\left( D\right) }\leq \varepsilon :=\frac{1}{%
2C},  \label{VC}
\end{equation}%
we see that (\ref{CDu}) is only possible if $u^{\ast }=\func{const}.$ Since $%
u\bot 1$ in $L^{2}\left( D\right) $, it follows that $u\equiv 0.$

Hence, $\mathcal{E}_{V,D}\left( u\right) \leq 0$ and $u\bot 1$ imply $%
u\equiv 0,$ whence $\func{Neg}\left( V,D\right) \leq 1\ $follows.
\end{proof}

\begin{corollary}
\label{CoreOm}Let $\Omega $ be a bounded domain in $\mathbb{R}^{2}$ and $%
\Phi :D\rightarrow \Omega $ be a $C^{1}$-diffeomorphism with finite $M_{\Phi
}$ and $\sup \left\vert J_{\Phi }\right\vert $. Then there is $\varepsilon
_{\Omega }>0$ such that%
\begin{equation*}
\left\Vert V\right\Vert _{L^{p}\left( \Omega \right) }\leq \varepsilon
_{\Omega }\Rightarrow \func{Neg}\left( V,\Omega \right) =1,
\end{equation*}%
where $\varepsilon _{\Omega }$ depends on $p$, $M_{\Phi }$ and $\sup
\left\vert J_{\Phi }\right\vert $.

\label{CoreOmr}Consequently, if $\Omega $ is bilipschitz equivalent to $%
D_{r} $ then 
\begin{equation}
\left\Vert V\right\Vert _{L^{p}\left( \Omega \right) }\leq
cr^{2/p-2}\Rightarrow \func{Neg}\left( V,\Omega \right) =1,  \label{r}
\end{equation}%
where $c>0$ depends on $p$ and on the bilipschitz constant of the mapping
between $D_{r}$ and $\Omega $.
\end{corollary}

\begin{proof}
\label{Rem: shorten this}By Lemma \ref{LemVW}, we have%
\begin{equation*}
\func{Neg}\left( V,\Omega \right) \leq \func{Neg}(\widetilde{V},D),
\end{equation*}%
where $\widetilde{V}$ is given by (\ref{Vti}). By Lemma \ref{LemNeg=1}, 
\begin{equation*}
\Vert \widetilde{V}\Vert _{L^{p}\left( D\right) }\leq \varepsilon
\Rightarrow \func{Neg}(\widetilde{V},D)=1.
\end{equation*}%
Using the notation of Lemma \ref{LemVW}, set $\Psi =\Phi ^{-1},$ $\widetilde{%
W}\equiv 1$ and define a function $W\left( x\right) $ on $\Omega $ by (\ref%
{WW}), that is, 
\begin{equation*}
W\left( x\right) =M_{\Phi }^{p}\left\vert J_{\Psi }\left( x\right)
\right\vert ^{1-p}\leq M_{\Phi }^{p}\sup \left\vert J_{\Phi }\right\vert
^{p-1}.
\end{equation*}%
Then by (\ref{VW}) we have%
\begin{equation*}
\int_{D}\widetilde{V}\left( y\right) ^{p}dy=\int_{\Omega }V\left( x\right)
^{p}W\left( x\right) dx\leq M_{\Phi }^{p}\sup \left\vert J_{\Phi
}\right\vert ^{p-1}\int_{\Omega }V\left( x\right) ^{p}dx,
\end{equation*}%
whence%
\begin{equation*}
\Vert \widetilde{V}\Vert _{L^{p}\left( D\right) }\leq M_{\Phi }\sup
\left\vert J_{\Phi }\right\vert ^{\frac{p-1}{p}}\left\Vert V\right\Vert
_{L^{p}\left( \Omega \right) }.
\end{equation*}%
Therefore, if 
\begin{equation}
\left\Vert V\right\Vert _{L^{p}\left( \Omega \right) }\leq \varepsilon
_{\Omega }:=\frac{\varepsilon }{M_{\Phi }\sup \left\vert J_{\Phi
}\right\vert ^{\frac{p-1}{p}}},  \label{eOm}
\end{equation}%
\emph{\ }then $\Vert \widetilde{V}\Vert _{L^{p}\left( D\right) }\leq
\varepsilon ,$ which implies by the above argument $\func{Neg}\left(
V,\Omega \right) =1.$

Let $\Omega =D_{r}.$ Then, for the mapping $\Phi \left( x\right) =rx$, we
have $M_{\Phi }=1$ and $\left\vert J_{\Phi }\right\vert =r^{2}$ whence we
obtain%
\begin{equation}
\varepsilon _{D_{r}}=\varepsilon r^{2/p-2}.  \label{eDr}
\end{equation}%
More generally, assume that there is a bilipschitz mapping $\Phi
:D_{r}\rightarrow \Omega $ with a bilipschitz constant $L$. Arguing as in
the first part of the proof but using $D_{r}$ instead of $D$, we obtain
similarly to (\ref{eOm}) that $\varepsilon _{\Omega }$ can be determined by%
\begin{equation*}
\varepsilon _{\Omega }=\frac{\varepsilon _{D_{r}}}{M_{\Phi }\sup \left\vert
J_{\Phi }\right\vert ^{\frac{p-1}{p}}}\geq cr^{2/p-2},
\end{equation*}%
where $c>0$ depends on $p$ and $L$, which was to be proved.
\end{proof}

\subsection{Negative eigenvalues in a square}

Denote by $Q$ the unit square in $\mathbb{R}^{2}$, that is,%
\begin{equation*}
Q=\left\{ \left( x_{1},x_{2}\right) \in \mathbb{R}^{2}:0<x_{1}<1,0<x_{2}<1%
\right\} .
\end{equation*}

\begin{lemma}
\label{LemMain}For any $p>1$ and for any potential $V$ in $Q,$%
\begin{equation}
\func{Neg}\left( {V,Q}\right) \leq 1+C\left\Vert V\right\Vert _{L^{p}\left(
Q\right) },  \label{Neg<}
\end{equation}%
where $C$ depends only on $p$.
\end{lemma}

\begin{remark}
\RM Combining Lemma \ref{LemMain} with Lemma \ref{LemVW} we obtain that if
an open set $\Omega \subset \mathbb{R}^{2}$ is bilipschitz equivalent to $Q$%
, then 
\begin{equation*}
\func{Neg}\left( V,\Omega \right) \leq 1+C\left\Vert V\right\Vert
_{L^{p}\left( \Omega \right) },
\end{equation*}%
where the constant $C$ depends on $p$ and on the Lipschitz constant.
\end{remark}

\begin{proof}
It suffices to construct a partition $\mathcal{P}$ of $Q$ into a family of $%
N $ disjoint subsets such that

\begin{enumerate}
\item $\func{Neg}\left( V,\Omega \right) =1$ for any $\Omega \in \mathcal{P}%
; $

\item $N\leq 1+C\left\Vert V\right\Vert _{L^{p}\left( Q\right) }.$
\end{enumerate}

Indeed, if such a partition exists then we obtain by Lemma \ref{LemSub} that%
\begin{equation}
\func{Neg}\left( {V,Q}\right) \leq \sum_{\Omega \in \mathcal{P}}\func{Neg}%
\left( V,\Omega \right) =N,  \label{NN}
\end{equation}%
and (\ref{Neg<}) follows from the above bound of $N$.

The elements of a partition -- tiles, will be of two shapes: any tile is
either a square of the side length $l\in (0,1]$ or a \emph{step}, that is, a
set of the form $\Omega =A\setminus B$ where $A$ is a square of the side
length $l$, and $B$ is a square of the side length $\leq l/2$ that is
attached to one of corners of $A$ (see Fig. \ref{pic3}).

\FRAME{ftbphFU}{5.0704in}{1.2592in}{0pt}{\Qcb{A square and a step of size $l$%
}}{\Qlb{pic3}}{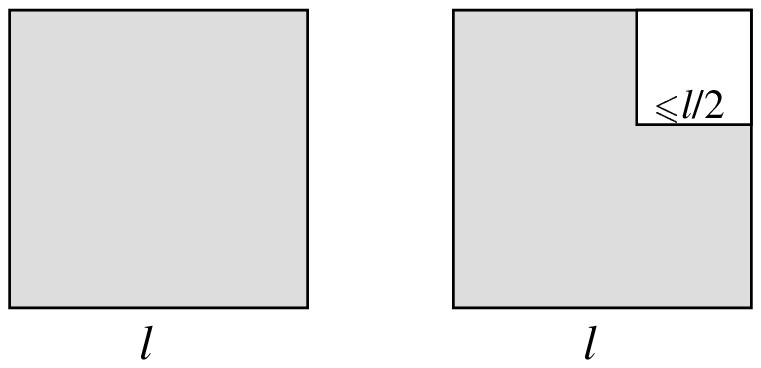}{\special{language "Scientific Word";type
"GRAPHIC";maintain-aspect-ratio TRUE;display "USEDEF";valid_file "F";width
5.0704in;height 1.2592in;depth 0pt;original-width 6.3737in;original-height
1.5238in;cropleft "0";croptop "1";cropright "1";cropbottom "0";filename
'pic3.eps';file-properties "XNPEU";}}

In the both cases we refer to $l$ as the size of $\Omega $. By Corollary \ref%
{CoreOmr}, the condition $\func{Neg}\left( V,\Omega \right) =1$ for a tile $%
\Omega $ will follow from 
\begin{equation}
\int_{\Omega }V^{p}dx\leq cl^{2-2p},  \label{cl}
\end{equation}%
with some constant $c>0$ depending only on $p$.

Apart from the shape, we will distinguish also the \emph{type} of a tile $%
\Omega \in \mathcal{P}$ of size $l$ as follows: we say that

\begin{itemize}
\item $\Omega $ is of a large type, if 
\begin{equation*}
\int_{\Omega }V^{p}dx>cl^{2-2p};
\end{equation*}

\item $\Omega $ is of a medium type if 
\begin{equation}
c^{\prime }l^{2-2p}<\int_{\Omega }V^{p}dx\leq cl^{2-2p};  \label{medium}
\end{equation}

\item $\Omega $ is of small type if%
\begin{equation}
\int_{\Omega }V^{p}dx\leq c^{\prime }l^{2-2p}.  \label{small}
\end{equation}%
Here $c$ is the constant from (\ref{cl}) and $c^{\prime }>0$ is another
constant that satisfies%
\begin{equation}
4c^{\prime }2^{2p-2}<c.  \label{c'}
\end{equation}
\end{itemize}

The construction of the partition $\mathcal{P}$ will be done by induction.
At each step $i\geq 1$ of induction we will have a partition $\mathcal{P}%
^{\left( i\right) }$ of $Q$ such that

\begin{enumerate}
\item each tile $\Omega \in \mathcal{P}^{\left( i\right) }$ is either a
square or a step;

\item If $\Omega \in \mathcal{P}^{\left( i\right) }$ is a step then $\Omega $
is of a medium type.
\end{enumerate}

At step $1$ we have just one set: $\mathcal{P}^{\left( 1\right) }=\left\{
Q\right\} $. At any step $i\geq 1$, partition $\mathcal{P}^{\left(
i+1\right) }$ is obtained from $\mathcal{P}^{\left( i\right) }$ as follows.
If $\Omega \in \mathcal{P}^{\left( i\right) }$ is small or medium then $%
\Omega $ becomes one of the elements of the partition $\mathcal{P}^{\left(
i+1\right) }.$ If $\Omega \in \mathcal{P}^{\left( i\right) }$ is large, then
it is a square, and it will be further partitioned into a few smaller tiles
that will become elements of $\mathcal{P}^{\left( i+1\right) }$. Denoting by 
$l$ the side length of the square $\Omega $, let us first split $\Omega $
into four equal squares $\Omega _{1},\Omega _{2},\Omega _{3},\Omega _{4}$ of
side length $l/2$ and consider the following cases (see Fig. \ref{pic4}).%
\FRAME{ftbphFU}{5.0695in}{1.1026in}{0pt}{\Qcb{Various possibilities of
partitioning of a square $\Omega $ (the shaded tiles are of medium or large
type, the hatched tile $\Omega _{1}$ can be of any type)}}{\Qlb{pic4}}{%
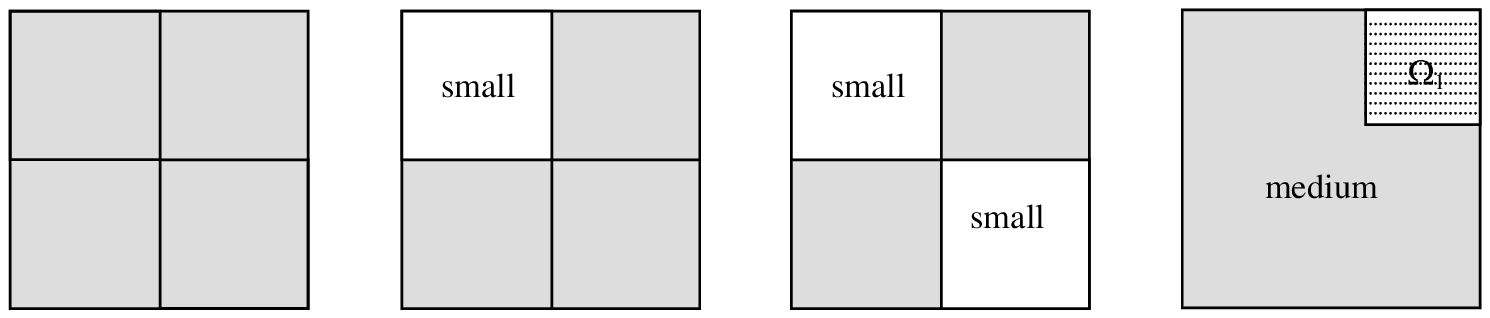}{\special{language "Scientific Word";type
"GRAPHIC";maintain-aspect-ratio TRUE;display "USEDEF";valid_file "F";width
5.0695in;height 1.1026in;depth 0pt;original-width 6.3737in;original-height
1.5394in;cropleft "0";croptop "1";cropright "1";cropbottom "0";filename
'pic4.eps';file-properties "XNPEU";}}

\begin{case}[Case 1.]
If among $\Omega _{1},...,\Omega _{4}$ the number of small type squares is
at most $2,$ then all the sets $\Omega _{1},...,\Omega _{4}$ become elements
of $\mathcal{P}^{\left( i+1\right) }$.
\end{case}

\begin{case}[Case 2.]
If among $\Omega _{1},...,\Omega _{4}$ there are exactly $3$ small type
squares, say, $\Omega _{2},\Omega _{3},\Omega _{4},$ then we have%
\begin{equation*}
\int_{\Omega \setminus \Omega _{1}}V^{p}dx=\int_{\Omega _{2}\cup \Omega
_{3}\cup \Omega _{4}}V^{p}dx\leq 3c^{\prime }\left( \frac{l}{2}\right)
^{2-2p}=3c^{\prime }2^{2p-2}l^{2-2p}<cl^{2-2p},
\end{equation*}%
where we have used (\ref{c'}). On the other hand, we have 
\begin{equation*}
\int_{\Omega }V^{p}dx>cl^{2-2p}.
\end{equation*}%
Therefore, by reducing the size of $\Omega _{1}$ (but keeping $\Omega _{1}$
attached to the corner of $\Omega $) one can achieve the equality 
\begin{equation*}
\int_{\Omega \setminus \Omega _{1}}V^{p}dx=cl^{2-2p}.
\end{equation*}%
Hence, we obtain a partition of $\Omega $ into two sets $\Omega _{1}$ and $%
\Omega \setminus \overline{\Omega _{1}}$, where the step $\Omega \setminus 
\overline{\Omega _{1}}$ is of medium type, while the square $\Omega _{1}$
can be of any type. The both sets $\Omega _{1}$ and $\Omega \setminus 
\overline{\Omega _{1}}$ become elements of $\mathcal{P}^{\left( i+1\right) }$%
.
\end{case}

\begin{case}[Case 3.]
Let us show that all $4$ squares $\Omega _{1},...,\Omega _{4}$ cannot be
small. Indeed, in this case we would have by (\ref{c'})%
\begin{equation*}
\int_{\Omega }V^{p}dx=\sum_{k=1}^{4}\int_{\Omega _{k}}V^{p}dx\leq 4c^{\prime
}\left( \frac{l}{2}\right) ^{2-2p}=\left( 4c^{\prime }2^{2p-2}\right)
l^{2-2p}<cl^{2-2p},
\end{equation*}%
which contradicts to the assumption that $\Omega $ is of large type.
\end{case}

As we see from the construction, at each step $i$ only large type squares
get partitioned further, and the size of the large type squares in $\mathcal{%
P}^{\left( i+1\right) }$ reduces at least by a factor $2.$ If the size of a
square is small enough then it is necessarily of small type, because the
right hand side of (\ref{small}) goes to $\infty $ as $l\rightarrow 0.$
Hence, the process stops after finitely many steps, and we obtain a
partition $\mathcal{P}$ where all the tiles are either of small or medium
types (see Fig. \ref{pic5}). In particular, we have $\func{Neg}\left(
V,\Omega \right) =1$ for any $\Omega \in \mathcal{P}.$\FRAME{ftbphFU}{%
4.4391in}{2.3402in}{0pt}{\Qcb{An example of a final partition $\mathcal{P}$.
The shaded tiles are of medium type, the white squares are of small type. }}{%
\Qlb{pic5}}{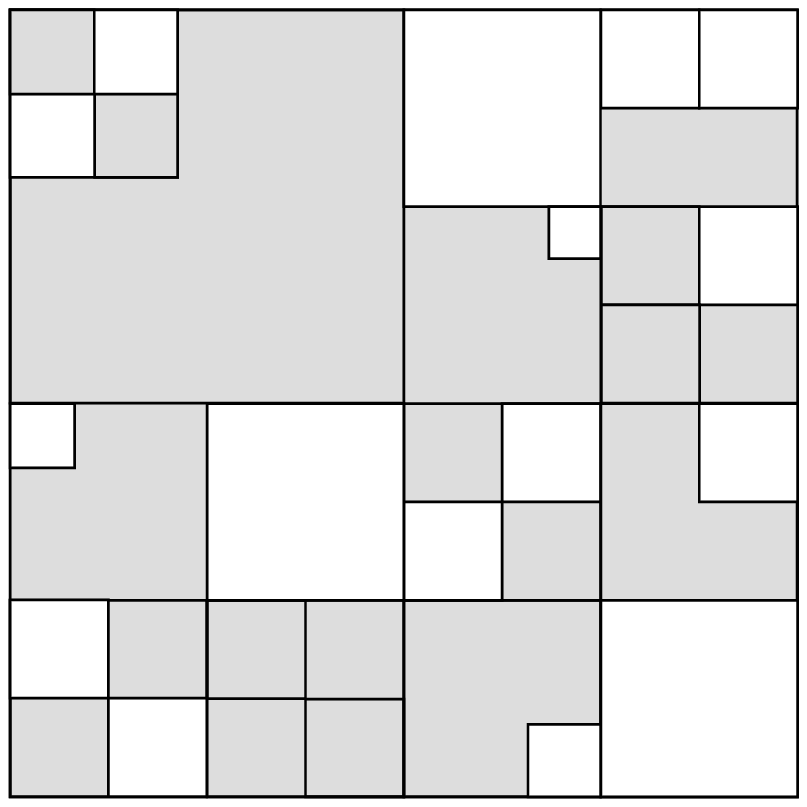}{\special{language "Scientific Word";type
"GRAPHIC";maintain-aspect-ratio TRUE;display "USEDEF";valid_file "F";width
4.4391in;height 2.3402in;depth 0pt;original-width 6.4411in;original-height
3.141in;cropleft "0";croptop "1";cropright "1";cropbottom "0";filename
'pic5.eps';file-properties "XNPEU";}}

Let $N$ be the number of tiles in $\mathcal{P}$. We need to show that 
\begin{equation}
N\leq 1+C\left\Vert V\right\Vert _{L^{p}\left( Q\right) }.  \label{N<1+V}
\end{equation}%
At each step of construction, denote by $L$ the number of large tiles, by $M$
the number of medium tiles, and by $S$ the number of small tiles. Let us
show that the quantity $2L+3M-S$ is non-decreasing during the construction.
Indeed, at each step we split one large square $\Omega $, so that by
removing this square, $L$ decreases by $1$. However, we add new tiles that
contribute to the quantity $2L+3M-S$ as follows.

\begin{enumerate}
\item If $\Omega $ is split into $s\leq 2$ small and $4-s$ medium/large
squares as in Case 1, then the value of $2L+3M-S$ has the increment at least%
\begin{equation*}
-2+2\left( 4-s\right) -s=6-3s\geq 0.
\end{equation*}

\item If $\Omega $ is split into $1$ square and $1$ step as in Case 2, then
one obtains at least $1$ medium tile and at most $1$ small tile, so that $%
2L+3M-S$ has the increment at least 
\begin{equation*}
-2+3-1=0.
\end{equation*}
\end{enumerate}

(Luckily, Case 3 cannot occur. In that case, we would have $4$ new small
squares so that $L$ and $M$ would not have increased, whereas $S$ would have
increased at least by $3$, so that no quantity of the type $C_{1}L+C_{2}M-S$
would have been monotone increasing).

Since for the partition $\mathcal{P}^{\left( 1\right) }$ we have $%
2L+3M-S\geq -1$, this inequality remains true at all steps of construction
and, in particular, it is satisfied for the final partition $\mathcal{P}$.
For the final partition we have $L=0$, whence it follows that $S\leq 1+3M$
and, hence,%
\begin{equation}
N=S+M\leq 1+4M.  \label{N4M}
\end{equation}

Let us estimate $M$. Let $\Omega _{1},...,\Omega _{M}$ be the medium type
tiles of $\mathcal{P}$ and let $l_{k}$ be the size of $\Omega _{k}$. Each $%
\Omega _{k}$ contains a square $\Omega _{k}^{\prime }\subset \Omega _{k}$ of
the size $l_{k}/2$, and all the squares $\left\{ \Omega _{k}^{\prime
}\right\} _{k=1}^{M}$ are disjoint, which implies that%
\begin{equation}
\sum_{k=1}^{M}l_{k}^{2}\leq 4.  \label{l4}
\end{equation}%
Using the H\"{o}lder inequality and (\ref{l4}), we obtain%
\begin{equation*}
M=\sum_{k=1}^{M}l_{k}^{\frac{2}{p^{\prime }}}l_{k}^{-\frac{2}{p^{\prime }}%
}\leq \left( \sum_{k=1}^{M}l_{k}^{2}\right) ^{1/p^{\prime }}\left(
\sum_{k=1}^{M}l_{k}^{-\frac{2p}{p^{\prime }}}\right) ^{1/p}\leq
4^{1/p^{\prime }}\left( \sum_{k=1}^{M}l_{k}^{2-2p}\right) ^{1/p}.
\end{equation*}%
Since by (\ref{medium}) $c^{\prime }l_{k}^{2-2p}<\int_{\Omega _{k}}V^{p}dx$,
it follows that 
\begin{equation*}
M\leq C\left( \sum_{k=1}^{M}\int_{\Omega _{k}}V^{p}dx\right) ^{1/p}\leq
C\left( \int_{Q}V^{p}dx\right) ^{1/p}.
\end{equation*}%
Combining this with $N\leq 1+4M$, we obtain $N\leq 1+C\left\Vert
V\right\Vert _{L^{p}\left( Q\right) }$, thus finishing the proof.
\end{proof}

\section{Negative eigenvalues and Green operator}

\setcounter{equation}{0}\label{SecR2one}

\subsection{Green operator in $\mathbb{R}^{2}$}

We start with the following statement.

\begin{lemma}
\label{LemV0}There exists non-negative non-zero function $V_{0}\in
C_{0}^{\infty }\left( \mathbb{R}^{2}\right) $ such that $\func{Neg}\left(
V_{0}\right) =1.$
\end{lemma}

\begin{proof}
Choose $V_{0}$ to be supported in the unit disk $D$ and such that $%
\left\Vert 2V_{0}\right\Vert _{L^{p}\left( D\right) }$ is small enough as in
Lemma \ref{LemNeg=1}, so that $\func{Neg}\left( 2V_{0},D\right) =1.$ By
Lemma \ref{Lem2D} we have $\func{Neg}\left( V_{0},\mathbb{R}^{2}\right) \leq 
\func{Neg}\left( 2V_{0},D\right) ,$ whence the claim follows.
\end{proof}

From now on let us fix a potential $V_{0}$ as in Lemma \ref{LemV0}. We can
always assume that $V_{0}$ is spherically symmetric. Consider the quadratic
form $\mathcal{E}_{0}$%
\begin{equation*}
\mathcal{E}_{0}\left( u\right) :=\int_{\mathbb{R}^{2}}\left\vert \nabla
u\right\vert ^{2}dx+\int_{\mathbb{R}^{2}}V_{0}u^{2}dx=\mathcal{E}%
_{-V_{0}}\left( u\right) ,
\end{equation*}%
defined on the space 
\begin{equation*}
\mathcal{F}_{0}=\left\{ u\in L_{loc}^{2}\left( \mathbb{R}^{2}\right) :\int_{%
\mathbb{R}^{2}}\left\vert \nabla u\right\vert ^{2}dx<\infty \right\} =%
\mathcal{F}_{V_{0}}.
\end{equation*}%
Since $V_{0}$ is bounded and has compact support, the condition $\int_{%
\mathbb{R}^{2}}V_{0}u^{2}dx<\infty $ is satisfied for any $u\in L_{loc}^{2}.$
Note also that $\mathcal{F}_{V}\subset \mathcal{F}_{0}$ for any potential $V$%
.

\begin{lemma}
\label{LemVV0}If, for all $u\in \mathcal{F}_{0}$,%
\begin{equation}
\mathcal{E}_{0}\left( u\right) \geq 2\int_{\mathbb{R}^{2}}Vu^{2}dx,
\label{2V0}
\end{equation}%
then $\func{Neg}\left( V\right) =1.$
\end{lemma}

\begin{proof}
If $\mathcal{E}_{V}\left( u\right) \leq 0$ that is, if%
\begin{equation*}
\int_{\mathbb{R}^{2}}\left\vert \nabla u\right\vert ^{2}dx\leq \int_{\mathbb{%
R}^{2}}Vu^{2}dx,
\end{equation*}%
then, substituting this into the right hand side of (\ref{2V0}), we obtain 
\begin{equation*}
\int_{\mathbb{R}^{2}}\left\vert \nabla u\right\vert ^{2}dx+\int_{\mathbb{R}%
^{2}}V_{0}u^{2}dx\geq 2\int_{\mathbb{R}^{2}}\left\vert \nabla u\right\vert
^{2}dx,
\end{equation*}%
whence%
\begin{equation*}
\int_{\mathbb{R}^{2}}\left\vert \nabla u\right\vert ^{2}dx\leq \int_{\mathbb{%
R}^{2}}V_{0}u^{2}dx,
\end{equation*}%
that is, $\mathcal{E}_{V_{0}}\left( u\right) \leq 0$. By Lemma \ref{LemL}
this implies $\func{Neg}\left( V\right) \leq \func{Neg}\left( V_{0}\right) $%
, whence the claim follows.
\end{proof}

Lemma \ref{LemVV0} provides the following method of proving that $\func{Neg}%
\left( V\right) =1$: it suffices to prove the inequality (\ref{2V0}) for all 
$u\in \mathcal{F}_{0}.$ For the latter, we will use the Green function of
the operator 
\begin{equation*}
H_{0}=-\Delta +V_{0}.
\end{equation*}%
It was shown in \cite[Example 10.14]{GrigW} that the operator $H_{0}$ has a
symmetric positive Green function $g\left( x,y\right) $ that satisfies the
following estimate 
\begin{equation}
g\left( x,y\right) \simeq \ln \langle y\rangle +\frac{\ln \langle y\rangle }{%
\ln \langle x\rangle }\ln _{+}\frac{1}{\left\vert x-y\right\vert }\ \ \ 
\text{if\ \ }\left\vert y\right\vert \leq \left\vert x\right\vert ,
\label{gH0}
\end{equation}%
and a symmetric estimate if $\left\vert y\right\vert \geq \left\vert
x\right\vert $, where we use the notation 
\begin{equation*}
\left\langle x\right\rangle =e+\left\vert x\right\vert .
\end{equation*}%
It follows from (\ref{gH0}) that, for all $x,y\in \mathbb{R}^{2}$,%
\begin{equation}
g\left( x,y\right) \simeq \ln \left\langle x\right\rangle \wedge \ln
\left\langle y\right\rangle +\ln _{+}\frac{1}{\left\vert x-y\right\vert },
\label{g<}
\end{equation}%
where $a\wedge b:=\min \left( a,b\right) .$ Here we have used the fact that $%
\left\langle x\right\rangle \simeq \left\langle y\right\rangle $ provided $%
\left\vert x-y\right\vert <1$; note that the latter is equivalent to $\ln
_{+}\frac{1}{\left\vert x-y\right\vert }>0.$

For comparison, let us recall that the operator $-\Delta $ in $\mathbb{R}%
^{2} $ has no positive Green function, so that adding a small perturbation $%
V_{0}$ changes this property.

Fix a potential $V$ on $\mathbb{R}^{2}$, consider a measure $\nu $ on $%
\mathbb{R}^{2}$ given by%
\begin{equation*}
d\nu =V\left( x\right) dx,
\end{equation*}%
and the integral operator $G_{V}$ in $L^{2}\left( \nu \right) =L^{2}\left( 
\mathbb{R}^{2},\nu \right) $ that acts by the rule%
\begin{equation*}
G_{V}f\left( x\right) =\int_{\mathbb{R}^{2}}g\left( x,y\right) f\left(
y\right) d\nu \left( y\right) .
\end{equation*}%
Denote by $\left\Vert G_{V}\right\Vert $ the norm of the operator $G$ from $%
L^{2}\left( \nu \right) $ to $L^{2}\left( \nu \right) $ (if $G_{V}$ does not
map $L^{2}\left( \nu \right) $ into itself then set $\left\Vert
G_{V}\right\Vert =\infty $).

\begin{lemma}
\label{LemGf}Assume that 
\begin{equation}
\frac{1}{V}\in L_{loc}^{1}.  \label{1/V}
\end{equation}%
Then following inequality holds for all $u\in \mathcal{F}_{0}$:%
\begin{equation}
\mathcal{E}_{0}\left( u\right) \geq \frac{1}{\left\Vert G_{V}\right\Vert }%
\int_{\mathbb{R}^{2}}Vu^{2}dx.  \label{E0>}
\end{equation}
\end{lemma}

\begin{proof}
If $\left\Vert G_{V}\right\Vert =\infty $ then (\ref{E0>}) is trivially
satisfied, so assume that $\left\Vert G_{V}\right\Vert <\infty .$ Consider
first the case when $u\in C_{0}^{\infty }\left( \mathbb{R}^{2}\right) $. Set 
$f=\frac{1}{V}H_{0}u$ so that $H_{0}u=fV.$ Then function $u$ can be
recovered from $f$ using the Green operator $G=G_{V}$ as follows:%
\begin{equation*}
u\left( x\right) =\int_{\mathbb{R}^{2}}g\left( x,y\right) \left( fV\right)
\left( y\right) dy=Gf\left( x\right) .
\end{equation*}%
Observe that $f\in L^{2}\left( \nu \right) $ because by (\ref{1/V}) 
\begin{equation*}
\int_{\mathbb{R}^{2}}f^{2}d\nu =\int_{\mathbb{R}^{2}}\frac{\left(
H_{0}u\right) ^{2}}{V}dx\leq \sup \left\vert H_{0}u\right\vert ^{2}\int_{%
\limfunc{supp}u}\frac{1}{V}dx<\infty .
\end{equation*}%
It follows that%
\begin{equation}
\mathcal{E}_{0}\left( u\right) =\left( H_{0}u,u\right) _{L^{2}\left(
dx\right) }=\left( fV,Gf\right) _{L^{2}\left( dx\right) }=\left( f,Gf\right)
_{L^{2}\left( \nu \right) }  \label{5}
\end{equation}%
and%
\begin{equation*}
\int_{\mathbb{R}^{2}}Vu^{2}dx=\left( u,u\right) _{L^{2}\left( \nu \right)
}=\left( Gf,Gf\right) _{L^{2}\left( \nu \right) }.
\end{equation*}%
Inequality (\ref{E0>}) will follows if we prove that, for all $f\in
L^{2}(\nu )$, 
\begin{equation}
\left( f,Gf\right) \geq \frac{1}{\left\Vert G\right\Vert }\left(
Gf,Gf\right) ,  \label{fG}
\end{equation}%
where the both inner products are in $L^{2}\left( \nu \right) .$

Recall that $G$ is a bounded symmetric (hence, self-adjoint) operator in $%
L^{2}\left( \nu \right) $. Observe that $G$ is non-negative definite.
Indeed, if $f\in C_{0}^{\infty }\left( \mathbb{R}^{2}\right) $ then, setting 
$u=Gf,$ we obtain the identities (\ref{5}) so that%
\begin{equation*}
\left( f,Gf\right) =\mathcal{E}_{0}\left( u\right) \geq 0.
\end{equation*}%
Then $\left( f,Gf\right) \geq 0$ follows from the fact that $C_{0}^{\infty
}\left( \mathbb{R}^{2}\right) $ is dense in $L^{2}\left( \nu \right) $.

Now, let us prove (\ref{fG}). For non-negative definite self-adjoint
operators the following inequality holds, for all $f,h\in L^{2}(\nu )$:%
\begin{equation*}
\left( Gf,h\right) ^{2}\leq \left( Gf,f\right) \left( Gh,h\right) .
\end{equation*}%
Setting $h=Gf,$ we obtain%
\begin{equation*}
\left( Gf,Gf\right) ^{2}\leq \left( Gf,f\right) \left\Vert G\right\Vert
\left\Vert h\right\Vert ^{2}=\left\Vert G\right\Vert \left( Gf,f\right)
\left( Gf,Gf\right) .
\end{equation*}%
Dividing by $\left( Gf,Gf\right) $, we obtain (\ref{fG}).

Hence, we have proved (\ref{E0>}) for $u\in C_{0}^{\infty }\left( \mathbb{R}%
^{2}\right) .$ Let us extend this inequality to all $u\in \mathcal{F}_{0}$.
Assume first that $u\in \mathcal{F}_{0}$ has a compact support. Then it
follows that $u\in W^{1,2}\left( \mathbb{R}^{2}\right) .$ Approximating $u$
in $W^{1,2}$ by a sequence $\left\{ u_{n}\right\} \subset C_{0}^{\infty
}\left( \mathbb{R}^{2}\right) $, applying (\ref{E0>}) for each $u_{n}$ and
passing to the limit using Fatou's lemma, we obtain (\ref{E0>}) for $u$.

Let us now prove (\ref{E0>}) for the case when the function $u\in \mathcal{F}%
_{0}$ is essentially bounded. There is a sequence of non-negative Lipschitz
functions $\varphi _{n}$ on $\mathbb{R}^{2}$ with compact supports such that 
$\varphi _{n}\uparrow 1$ as $n\rightarrow \infty $ and 
\begin{equation}
\int_{\mathbb{R}^{2}}\left\vert \nabla \varphi _{n}\right\vert
^{2}dx\rightarrow 0.  \label{fin0}
\end{equation}%
For example, one can take $\varphi _{n}$ as in (\ref{fiab}) with $a_{n}=n$
and $b_{n}=n^{2}$, that is,%
\begin{equation}
\varphi _{n}\left( x\right) =\min \left( 1,\frac{1}{\ln n}\ln _{+}\frac{n^{2}%
}{\left\vert x\right\vert }\right) .  \label{finR2}
\end{equation}%
Clearly, $\varphi _{n}\in \mathcal{F}_{0}$. By (\ref{Efi}) we have 
\begin{equation*}
\int_{\mathbb{R}^{2}}\left\vert \nabla \varphi _{n}\right\vert ^{2}dx=\frac{%
2\pi }{\ln n}\rightarrow 0\ \text{as }n\rightarrow \infty .
\end{equation*}%
Since (\ref{E0>}) holds for the functions $u_{n}=u\varphi _{n}$ with compact
support, it suffices to show that passing to the limit as $n\rightarrow
\infty $, we obtain (\ref{E0>}) for the function $u$. The terms $\int
V_{0}u_{n}^{2}dx$ and $\int Vu_{n}^{2}dx$ are obviously survive under the
monotone limit. We are left to verify that 
\begin{equation}
\int_{\mathbb{R}^{2}}\left\vert \nabla u_{n}\right\vert ^{2}dx\rightarrow
\int_{\mathbb{R}^{2}}\left\vert \nabla u\right\vert ^{2}dx.  \label{fin}
\end{equation}%
We have%
\begin{eqnarray}
&&\int_{\mathbb{R}^{2}}\left\vert \nabla \left( u\varphi _{n}\right)
\right\vert ^{2}dx  \notag \\
&=&\int_{\mathbb{R}^{2}}\left\vert \nabla u\right\vert ^{2}\varphi
_{n}^{2}dx+2\int_{\mathbb{R}^{2}}\langle \nabla u,\nabla \varphi _{n}\rangle
u\varphi _{n}dx+\int_{\mathbb{R}^{2}}u^{2}\left\vert \nabla \varphi
_{n}\right\vert ^{2}dx.  \label{1=3}
\end{eqnarray}%
The first term in the right hand side of (\ref{1=3}) converges to $\int_{%
\mathbb{R}^{2}}\left\vert \nabla u\right\vert ^{2}dx$. For the third term we
have by (\ref{fin0}) 
\begin{equation*}
\int_{\mathbb{R}^{2}}u^{2}\left\vert \nabla \varphi _{n}\right\vert
^{2}dx\leq \left\Vert u\right\Vert _{L^{\infty }}^{2}\int_{\mathbb{R}%
^{2}}\left\vert \nabla \varphi _{n}\right\vert ^{2}dx\rightarrow 0\ \text{as 
}n\rightarrow \infty .
\end{equation*}%
Similarly, the middle term converges to $0$ as $n\rightarrow \infty $ by%
\begin{equation*}
\left\vert \int_{\mathbb{R}^{2}}\langle \nabla u,\nabla \varphi _{n}\rangle
u\varphi _{n}dx\right\vert \leq \left( \int_{\mathbb{R}^{2}}\left\vert
\nabla u\right\vert ^{2}\varphi _{n}^{2}dx\right) ^{1/2}\left( \int_{\mathbb{%
R}^{2}}u^{2}\left\vert \nabla \varphi _{n}\right\vert ^{2}dx\right)
\rightarrow 0,
\end{equation*}%
which proves (\ref{fin}).

Finally, for a general function $u\in \mathcal{F}_{0}$, consider an
approximating sequence 
\begin{equation*}
u_{n}=\max \left( \min \left( u,n\right) ,-n\right) .
\end{equation*}%
The function $u_{n}$ is bounded so that (\ref{E0>}) holds for $u_{n}$.
Letting $n\rightarrow \infty $, we obtain (\ref{E0>}) for the function $u$.
\end{proof}

\begin{corollary}
\label{CorR2}Under the hypothesis \emph{(\ref{1/V})}, 
\begin{equation}
\left\Vert G_{V}\right\Vert \leq \frac{1}{2}\Rightarrow \func{Neg}\left( {V}%
\right) =1.  \label{bV1}
\end{equation}
\end{corollary}

\begin{proof}
Indeed, (\ref{bV1}) implies (\ref{2V0}), whence $\func{Neg}\left( V\right)
=1 $ holds by Lemma \ref{LemVV0}.
\end{proof}

\subsection{Green operator in a strip}

Consider a strip 
\begin{equation*}
S=\left\{ \left( x_{1},x_{2}\right) \in \mathbb{R}^{2}:x_{1}\in \mathbb{R},\
0<x_{2}<\pi \right\}
\end{equation*}%
and a potential $V$ on $S$. The analytic function $\Psi \left( z\right)
=e^{z}$ provides a biholomorphic mapping from $S$ onto the upper half-plane $%
H_{+}$. Set $\Phi =\Psi ^{-1}$ so that $\Phi \left( z\right) =\ln z$.
Consider the function 
\begin{equation}
\gamma \left( z,w\right) =g\left( \Psi \left( z\right) ,\Psi \left( w\right)
\right) ,  \label{gadef}
\end{equation}%
where $z,w\in S$ and $g\left( x,y\right) $ is the Green function from
Section \ref{SecR2one}. Consider also the corresponding integral operator%
\begin{equation}
\Gamma _{V}f\left( z\right) =\int_{S}\gamma \left( z,\cdot \right) f\left(
\cdot \right) d\nu ,  \label{Gadef}
\end{equation}%
where measure $\nu $ is defined as above by $d\nu =V\left( x\right) dx$.
Denote by $\left\Vert \Gamma _{V}\right\Vert $ the norm of $\Gamma _{V}$ in $%
L^{2}\left( S,\nu \right) $.

\begin{lemma}
\label{Lem1/8}Let 
\begin{equation}
\frac{1}{V}\in L_{loc}^{1}\left( S\right) .  \label{1/VS}
\end{equation}%
Then 
\begin{equation*}
\left\Vert \Gamma _{V}\right\Vert \leq \frac{1}{8}\Rightarrow \func{Neg}%
\left( V,S\right) =1.
\end{equation*}
\end{lemma}

\begin{proof}
Consider the potential $\widetilde{V}$ on the half-plane $H_{+}=\left\{
x_{2}>0\right\} $ given by 
\begin{equation*}
\widetilde{V}\left( x\right) =V\left( \Phi \left( x\right) \right)
\left\vert \Phi ^{\prime }\left( x\right) \right\vert ^{2},
\end{equation*}%
for which we have by Lemma \ref{LemVW} that%
\begin{equation}
\func{Neg}\left( V,S\right) \leq \func{Neg}(\widetilde{V},H_{+}).
\label{NSNH}
\end{equation}%
Let us extend $\widetilde{V}$ from $H_{+}$ to $\mathbb{R}^{2}$ by symmetry
in the axis $x_{1}$. By Lemma \ref{LemH+} we have%
\begin{equation*}
\func{Neg}(\widetilde{V},H_{+})\leq \func{Neg}(2\widetilde{V},\mathbb{R}%
^{2}).
\end{equation*}%
Consider the operator $G_{\widetilde{V}}$ that acts in $L^{2}\left( \mathbb{R%
}^{2},\widetilde{\nu }\right) $ where $d\widetilde{\nu }=\widetilde{V}dx,$
and $G_{2\widetilde{V}}$ that acts in $L^{2}\left( \mathbb{R}^{2},2%
\widetilde{\nu }\right) $. It is easy to see that 
\begin{equation}
\left\Vert G_{2\widetilde{V}}\right\Vert =2\left\Vert G_{\widetilde{V}%
}\right\Vert ,  \label{gv0}
\end{equation}%
Denote by $\left\Vert G_{\widetilde{V}}\right\Vert _{+}$ the norm of the
operator $G_{\widetilde{V}}$ acting in $L^{2}\left( H_{+},\widetilde{\nu }%
\right) .$ Using the symmetry of the potential $\widetilde{V}$ in the axis $%
x_{1}$ and that of the Green function $g\left( x,y\right) $, one can easily
show that\label{rem: details about axial subsymmetry of the Green function}%
\begin{equation}
\left\Vert G_{\widetilde{V}}\right\Vert \leq 2\left\Vert G_{\widetilde{V}%
}\right\Vert _{+}.  \label{gv1}
\end{equation}%
Let us verify that%
\begin{equation}
\left\Vert G_{\widetilde{V}}\right\Vert _{+}=\left\Vert \Gamma
_{V}\right\Vert .  \label{gv2}
\end{equation}%
In fact, the operators $\Gamma _{V}$ in $L^{2}\left( S,\nu \right) $ and $G_{%
\widetilde{V}}$ in $L^{2}\left( H_{+},\widetilde{\nu }\right) $ are unitary
equivalent. Indeed, consider a mapping $f\mapsto \widetilde{f}$ from $%
L^{2}\left( S,\nu \right) $ to $L^{2}\left( H_{+},\widetilde{\nu }\right) $
defined by%
\begin{equation*}
\widetilde{f}\left( x\right) =f\left( \Phi \left( x\right) \right) .
\end{equation*}%
Then we have%
\begin{eqnarray*}
\Vert \widetilde{f}\Vert _{L^{2}\left( H_{+},\widetilde{\nu }\right) }^{2}
&=&\int_{H_{+}}\widetilde{f}^{2}\left( x\right) \widetilde{V}\left( x\right)
dx=\int_{H_{+}}f\left( \Phi \left( x\right) \right) ^{2}V\left( \Phi \left(
x\right) \right) \left\vert \Phi ^{\prime }\left( x\right) \right\vert ^{2}dx
\\
&=&\int_{S}f\left( z\right) ^{2}V\left( z\right) dz=\left\Vert f\right\Vert
_{L^{2}\left( S,\nu \right) }^{2},
\end{eqnarray*}%
so that this mapping is unitary. Next, we have, for any $x\in H_{+}$,%
\begin{eqnarray*}
\widetilde{\Gamma _{V}f}\left( x\right) &=&\Gamma _{V}f\left( \Phi \left(
x\right) \right) =\int_{S}\gamma \left( \Phi \left( x\right) ,w\right)
f\left( w\right) V\left( w\right) dw \\
&=&\int_{H_{+}}\gamma \left( \Phi \left( x\right) ,\Phi \left( y\right)
\right) f\left( \Phi \left( y\right) \right) V\left( \Phi \left( y\right)
\right) \left\vert \Phi ^{\prime }\left( y\right) \right\vert ^{2}dy \\
&=&\int_{H_{+}}g\left( x,y\right) \widetilde{f}\left( y\right) \widetilde{V}%
\left( y\right) dy \\
&=&G_{\widetilde{V}}\widetilde{f}\left( x\right) ,
\end{eqnarray*}%
that is, $\widetilde{\Gamma _{V}f}=G_{\widetilde{V}}\widetilde{f}$ which
implies the unitary equivalence of $\Gamma _{V}$ and $G_{\widetilde{V}}$. %
\label{rem: details}

Combining (\ref{gv0})-(\ref{gv2}), we conclude that 
\begin{equation*}
\left\Vert G_{2\widetilde{V}}\right\Vert \leq 4\left\Vert \Gamma
_{V}\right\Vert \leq \frac{1}{2}.
\end{equation*}%
Since $\frac{1}{\widetilde{V}}\in L_{loc}^{1}$, we have by Corollary \ref%
{CorR2} that $\func{Neg}(2\widetilde{V},\mathbb{R}^{2})=1.$
\end{proof}

In the next lemma, we prove an upper bound for the Green kernel $\gamma $.

\begin{lemma}
\label{Lemga<}For all $x,y\in S$, we have%
\begin{equation}
\gamma \left( x,y\right) \leq C\left( 1+\left\vert x_{1}\right\vert \wedge
\left\vert y_{1}\right\vert \right) +C\ln _{+}\frac{1}{\left\vert
x-y\right\vert }  \label{ga<}
\end{equation}%
with an absolute constant $C$.
\end{lemma}

\begin{proof}
By (\ref{g<}) and (\ref{gadef}) we have%
\begin{equation*}
\gamma \left( x,y\right) \leq C\ln \left\langle e^{x}\right\rangle \wedge
\ln \left\langle e^{y}\right\rangle +C\ln _{+}\frac{1}{\left\vert
e^{x}-e^{y}\right\vert },
\end{equation*}%
where in the expressions $e^{x},e^{y}$ we regards $x,y$ are complex numbers.
Observe that 
\begin{equation}
\ln \left\langle e^{x}\right\rangle =\ln \left( e+\left\vert
e^{x}\right\vert \right) =\ln \left( e+e^{x_{1}}\right) \leq e+\left\vert
x_{1}\right\vert .  \label{ga1}
\end{equation}%
Let us show that%
\begin{equation}
\ln _{+}\frac{1}{\left\vert e^{x}-e^{y}\right\vert }\leq C+\left\vert
x\right\vert \wedge \left\vert y\right\vert +\ln _{+}\frac{1}{\left\vert
x-y\right\vert },  \label{ga2}
\end{equation}%
with some absolute constant $C$. Indeed, by symmetry between $x,y$, it
suffices to prove that 
\begin{equation}
\left\vert e^{x}-e^{y}\right\vert \geq ce^{-\left\vert x\right\vert }\min
\left( 1,\left\vert x-y\right\vert \right)  \label{ezu}
\end{equation}%
for all $x,y\in S$ and for some positive constant $c$. Indeed, setting $%
z=y-x $ we see that (\ref{ezu}) is equivalent to 
\begin{equation*}
\left\vert 1-e^{z}\right\vert \geq c\min \left( 1,\left\vert z\right\vert
\right) ,
\end{equation*}%
and the latter is true for $\left\vert z\right\vert \leq 1$ because 
\begin{equation*}
\left\vert 1-e^{z}\right\vert =\left\vert z+\sum_{k=2}^{\infty }\frac{z^{k}}{%
k!}\right\vert \geq \left\vert z\right\vert -\left\vert z\right\vert
\sum_{k=2}^{\infty }\frac{1}{k!}=\left( 3-e\right) \left\vert z\right\vert ,
\end{equation*}%
and for $\left\vert z\right\vert >1$ because the set $\left\{ z\in
S:\left\vert z\right\vert >1\right\} $ is separated from the only point $z=0$
in $\overline{S}$ where $e^{z}=1.$

Combining (\ref{ga1}), (\ref{ga2}) and noticing that $\left\vert
x\right\vert \leq \pi +\left\vert x_{1}\right\vert $, we obtain (\ref{ga<}).
\end{proof}

\section{Estimates of the norms of some integral operators}

\label{SecNorm}\setcounter{equation}{0}In this section we introduce tools
for estimating the norm of the operator $\Gamma _{V}$ from the previous
section. We start with an one-dimensional case that contains already all
difficulties.

\begin{lemma}
\label{Lemaln}Let $\mu $ be a Radon measure on $\mathbb{R}$ and consider the
following operator acting on $L^{2}\left( \mathbb{R},\mu \right) $:%
\begin{equation*}
Tf\left( x\right) =\int_{\mathbb{R}}\left( 1+\left\vert x\right\vert \wedge
\left\vert y\right\vert \right) f\left( y\right) d\mu \left( y\right) .
\end{equation*}%
For any $n\in \mathbb{Z}$, set%
\begin{equation*}
I_{n}=\left[ 2^{n-1},2^{n}\right] \ \text{for}\ n>0,\ \ I_{0}=\left[ -1,1%
\right] ,\ \ I_{n}=\left[ -2^{\left\vert n\right\vert },-2^{\left\vert
n\right\vert -1}\right] \ \text{for\ }n<0
\end{equation*}%
and%
\begin{equation}
\alpha _{n}=2^{\left\vert n\right\vert }\mu \left( I_{n}\right) .
\label{aln}
\end{equation}%
Then the following estimate holds:%
\begin{equation*}
\left\Vert T\right\Vert \leq 64\sup_{n\in \mathbb{Z}}\alpha _{n}.
\end{equation*}
\end{lemma}

\begin{proof}
Let us represent the operator $T$ as the sum $T=T_{1}+T_{2}$ where%
\begin{eqnarray*}
T_{1}f\left( x\right) &=&\int_{\left\{ \left\vert y\right\vert \leq
\left\vert x\right\vert \right\} }\left( 1+\left\vert y\right\vert \right)
f\left( y\right) d\mu \left( y\right) \\
T_{2}f\left( x\right) &=&\int_{\left\{ \left\vert y\right\vert \geq
\left\vert x\right\vert \right\} }\left( 1+\left\vert x\right\vert \right)
f\left( y\right) d\mu \left( y\right) .
\end{eqnarray*}%
These operators are clearly adjoint in $L^{2}\left( \mathbb{R},\mu \right) $
which implies that $\left\Vert T_{1}\right\Vert =\left\Vert T_{2}\right\Vert
.$ Hence, $\left\Vert T\right\Vert \leq 2\left\Vert T_{1}\right\Vert .$ The
operator $T_{1}$ can be further split into the sum $T_{1}=T_{3}+T_{4}$ where%
\begin{eqnarray*}
T_{3}f\left( x\right) &=&\int_{0}^{\left\vert x\right\vert }\left(
1+\left\vert y\right\vert \right) f\left( y\right) d\mu \left( y\right) \\
T_{4}f\left( x\right) &=&\int_{-\left\vert x\right\vert }^{0}\left(
1+\left\vert y\right\vert \right) f\left( y\right) d\mu \left( y\right) .
\end{eqnarray*}%
We will estimate $\left\Vert T_{3}\right\Vert $ via $\alpha =\sup \alpha
_{n} $, and by symmetry $\left\Vert T_{4}\right\Vert $ could be estimated in
the same way. The operator $T_{3}$ splits further into the sum $%
T_{3}=T_{5}+T_{6} $ where%
\begin{eqnarray*}
T_{5}f\left( x\right) &=&\int_{0}^{x_{+}}\left( 1+y\right) f\left( y\right)
d\mu \left( y\right) \\
T_{6}f\left( x\right) &=&\int_{0}^{x_{-}}\left( 1+y\right) f\left( y\right)
d\mu \left( y\right) .
\end{eqnarray*}%
Clearly, we have $\left\Vert T_{5}\right\Vert =\left\Vert T_{6}\right\Vert $
and, hence, $\left\Vert T_{3}\right\Vert \leq 2\left\Vert T_{5}\right\Vert .$
Since $T_{5}f\left( x\right) $ vanishes for $x\leq 0$, it suffices to
estimate $\left\Vert T_{5}\right\Vert $ in the space $L^{2}\left( \mathbb{R}%
_{+},\mu \right) .$

In what follows we redefine $I_{0}$ to be $I_{0}=\left[ 0,1\right] $, which
only reduces $\alpha _{0}$ and improves the estimates. Fix a non-negative
function $f\in L^{2}\left( \mathbb{R}_{+},\mu \right) $ and set for any
non-negative integer $n$ 
\begin{equation*}
w_{n}=\frac{1}{\sqrt{\alpha _{n}}}\int_{I_{n}}fd\mu .
\end{equation*}%
For any $x\in I_{n}$ we have%
\begin{equation*}
T_{5}f\left( x\right) \leq \int_{0}^{2^{n}}\left( 1+y\right) f\left(
y\right) d\mu \left( y\right) \leq \sum_{k=0}^{n}\left( 1+2^{k}\right)
\int_{I_{k}}fd\mu \leq \sum_{k=0}^{n}2^{k+1}w_{k}\sqrt{\alpha _{k}}.
\end{equation*}%
It follows that%
\begin{eqnarray*}
\left\Vert T_{5}f\right\Vert _{L^{2}\left( \mathbb{R}_{+},\mu \right) }^{2}
&=&\sum_{n=0}^{\infty }\int_{I_{n}}\left( T_{5}f\left( x\right) \right)
^{2}d\mu \left( x\right)  \\
&\leq &\sum_{n=0}^{\infty }\left( \sum_{k=0}^{n}2^{k+1}w_{k}\sqrt{\alpha _{k}%
}\right) ^{2}\mu \left( I_{n}\right)  \\
&=&4\sum_{n=0}^{\infty }\left( \sum_{k=0}^{n}2^{k}w_{k}\sqrt{\alpha _{k}}%
\right) ^{2}\frac{\alpha _{n}}{2^{n}}.
\end{eqnarray*}%
Using $\alpha _{n}\leq \alpha $, we obtain%
\begin{equation}
\left\Vert T_{5}f\right\Vert _{L^{2}\left( \mathbb{R}_{+},\mu \right)
}^{2}\leq 4\alpha ^{2}\sum_{n=0}^{\infty }\frac{1}{2^{n}}\left(
\sum_{k=0}^{n}2^{k}w_{k}\right) ^{2}.  \label{T5f}
\end{equation}%
On the other hand, we have%
\begin{eqnarray}
\left\Vert f\right\Vert _{L^{2}\left( \mathbb{R}_{+},\mu \right) }^{2}
&=&\sum_{n=0}^{\infty }\int_{I_{n}}f^{2}d\mu \geq \sum_{n=0}^{\infty }\frac{1%
}{\mu \left( I_{n}\right) }\left( \int_{I_{n}}fd\mu \right) ^{2}  \notag \\
&=&\sum_{n=0}^{\infty }\frac{2^{n}}{\alpha _{n}}w_{n}^{2}\alpha
_{n}=\sum_{n=0}^{\infty }2^{n}w_{n}^{2}.  \label{f25}
\end{eqnarray}%
Let us prove that%
\begin{equation}
\sum_{n=0}^{\infty }\frac{1}{2^{n}}\left( \sum_{k=0}^{n}2^{k}w_{k}\right)
^{2}\leq 16\sum_{n=0}^{\infty }2^{n}w_{n}^{2}.  \label{wHardy}
\end{equation}%
This is nothing other than a discrete weighted Hardy inequality. By \cite%
{Bennett}, if, for some $r>s\geq 1$ and for non-negative sequences $\left\{
u_{k}\right\} _{k=0}^{N},\left\{ v_{k}\right\} _{k=0}^{N}$, the following
inequality is satisfied%
\begin{equation}
\sum_{n=0}^{m}u_{n}\left( \sum_{k=0}^{n}v_{k}\right) ^{r}\leq \left(
\sum_{k=0}^{m}v_{k}\right) ^{s}\ \ \text{for }m=0,...,N,  \label{uvk}
\end{equation}%
then, for all non-negative sequences $\left\{ w_{k}\right\} _{k=0}^{N}$,%
\begin{equation}
\sum_{n=0}^{N}u_{n}\left( \sum_{k=0}^{n}v_{k}w_{k}\right) ^{r}\leq \left( 
\frac{r}{r-s}\right) ^{r}\left( \sum_{k=0}^{N}v_{k}w_{k}^{r/s}\right) ^{s}.
\label{rsr}
\end{equation}%
We apply this result with $r=2$, $s=1,$ $v_{n}=2^{k}$ and $u_{n}=2^{-n-2}.$
Then (\ref{uvk}) holds because 
\begin{equation*}
\sum_{n=0}^{m}2^{-n-2}\left( \sum_{k=0}^{n}2^{k}\right)
^{2}=\sum_{n=0}^{m}2^{-n-2}\left( 2^{n+1}-1\right) ^{2}\leq
\sum_{n=0}^{m}2^{n},
\end{equation*}%
and (\ref{rsr}) yields%
\begin{equation*}
\sum_{n=0}^{N}\frac{1}{2^{n+2}}\left( \sum_{k=0}^{n}2^{k}w_{k}\right)
^{2}\leq 4\sum_{k=0}^{N}2^{k}w_{k}^{2},
\end{equation*}%
which is equivalent to (\ref{wHardy}). The latter together with (\ref{T5f})
and (\ref{f25}) implies that $\left\Vert T_{5}\right\Vert \leq 8\alpha .$ It
follows that $\left\Vert T_{3}\right\Vert \leq 16\alpha $. As $T_{4}$ admits
the same estimate, we obtain $\left\Vert T_{1}\right\Vert \leq 32\alpha ,$
whence $\left\Vert T\right\Vert \leq 64\alpha ,$ which was to be proved.
\end{proof}

Consider the strip 
\begin{equation*}
S=\left\{ \left( x_{1},x_{2}\right) \in \mathbb{R}^{2}:0<x_{2}<\pi \right\}
\end{equation*}%
and its partition into rectangles $S_{n},$ $n\in \mathbb{Z}$, defined by%
\begin{equation}
S_{n}=\left\{ 
\begin{array}{ll}
\left\{ x\in S:2^{n-1}<x_{1}<2^{n}\right\} , & n>0, \\ 
\left\{ x\in S:-1<x_{1}<1\right\} , & n=0, \\ 
\left\{ x\in S:-2^{\left\vert n\right\vert }<x_{1}<-2^{\left\vert
n\right\vert -1}\right\} , & n<0.%
\end{array}%
\right.  \label{Sn}
\end{equation}

\begin{lemma}
\label{Leman}Let $\nu $ be a Radon measure on the strip $S$, absolutely
continuous with respect to the Lebesgue measure. For any $n\in \mathbb{Z}$,
set 
\begin{equation*}
a_{n}=2^{\left\vert n\right\vert }\nu \left( S_{n}\right) .
\end{equation*}%
Then the following integral operator 
\begin{equation*}
Tf\left( x\right) =\int_{S}\left( 1+\left\vert x_{1}\right\vert \wedge
\left\vert y_{1}\right\vert \right) f\left( y\right) d\nu \left( y\right)
\end{equation*}%
admits the following norm estimate in $L^{2}\left( S,\nu \right) $:%
\begin{equation}
\left\Vert T\right\Vert \leq 64\sup_{n\in \mathbb{Z}}a_{n}.  \label{Tan}
\end{equation}
\end{lemma}

\begin{proof}
Introduce a Radon measure $\mu $ on $\mathbb{R}$ by%
\begin{equation*}
\mu \left( A\right) =\nu \left( A\times \left( 0,\pi \right) \right) .
\end{equation*}%
For the quantities $\alpha _{n},$ defined for measure $\mu $ by (\ref{aln}),
we obviously have the identity $\alpha _{n}=a_{n}$. The estimate (\ref{Tan})
will follow from Lemma \ref{Lemaln} if we prove that $\left\Vert
T\right\Vert \leq \left\Vert T^{\prime }\right\Vert $ where $T^{\prime }$ is
the following operator in $L^{2}\left( \mathbb{R},\mu \right) $:%
\begin{equation*}
T^{\prime }f\left( t\right) =\int_{\mathbb{R}}\left( 1+t\wedge s\right)
f\left( s\right) d\mu \left( s\right) .
\end{equation*}%
It suffices to prove that, for any non-zero bounded function $f\in
L^{2}\left( S,\nu \right) $, there is a function $f^{\prime }\in L^{2}\left( 
\mathbb{R},\mu \right) $ such that%
\begin{equation*}
\frac{\left\Vert Tf\right\Vert }{\left\Vert f\right\Vert }\leq \frac{%
\left\Vert T^{\prime }f^{\prime }\right\Vert }{\left\Vert f^{\prime
}\right\Vert },
\end{equation*}%
where the norms are taken in the appropriate spaces. For any measurable set $%
A\subset \mathbb{R}$, set%
\begin{equation*}
\mu _{f}\left( A\right) =\int_{A\times \left( 0,\pi \right) }fd\nu .
\end{equation*}%
Since $f$ is bounded, measure $\mu _{f}$ is absolutely continuous with
respect to $\mu $, so that there exists a function $f^{\prime }$ such that $%
d\mu _{f}=f^{\prime }d\mu .$ In the same way, there is a function $f^{\prime
\prime }$ such that $d\mu _{f^{2}}=f^{\prime \prime }d\mu $. Since%
\begin{equation*}
\left( \int_{A\times \left( 0,\pi \right) }fd\nu \right) ^{2}\leq \left(
\int_{A\times \left( 0,\pi \right) }f^{2}d\nu \right) \mu \left( A\right) ,
\end{equation*}%
it follows that%
\begin{equation*}
\mu _{f}\left( A\right) ^{2}\leq \mu _{f^{2}}\left( A\right) \mu \left(
A\right)
\end{equation*}%
whence%
\begin{equation*}
\left( f^{\prime }\right) ^{2}\leq f^{\prime \prime }.
\end{equation*}%
It follows that 
\begin{equation*}
\left\Vert f^{\prime }\right\Vert _{L^{2}\left( \mathbb{R},\mu \right)
}^{2}=\int_{\mathbb{R}}\left( f^{\prime }\right) ^{2}d\mu \leq \int_{\mathbb{%
R}}f^{\prime \prime }d\mu =\int_{\mathbb{R}}\mu _{f^{2}}=\int_{S}f^{2}d\nu
=\left\Vert f^{2}\right\Vert _{L^{2}\left( S,\nu \right) }^{2}.
\end{equation*}

It remains to show that $\left\Vert Tf\right\Vert =\left\Vert T^{\prime
}f^{\prime }\right\Vert $. Using the notation $\tau \left( t,s\right)
=1+t\wedge s,$ we have\label{rem: details of Fubini}%
\begin{eqnarray*}
Tf\left( x\right) &=&\int_{S}\tau \left( x_{1},y_{1}\right) f\left( y\right)
d\nu \left( y\right) =\int_{\mathbb{R}}\tau \left( x_{1},y_{1}\right) d\mu
_{f}\left( y_{1}\right) \\
&=&\int_{\mathbb{R}}\tau \left( x_{1},y_{1}\right) f^{\prime }\left(
y_{1}\right) d\mu \left( y_{1}\right) =T^{\prime }f^{\prime }\left(
x_{1}\right) .
\end{eqnarray*}%
It follows that%
\begin{equation*}
\left\Vert Tf\right\Vert _{L^{2}\left( S,\nu \right) }^{2}=\int_{S}\left(
Tf\left( x\right) \right) ^{2}d\nu =\int_{\mathbb{R}}\left( T^{\prime
}f^{\prime }\left( x_{1}\right) \right) ^{2}d\mu =\left\Vert T^{\prime
}f^{\prime }\right\Vert _{L^{2}\left( \mathbb{R},\mu \right) }^{2},
\end{equation*}%
which finishes the proof.
\end{proof}

\section{Estimating the number of negative eigenvalues in a strip}

\label{SecW}\setcounter{equation}{0}Let $V$ be a potential in the strip 
\begin{equation*}
S=\left\{ \left( x_{1},x_{2}\right) \in \mathbb{R}^{2}:x_{1}\in \mathbb{R},\
0<x_{2}<\pi \right\} .
\end{equation*}
For any $n\in \mathbb{Z}$ set%
\begin{equation}
a_{n}\left( V\right) =\int_{S_{n}}\left( 1+\left\vert x_{1}\right\vert
\right) V\left( x\right) dx,  \label{anV}
\end{equation}%
where $S_{n}$ is defined by (\ref{Sn}). It is easy to see that%
\begin{equation}
2^{\left\vert n\right\vert -1}\int_{S_{n}}V\left( x\right) dx\leq
a_{n}\left( V\right) \leq 2^{\left\vert n\right\vert +1}\int_{S_{n}}V\left(
x\right) dx.  \label{an2n}
\end{equation}%
Fix $p>1$ and set also%
\begin{equation}
b_{n}\left( V\right) =\left( \int_{S\cap \left\{ n<x_{1}<n+1\right\}
}V^{p}\left( x\right) dx\right) ^{1/p}.  \label{bnV}
\end{equation}

\subsection{Condition for one negative eigenvalue}

\label{SecOneStrip}

\begin{lemma}
\label{LemGaV}The operator $\Gamma _{V}$ defined by \emph{(\ref{gadef})-(\ref%
{Gadef})} admits the following norm estimate in $L^{2}\left( S,Vdx\right) $:%
\begin{equation}
\left\Vert \Gamma _{V}\right\Vert \leq C\sup_{n\in \mathbb{Z}}a_{n}\left(
V\right) +C\sup_{x\in S}\int_{S}\ln _{+}\frac{1}{\left\vert x-y\right\vert }%
V\left( y\right) dy.  \label{normGaV}
\end{equation}%
where $C$ is an absolute constant. Consequently,%
\begin{equation}
\left\Vert \Gamma _{V}\right\Vert \leq C\sup_{n\in \mathbb{Z}}a_{n}\left(
V\right) +C_{p}\sup_{n\in Z}b_{n}\left( V\right) ,  \label{normGab}
\end{equation}%
where $C_{p}$ depends on $p$.
\end{lemma}

\begin{proof}
Consider measure $\nu $ in $S$ given by $d\nu =Vdx.$ By (\ref{ga<}) we have,
for any non-negative function $f$ on $S$,%
\begin{equation}
\Gamma _{V}f\left( x\right) \leq C\int_{S}\left( 1+\left\vert
x_{1}\right\vert \wedge \left\vert y_{1}\right\vert \right) f\left( y\right)
d\nu \left( y\right) +C\int_{S}\ln _{+}\frac{1}{\left\vert x-y\right\vert }%
f\left( y\right) d\nu \left( y\right) .  \label{G12}
\end{equation}%
By Lemma \ref{Leman}, the norm of the first integral operator in (\ref{G12})
is bounded by 
\begin{equation*}
64\sup_{n}2^{\left\vert n\right\vert }\nu \left( S_{n}\right) \leq
128\sup_{n}a_{n}\left( V\right) ,
\end{equation*}%
where we have used (\ref{an2n}). The norm of the second integral operator in
(\ref{G12}) is trivially bounded by%
\begin{equation*}
\sup_{x\in S}\int_{S}\ln _{+}\frac{1}{\left\vert x-y\right\vert }d\nu \left(
y\right) ,
\end{equation*}%
whence (\ref{normGaV}) follows.

For the second part, we have by the H\"{o}lder inequality%
\begin{eqnarray*}
\int_{S}\ln _{+}\frac{1}{\left\vert x-y\right\vert }V\left( y\right) dy
&\leq &\left( \int_{D_{1}\left( x\right) }\left( \ln _{+}\frac{1}{\left\vert
x-y\right\vert }\right) ^{p^{\prime }}dy\right) ^{1/p^{\prime }} \\
&&\times \left( \int_{D_{1}\left( x\right) \cap S}V^{p}\left( y\right)
dy\right) ^{1/p},
\end{eqnarray*}%
where $p^{\prime }$ is the H\"{o}lder conjugate to $p$ and $D_{1}\left(
x\right) $ is the disk of radius $1$ centered at $x$. The first integral is
equal to a finite constant depending only on $p$, but independent of $x$.
Since $D_{1}\left( x\right) \cap S$ is covered by at most $3$ rectangles $%
Q_{n}$, the second integral is bounded by $3\sup_{n}b_{n}\left( V\right) .$
Substituting into (\ref{normGaV}), we obtain (\ref{normGab}).
\end{proof}

\begin{remark}
\RM Although the norm of the first integral operator in (\ref{G12}) can be
trivially bounded by%
\begin{equation*}
\sup_{x\in S}\int_{S}\left( 1+\left\vert x_{1}\right\vert \wedge \left\vert
y_{1}\right\vert \right) d\nu \left( y\right) ,
\end{equation*}%
this estimate is weaker than the one by Lemma \ref{Leman} and is certainly
not good enough for our purposes.
\end{remark}

\begin{proposition}
\label{Pab}\label{Corab}There is a constant $c>0$ such that 
\begin{equation}
\sup_{n}a_{n}\left( V\right) \leq c\ \ \text{and\ \ \ }\sup_{n}b_{n}\left(
V\right) \leq c\ \ \Rightarrow \ \func{Neg}\left( V,S\right) =1.  \label{cc}
\end{equation}
\end{proposition}

\begin{proof}
Assume first that $\frac{1}{V}\in L_{loc}^{1}\left( S\right) $. By Lemma \ref%
{Lem1/8} it suffices to show that $\left\Vert \Gamma _{V}\right\Vert \leq 
\frac{1}{8}$. Assuming that the constant $c$ in (\ref{cc}) is small enough,
we obtain from (\ref{normGab}) that indeed $\left\Vert \Gamma
_{V}\right\Vert \leq \frac{1}{8}$ and, hence, $\func{Neg}\left( V,S\right)
=1.$

Consider now a general potential $V$. In this case consider bit larger
potential 
\begin{equation*}
V^{\prime }=V+\varepsilon e^{-\left\vert x\right\vert },
\end{equation*}%
where $\varepsilon >0$. Clearly, $\frac{1}{V^{\prime }}\in L_{loc}^{1}$
while $a_{n}\left( V^{\prime }\right) $ and $b_{n}\left( V^{\prime }\right) $
are still small enough provided $\varepsilon $ is chosen sufficiently small.
Assuming that the constant $c$ in (\ref{cc}) is small enough, we obtain by
the first part of the proof that 
\begin{equation}
\func{Neg}\left( 2V^{\prime },S\right) =1.  \label{2V'}
\end{equation}

We would like to deduce from (\ref{2V'}) that $\func{Neg}\left( V,S\right) =1
$. Since in general $\func{Neg}\left( V,S\right) $ is not monotone with
respect to $V$, we have to use an additional argument. We use the counting
function $\func{Neg}^{b}\left( V,S\right) $ based on bounded test functions
(cf. Section \ref{SecNegb}). 

Observe first that%
\begin{equation}
\func{Neg}^{b}\left( 2V,S\right) \leq \func{Neg}^{b}\left( 2V^{\prime
},S\right) .  \label{VV'}
\end{equation}%
Since $\mathcal{E}_{2V,S}\leq \mathcal{E}_{2V^{\prime },S}$, (\ref{VV'})
will follow from the identity of the spaces $\mathcal{F}_{2V,S}^{b}$ and $%
\mathcal{F}_{2V^{\prime },S}^{b}$, where the latter amounts to%
\begin{equation*}
\int_{\Omega }Vf^{2}dx<\infty \ \Leftrightarrow \ \int_{\Omega }V^{\prime
}f^{2}dx<\infty .
\end{equation*}%
The implication $\Leftarrow $ here is trivial, while the opposite direction $%
\Rightarrow $ follows from%
\begin{equation*}
\int_{\Omega }V^{\prime }f^{2}dx=\int_{\Omega }Vf^{2}dx+\varepsilon
\int_{\Omega }f^{2}e^{-\left\vert x\right\vert }dx
\end{equation*}%
and the finiteness of the last integral, which is true by the boundedness of
the test function $f$. 

Since 
\begin{equation*}
\func{Neg}^{b}\left( 2V^{\prime },S\right) \leq \func{Neg}\left( 2V^{\prime
},S\right) ,
\end{equation*}%
combining this with (\ref{2V'}) and (\ref{VV'}) we obtain   
\begin{equation*}
\func{Neg}^{b}\left( 2V,S\right) =1.
\end{equation*}%
Finally, we conclude by Lemma  \ref{Lem2V} that $\func{Neg}\left( V,S\right)
=1.$
\end{proof}

\subsection{Extension of functions from a rectangle to a strip}

\label{SecRec}For all $\alpha \in \lbrack -\infty ,+\infty ),$ $\beta \in
(-\infty ,+\infty ]$ such that $\alpha <\beta $, denote by $P_{\alpha ,\beta
}$ the rectangle%
\begin{equation*}
P_{\alpha ,\beta }=\left\{ \left( x_{1},x_{2}\right) \in \mathbb{R}%
^{2}:\,\alpha <x_{1}<\beta ,\ \ 0<x_{2}<\pi \right\} .
\end{equation*}

\begin{lemma}
\label{LemPab}For any potential $V$ in a rectangle $P_{\alpha ,\beta }$ with 
$\beta -\alpha \geq 1$, we have%
\begin{equation}
\func{Neg}\left( V,P_{\alpha ,\beta }\right) \leq \func{Neg}\left(
17V,S\right) ,  \label{5V}
\end{equation}%
assuming that $V$ is extended to $S$ by setting $V=0$ outside $P_{\alpha
,\beta }$.
\end{lemma}

\begin{proof}
By Lemma \ref{LemL} it suffices to show that any function $u\in \mathcal{F}%
_{V,P}$ can be extended to a function $u\in \mathcal{F}_{V,S}$ so that%
\begin{equation}
\int_{S}\left\vert \nabla u\right\vert ^{2}dx\leq 17\int_{P}\left\vert
\nabla u\right\vert ^{2}dx.  \label{u5}
\end{equation}%
Assume first that both $\alpha ,\beta $ are finite. Attach to $P$ from each
side one rectangle, say $P^{\prime }$ from the left and $P^{\prime \prime }$
from the right, each having the length $4\left( \beta -\alpha \right) $ (to
ensure that the latter is $>\pi $). Extend function $u$ to $P^{\prime }$ by
applying four times symmetries in the vertical sides (cf. Example \ref{ExL}%
). Then we have%
\begin{equation*}
\int_{P^{\prime }}\left\vert \nabla u\right\vert ^{2}dx=4\int_{P}\left\vert
\nabla u\right\vert ^{2}dx.
\end{equation*}%
\FRAME{ftbhFU}{5.7009in}{1.3629in}{0pt}{\Qcb{Extension of function $u$ from $%
P$ to $S$.}}{\Qlb{pic2}}{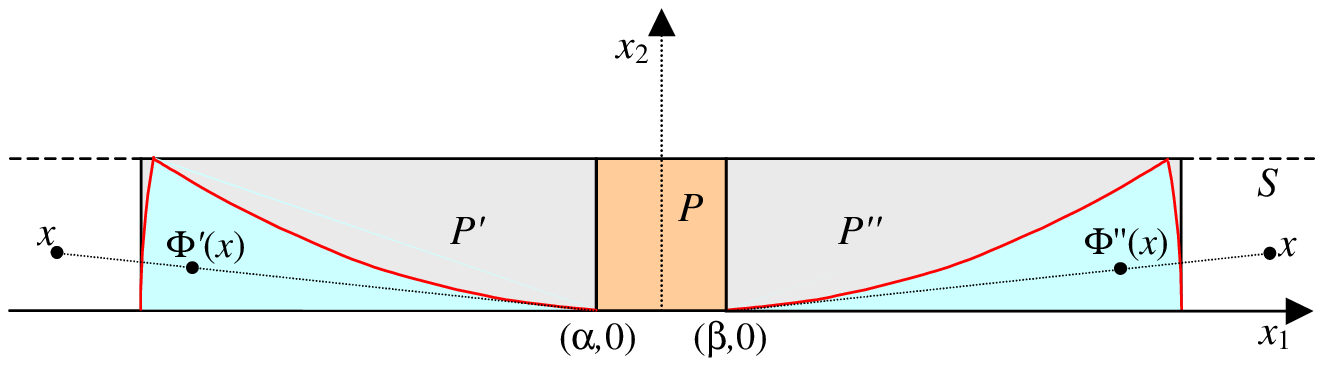}{\special{language "Scientific Word";type
"GRAPHIC";maintain-aspect-ratio TRUE;display "USEDEF";valid_file "F";width
5.7009in;height 1.3629in;depth 0pt;original-width 6.9574in;original-height
1.6769in;cropleft "0";croptop "1";cropright "1";cropbottom "0";filename
'pic2.eps';file-properties "XNPEU";}}

Then slightly reduce $P^{\prime }$ by taking its intersections with the disk
of radius $\beta -\alpha $ centered at $\left( \alpha ,0\right) $ (cf. Fig. %
\ref{pic2}). Now we extend $u$ from $P^{\prime }$ to the left by using the
inversion $\Phi ^{\prime }$ at the point $\left( \alpha ,0\right) $ in the
circle of radius $\beta -\alpha $ centered at $\left( \alpha ,0\right) $
(cf. Example \ref{ExC}). By the conformal invariance of the Dirichlet
integral, we have%
\begin{equation*}
\int_{S\cap \left\{ x_{1}<\alpha \right\} }\left\vert \nabla u\right\vert
^{2}\leq 8\int_{P}\left\vert \nabla u\right\vert ^{2}dx.
\end{equation*}%
Extending $u$ in the same way to the right of $P$, we obtain (\ref{u5}). The
case when one of the endpoints $\alpha ,\beta $ is at infinity is treated
similarly.
\end{proof}

\subsection{Sparse potentials}

\label{SecSparse}

\begin{definition}
\RM We say that a potential $V$ in $S$ is \emph{sparse} if%
\begin{equation}
\sup_{n}b_{n}\left( V\right) <c_{0},  \label{sparse}
\end{equation}%
where $c_{0}$ is a small enough positive constant, depending only on $p.$ We
say that a potential $V$ is sparse in a domain $\Omega \subset S$ if its
trivial extension to $S$ is sparse.
\end{definition}

Let us choose $c_{0}$ smaller that the constant $c$ from (\ref{cc}). It
follows from Proposition \ref{Pab} that, for a sparse potential,%
\begin{equation*}
\sup_{n}a_{n}\left( V\right) \leq c\ \Rightarrow \func{Neg}\left( V,S\right)
=1.
\end{equation*}%
Consider some estimates for $\func{Neg}\left( V,\Omega \right) $ for sparse
potentials.

\begin{corollary}
\label{CorPr}Let $V$ be a sparse potential on a rectangle $P_{\alpha ,\beta
} $ with $\beta -\alpha \geq 1$. Then%
\begin{equation}
\left( \beta -\alpha \right) \int_{P_{\alpha ,\beta }}V\left( x\right)
dx\leq c\Rightarrow \func{Neg}\left( V,P_{\alpha ,\beta }\right) =1,
\label{PabV}
\end{equation}%
where $c$ is a positive constant depending only on $p$.
\end{corollary}

\begin{proof}
By shifting $P_{\alpha .\beta }$ and $V$ along the axis $x_{1}$, we can
assume that $\alpha =0,$ so that $\beta \geq 1$. Let $m$ be a non-negative
integer such that $2^{m-1}<\beta \leq 2^{m}$ (cf. Fig. \ref{pic9}).\FRAME{%
ftbphFU}{5.7009in}{1.2868in}{0pt}{\Qcb{Rectangle $P_{0,\protect\beta }$ is
covered by the sequence $S_{n},0\leq n\leq m$}}{\Qlb{pic9}}{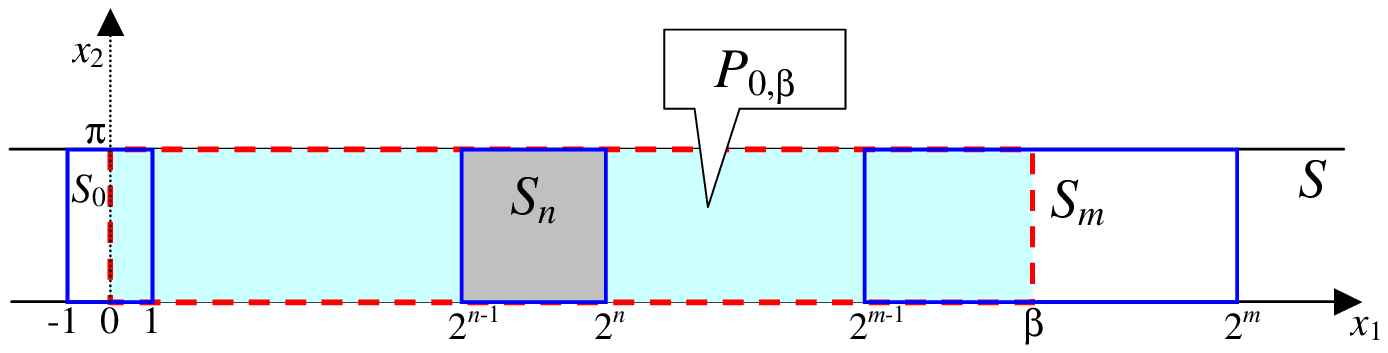}{%
\special{language "Scientific Word";type "GRAPHIC";maintain-aspect-ratio
TRUE;display "USEDEF";valid_file "F";width 5.7009in;height 1.2868in;depth
0pt;original-width 6.3027in;original-height 1.3984in;cropleft "0";croptop
"1";cropright "1";cropbottom "0";filename 'pic9.eps';file-properties
"XNPEU";}}

Then $a_{n}\left( V\right) =0$ for $n<0$ and for $n\geq m+1$. For $0\leq
n\leq m$ we have by (\ref{an2n})%
\begin{equation}
a_{n}\left( V\right) \leq 2^{n+1}\int_{S_{n}}V\left( x\right) dx\leq
2^{m+1}\int_{P_{0,\beta }}V\left( x\right) dx\leq 4\beta \int_{P_{0,\beta
}}V\left( x\right) dx.  \label{anV<}
\end{equation}%
The hypotheses (\ref{PabV}) with small enough $c$ and (\ref{anV<}) imply
that $a_{n}\left( 17V\right) $ are sufficiently small for all $n\in \mathbb{Z%
}$. By Proposition \ref{Pab} we obtain $\func{Neg}\left( 17V,S\right) =1,$
and by Lemma \ref{LemPab} $\func{Neg}\left( V,P_{0,\beta }\right) =1.$
\end{proof}

The next statement is the main technical lemma about sparse potentials.

\begin{lemma}
\label{LemLS}\label{PBS}Let $V$ be a sparse potential in a rectangle $%
P_{\alpha ,\beta }$ with $\beta -\alpha \geq 1$. Then%
\begin{equation}
\func{Neg}\left( V,P_{\alpha ,\beta }\right) \leq 1+C\left( \left( \beta
-\alpha \right) \int_{P_{\alpha ,\beta }}V\left( x\right) dx\right) ^{1/2},
\label{x1}
\end{equation}%
where the constant $C$ depends only on $p.$ In particular, for any $n\in 
\mathbb{Z},$%
\begin{equation}
\func{Neg}\left( V,S_{n}\right) \leq 1+C\sqrt{a_{n}\left( V\right) }.
\label{ran}
\end{equation}
\end{lemma}

\begin{proof}
Without loss of generality set $\alpha =0.$ Set also%
\begin{equation*}
J=\int_{P_{0,\beta }}V\left( x\right) dx
\end{equation*}%
and recall that, by Corollary \ref{CorPr}, if $\beta J\leq c$ for
sufficiently small $c$ then $\func{Neg}\left( V,P_{0,\beta }\right) =1.$
Hence, in this case (\ref{x1}) is trivially satisfied, and we assume in the
sequel that $\beta J>c$.

Due to Lemma \ref{LemPab}, it suffices to prove the estimate%
\begin{equation*}
\func{Neg}\left( V,S\right) \leq C\left( \beta J\right) ^{1/2},
\end{equation*}%
assuming that $V$ vanishes outside $P_{0,\beta }$. Consider a sequence of
reals $\left\{ r_{k}\right\} _{k=0}^{N}$ such that%
\begin{equation*}
0=r_{0}<r_{1}<...<r_{N-1}<\beta \leq r_{N}
\end{equation*}%
and the corresponding sequence of rectangles%
\begin{equation*}
R_{k}:=P_{r_{k-1},r_{k}}=\left\{ \left( x_{1},x_{2}\right)
:r_{k-1}<x_{1}<r_{k},\ \ 0<x_{2}<\pi \right\}
\end{equation*}%
where $k=1,...,N$, that covers $P_{0,\beta }$ (see Fig. \ref{pic8}). \FRAME{%
ftbphFU}{5.7009in}{1.2868in}{0pt}{\Qcb{The sequence $\left\{ R_{k}\right\}
_{k=1}^{N}$ of rectangles covering $P_{0,\protect\beta }$}}{\Qlb{pic8}}{%
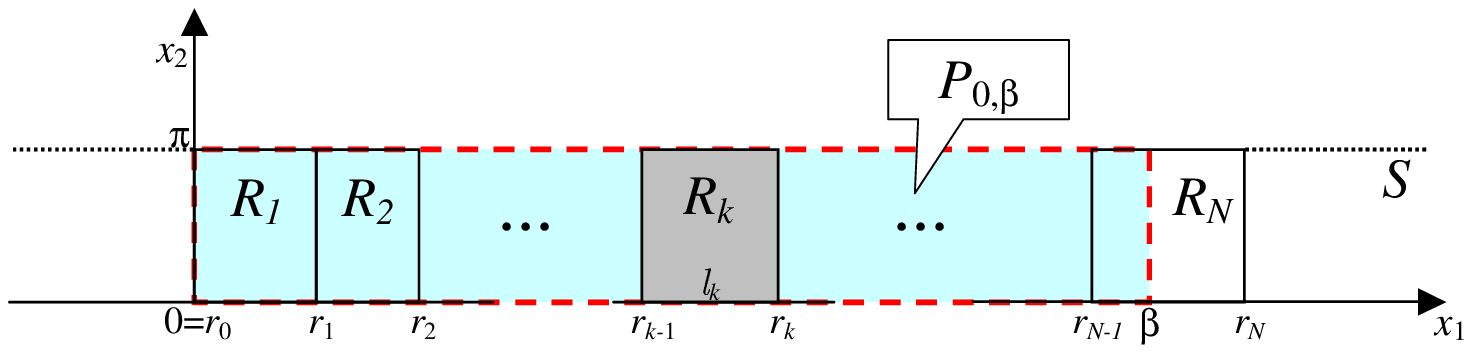}{\special{language "Scientific Word";type
"GRAPHIC";maintain-aspect-ratio TRUE;display "USEDEF";valid_file "F";width
5.7009in;height 1.2868in;depth 0pt;original-width 6.3027in;original-height
1.3984in;cropleft "0";croptop "1";cropright "1";cropbottom "0";filename
'pic8.eps';file-properties "XNPEU";}}

Denote $l_{k}=r_{k}-r_{k-1}$ and 
\begin{equation*}
J_{k}=\int_{R_{k}}V\left( x\right) dx.
\end{equation*}%
By Corollary \ref{CorPr}, if 
\begin{equation}
l_{k}\geq 1\ \text{and\ \ }l_{k}J_{k}\leq c  \label{LJ}
\end{equation}%
then 
\begin{equation*}
\func{Neg}\left( V,R_{k}\right) =1.
\end{equation*}

It is easy to construct the sequence $\left\{ r_{k}\right\} $ inductively so
that both conditions in (\ref{LJ}) are satisfied for all $k=1,...,N$ (where $%
N$ is yet to be determined). If $r_{k-1}$ is already defined and is smaller
than $\beta $ then choose $r_{k}>r_{k-1}$ to satisfy the identity%
\begin{equation}
l_{k}J_{k}=c.  \label{rV}
\end{equation}%
If such $r_{k}$ does not exist then set $r_{k}=\beta +1$; in this case, we
have%
\begin{equation*}
l_{k}J_{k}<c.
\end{equation*}%
Let us show that in the both cases $l_{k}=r_{k}-r_{k-1}\geq 1$. Indeed, if $%
l_{k}<1$ then $r_{k}<\beta +1$ so that (\ref{rV}) is satisfied. Using the H%
\"{o}lder inequality, (\ref{rV}) and $l_{k}<1$, we obtain%
\begin{equation}
\left( \int_{R_{k}}V^{p}dx\right) ^{1/p}\geq \frac{1}{\left( \pi
l_{k}\right) ^{1/p^{\prime }}}\int_{R_{k}}Vdx=\frac{c}{\left( \pi
l_{k}\right) ^{1/p^{\prime }}l_{k}}\geq \frac{c}{\pi ^{1/p^{\prime }}}.
\label{c/2}
\end{equation}%
However, if the constant $c_{0}$ in the definition (\ref{sparse}) of a
sparse potential is small enough, then we obtain that (\ref{c/2}) and (\ref%
{sparse}) contradict each other, which proves that $l_{k}\geq 1$.

As soon as we reach $r_{k}\geq \beta $ we stop the process and set $N=k.$
Since always $l_{k}\geq 1$, the process will indeed stop in a finite number
of steps.

We obtain a partition of $S$ into $N$ rectangles $R_{1},...,R_{N}$ and two
half-strips: $S\cap \left\{ x_{1}<0\right\} $ and $S\cap \left\{
x_{1}>r_{N}\right\} $, and in the both half-strips we have $V\equiv 0.$ In
each $R_{k}$ we have $\func{Neg}\left( V,R_{k}\right) =1$ whence it follows
that%
\begin{equation*}
\func{Neg}\left( V,S\right) \leq 2+\sum_{k=1}^{N}\func{Neg}\left(
V,R_{k}\right) =N+2.
\end{equation*}%
Let us estimate $N$ from above. In each $R_{k}$ with $k\leq N-1$ we have by (%
\ref{rV}) $\frac{1}{J_{k}}=\frac{1}{c}l_{k}.$ Therefore, we have%
\begin{equation*}
N-1=\sum_{k=1}^{N-1}\frac{1}{\sqrt{J_{k}}}\sqrt{J_{k}}\leq \left( \frac{1}{c}%
\sum_{k=1}^{N-1}l_{k}\right) ^{1/2}\left( \sum_{k=1}^{N-1}J_{k}\right)
^{1/2}\leq \left( \frac{1}{c}\beta \right) ^{1/2}J^{1/2}.
\end{equation*}%
Using also $3\leq 3\left( \frac{1}{c}\beta J\right) ^{1/2}$, we obtain $%
N+2\leq 4\left( c^{-1}\beta J\right) ^{1/2}$, which finishes the proof of (%
\ref{x1}).

The estimate (\ref{ran}) follows trivially from (\ref{x1}). Indeed, $S_{n}$
is a rectangle $P_{\alpha ,\beta }$ with the length $1\leq \beta -\alpha
\leq 2^{\left\vert n\right\vert +1}.$ Using (\ref{x1}) and (\ref{an2n}), we
obtain%
\begin{equation*}
\func{Neg}\left( V,S_{n}\right) \leq 1+C\left( 2^{n+1}\int_{S_{n}}V\left(
x\right) dx\right) ^{1/2}\leq 1+C^{\prime }\sqrt{a_{n}\left( V\right) }
\end{equation*}%
with $C^{\prime }=2C$, which proves (\ref{ran}).
\end{proof}

\begin{proposition}
\label{Propran}For any sparse potential $V$ in the strip $S$,%
\begin{equation}
\func{Neg}\left( V,S\right) \leq 1+C\sum_{\left\{ n:a_{n}\left( V\right)
>c\right\} }\sqrt{a_{n}\left( V\right) },  \label{sumran}
\end{equation}%
for some constant $C,c>0$ depending only on $p$.
\end{proposition}

\begin{proof}
Let us enumerate in the increasing order those values $n$ where $a_{n}\left(
V\right) >c.$ So, we obtain an increasing sequence $\left\{ n_{i}\right\} $,
finite or infinite, such that $a_{n_{i}}\left( V\right) >c$ for any index $i$%
. The difference $S\setminus \tbigcup_{i}S_{n_{i}}$ can be partitions into a
sequence $\left\{ T_{j}\right\} $ of rectangles, where each rectangle $T_{j}$
either fills the gap in $S$ between successive rectangles $%
S_{n_{i}},S_{n_{i+1}}$ as on Fig. \ref{pic1} or $T_{j}$ may be a half-strip
that fills the gap between $S_{n_{i}}$ and $+\infty $ or $-\infty ,$ when $%
n_{i}$ is the maximal, respectively minimal, value in the sequence $\left\{
n_{i}\right\} $.\FRAME{ftbphFU}{5.7389in}{0.9729in}{0pt}{\Qcb{Partitioning
of the strip $S$ into rectangles $S_{n_{i}}$ and $T_{j}$}}{\Qlb{pic1}}{%
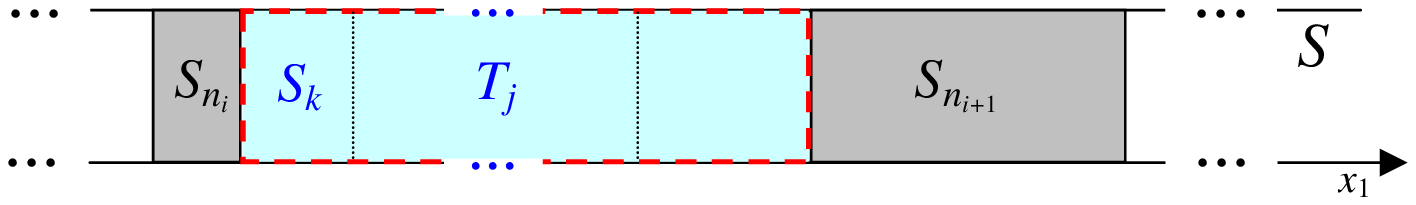}{\special{language "Scientific Word";type
"GRAPHIC";maintain-aspect-ratio TRUE;display "USEDEF";valid_file "F";width
5.7389in;height 0.9729in;depth 0pt;original-width 6.5959in;original-height
1.1632in;cropleft "0";croptop "1";cropright "1";cropbottom "0";filename
'pic1.eps';file-properties "XNPEU";}}

By construction, each $T_{j}$ is a union of some rectangles $S_{k}$ with $%
a_{k}\left( V\right) \leq c.$ Consider the potential $V_{j}=V\mathbf{1}%
_{T_{j}}$. By Lemma \ref{LemPab} we have 
\begin{equation*}
\func{Neg}\left( V,T_{j}\right) =\func{Neg}\left( V_{j},T_{j}\right) \leq 
\func{Neg}\left( 17V_{j},S\right) .
\end{equation*}%
For those $k$ where $S_{k}\subset T_{j}$, we have $a_{k}\left(
17V_{j}\right) \leq 17c$, while $a_{k}\left( 17V_{j}\right) =0$ otherwise.
Assuming that $c$ is small enough, we conclude by Proposition \ref{Pab} that 
$\func{Neg}\left( 17V_{j},S\right) =1$ and, hence,\textbf{\ }$\func{Neg}%
\left( V,T_{j}\right) =1$.

Since by construction%
\begin{equation*}
\#\left\{ T_{j}\right\} \leq 1+\#\left\{ S_{n_{i}}\right\} ,
\end{equation*}%
it follows that%
\begin{eqnarray*}
\func{Neg}\left( V,S\right) &\leq &\sum_{j}\func{Neg}\left( V,T_{i}\right)
+\sum_{i}\func{Neg}\left( V,S_{n_{i}}\right) \\
&\leq &1+\#\left\{ S_{n_{i}}\right\} +\sum_{i}\func{Neg}\left(
V,S_{n_{i}}\right) \\
&\leq &1+2\sum_{i}\func{Neg}\left( V,S_{n_{i}}\right) .
\end{eqnarray*}%
In each $S_{n_{i}}$ we have by (\ref{ran}) and $a_{n_{i}}\left( V\right) >c$
that%
\begin{equation*}
\func{Neg}\left( V,S_{n_{i}}\right) \leq C\sqrt{a_{n_{i}}\left( V\right) }.
\end{equation*}%
Substituting into the previous estimate, we obtain (\ref{sumran}).
\end{proof}

\subsection{Arbitrary potentials in a strip}

We use notation $a_{n}\left( V\right) $ and $b_{n}\left( V\right) $ defined
by (\ref{anV}) and (\ref{bnV}), respectively.

\begin{theorem}
\label{TStrip}For any $p>1$ and for any potential $V$ in the strip $S$, we
have%
\begin{equation}
\func{Neg}\left( V,S\right) \leq 1+C\sum_{\left\{ n\in \mathbb{Z}%
:a_{n}\left( V\right) >c\right\} }\sqrt{a_{n}\left( V\right) }%
+C\sum_{\left\{ n\in \mathbb{Z}:b_{n}\left( V\right) >c\right\} }b_{n}\left(
V\right) ,  \label{NVSmain}
\end{equation}%
where the positive constants $C,c$ depend only on $p.$
\end{theorem}

\begin{proof}
Define $Q_{n}=S\cap \left\{ n<x_{1}<n+1\right\} $ so that 
\begin{equation*}
b_{n}\left( V\right) =\left( \int_{Q_{n}}V^{p}dx\right) ^{1/p}.
\end{equation*}%
Let $\left\{ n_{i}\right\} $ be a sequence of all those $n\in \mathbb{Z}$
for which 
\begin{equation}
b_{n}\left( V\right) >c,  \label{sat}
\end{equation}%
where $c$ is a positive constant whose value will be determined below. If
this sequence is empty then the potential $V$ is sparse, and (\ref{NVSmain})
follows from Proposition \ref{Propran}.

Assume in the sequel that the sequence $\left\{ n_{i}\right\} $ is
non-empty. Denote by $\left\{ T_{j}\right\} $ a sequence of rectangles that
fill the gaps in $S$ between successive rectangles $Q_{n_{i}}$ or between
one of $Q_{n_{i}}$ and $\pm \infty $ (cf. Fig. \ref{pic10}).\FRAME{ftbphFU}{%
5.713in}{0.774in}{0pt}{\Qcb{Partitioning of the strip $S$ into rectangles $%
Q_{n_{i}}$ and $T_{j}$}}{\Qlb{pic10}}{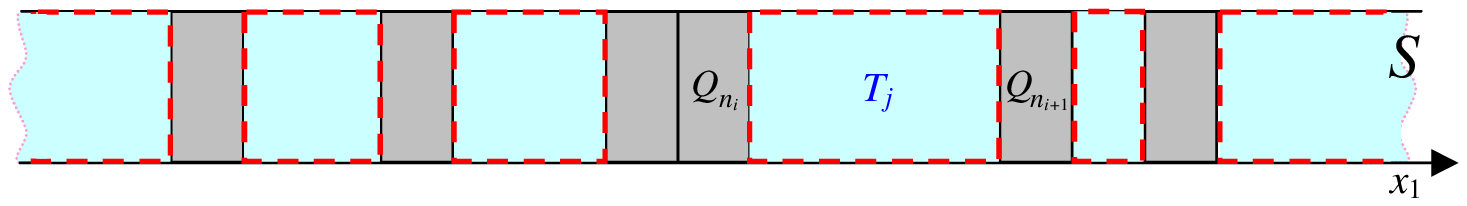}{\special{language
"Scientific Word";type "GRAPHIC";maintain-aspect-ratio TRUE;display
"USEDEF";valid_file "F";width 5.713in;height 0.774in;depth
0pt;original-width 6.5959in;original-height 1.1632in;cropleft "0";croptop
"1";cropright "1";cropbottom "0";filename 'pic10.eps';file-properties
"XNPEU";}}

Consider the potentials $V^{\prime }=V\mathbf{1}_{\cup T_{j}}$ and $%
V^{\prime \prime }=V\mathbf{1}_{\cup Q_{n_{i}}}$. Since $V=V^{\prime
}+V^{\prime \prime }$, by Lemma \ref{LemV1V2} we obtain%
\begin{equation*}
\func{Neg}\left( V,S\right) \leq \func{Neg}\left( 2V^{\prime },S\right) +%
\func{Neg}\left( 2V^{\prime \prime },S\right) .
\end{equation*}%
The potential $2V^{\prime }$ is sparse by construction provided the constant 
$c$ in (\ref{sat}) is small enough. Hence, we obtain by Proposition \ref%
{Propran}%
\begin{equation}
\func{Neg}\left( 2V^{\prime },S\right) \leq 1+C\sum_{\left\{ n:a_{n}\left(
V^{\prime }\right) >c\right\} }\sqrt{a_{n}\left( V^{\prime }\right) }.
\label{V'}
\end{equation}%
By Lemma \ref{LemSub} and Lemma \ref{LemMain}, we obtain 
\begin{eqnarray*}
\func{Neg}\left( 2V^{\prime \prime },S\right) &\leq &\sum_{j}\func{Neg}%
\left( 2V^{\prime \prime },T_{j}\right) +\sum_{i}\func{Neg}\left( 2V^{\prime
\prime },Q_{n_{i}}\right) \\
&=&\#\left\{ T_{j}\right\} +\sum_{i}\left( 1+C\left\Vert 2V^{\prime \prime
}\right\Vert _{L^{p}\left( Q_{n_{i}}\right) }\right) \\
&=&\#\left\{ T_{j}\right\} +\#\left\{ Q_{n_{i}}\right\}
+2C\sum_{i}b_{n_{i}}\left( V\right) .
\end{eqnarray*}%
By construction we have $\#\left\{ T_{j}\right\} \leq 1+\#\left\{
Q_{n_{i}}\right\} .$ By the choice of $n_{i}$, we have $1<c^{-1}b_{n_{i}}%
\left( V\right) ,$ whence%
\begin{eqnarray*}
\#\left\{ T_{j}\right\} +\#\left\{ Q_{n_{i}}\right\} &\leq &1+2\#\left\{
Q_{n_{i}}\right\} \\
&\leq &1+2c^{-1}\sum_{i}b_{n_{i}}\left( V\right) \leq
3c^{-1}\sum_{i}b_{n_{i}}\left( V\right)
\end{eqnarray*}%
Combining these estimates together, we obtain%
\begin{equation}
\func{Neg}\left( 2V^{\prime \prime },S\right) \leq C^{\prime
}\sum_{i}b_{n_{i}}\left( V\right) =C^{\prime }\sum_{\left\{ n:b_{n}\left(
V\right) >c\right\} }b_{n}\left( V\right)  \label{V''}
\end{equation}%
Adding up (\ref{V'}) and (\ref{V''}) yields%
\begin{equation}
\func{Neg}\left( V,S\right) \leq 1+C\sum_{\left\{ n:a_{n}\left( V^{\prime
}\right) >c\right\} }\sqrt{a_{n}\left( V^{\prime }\right) }+C\sum_{\left\{
n:b_{n}\left( V\right) >c\right\} }b_{n}\left( V\right) .  \label{NVSmain'}
\end{equation}%
Since $V^{\prime }\leq V$, (\ref{NVSmain'}) implies (\ref{NVSmain}), which
finishes the proof.
\end{proof}

\begin{remark}
\RM In fact, we have proved a slightly better inequality (\ref{NVSmain'})
than (\ref{NVSmain}).
\end{remark}

\section{Negative eigenvalues in $\mathbb{R}^{2}$}

\setcounter{equation}{0}\label{SecProof}Here we prove the main Theorem \ref%
{Tmainp}. Recall that Theorem \ref{Tmainp} states the following: for any
potential $V$ in $\mathbb{R}^{2}$,%
\begin{equation}
\func{Neg}\left( V,\mathbb{R}^{2}\right) \leq 1+C\sum_{\left\{ n\in \mathbb{Z%
}:A_{n}>c\right\} }\sqrt{A_{n}}+C\sum_{\left\{ n\in \mathbb{Z}%
:B_{n}>c\right\} }B_{n},  \label{Ng0}
\end{equation}%
where $A_{n}$ and $B_{n}$ are defined in (\ref{An}) and (\ref{Bn}), and $c,C$
are positive constants that depend only on $p>1$.

\begin{proof}[Proof of Theorem \protect\ref{Tmainp}]
Consider an open set $\Omega =\mathbb{R}^{2}\setminus L$ where $L=\left\{ \
x_{1}\geq 0,x_{2}=0\right\} $ is a ray. By Lemma \ref{Lem-K} we have%
\begin{equation}
\func{Neg}\left( V,\mathbb{R}^{2}\right) \leq \func{Neg}\left( V,\Omega
\right) .  \label{Ng1}
\end{equation}%
The function $\Psi \left( z\right) =\ln z$ is holomorphic in $\Omega $ and
provides a biholomorphic mapping from $\Omega $ onto the strip%
\begin{equation*}
\widetilde{S}=\left\{ \left( y_{1},y_{2}\right) \in \mathbb{R}^{2}:\ \
0<y_{2}<2\pi \right\}
\end{equation*}%
(see Fig. \ref{pic6}).\FRAME{ftbphFU}{5.61in}{1.5956in}{0pt}{\Qcb{Conformal
mapping $\Psi :\Omega \rightarrow \widetilde{S}$}}{\Qlb{pic6}}{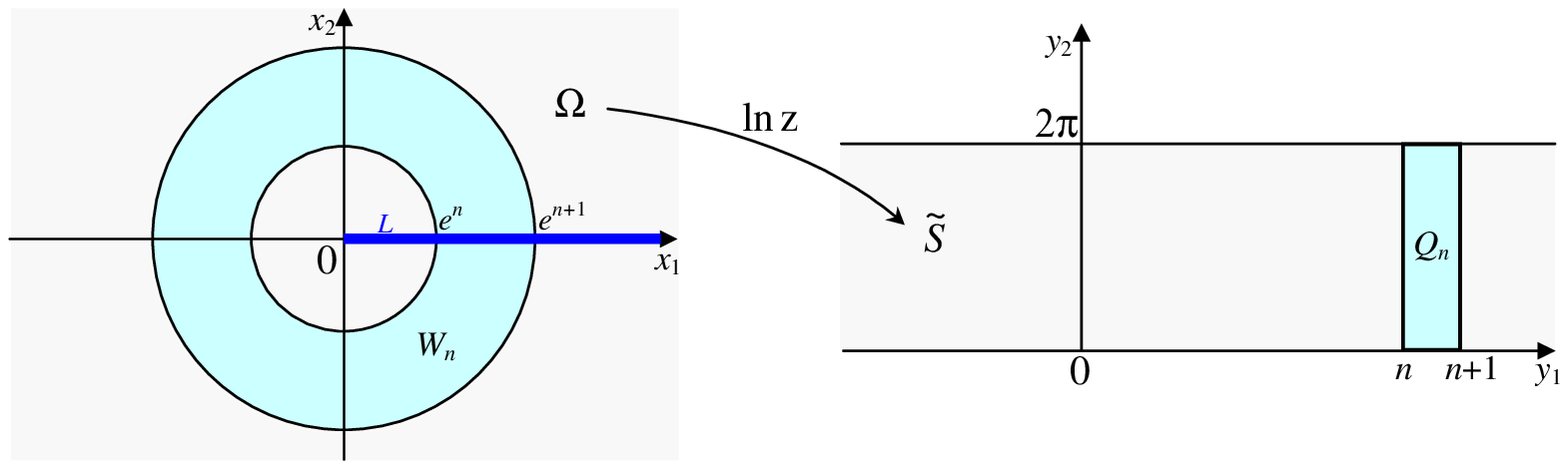}{%
\special{language "Scientific Word";type "GRAPHIC";maintain-aspect-ratio
TRUE;display "USEDEF";valid_file "F";width 5.61in;height 1.5956in;depth
0pt;original-width 6.5665in;original-height 1.8447in;cropleft "0";croptop
"1";cropright "1";cropbottom "0";filename 'pic6.eps';file-properties
"XNPEU";}}

Let $\widetilde{V}$ be a push-forward of $V$ under $\Psi $ (cf. by (\ref{Vti}%
)), so that by Lemma \ref{LemVW}%
\begin{equation}
\func{Neg}\left( V,\Omega \right) =\func{Neg}(\widetilde{V},\widetilde{S}).
\label{Ng2}
\end{equation}%
Since $\widetilde{S}$ and the strip $S$ from Section \ref{SecW} are
bilipschitz equivalent, the estimate (\ref{NVSmain}) of Theorem \ref{TStrip}
holds also for $\widetilde{S}$, that is,%
\begin{equation}
\func{Neg}(\widetilde{V},\widetilde{S})\leq 1+C\sum_{\left\{
n:a_{n}>c\right\} }\sqrt{a_{n}}+C\sum_{\left\{ n:b_{n}\left( V\right)
>c\right\} }b_{n},  \label{Ng3}
\end{equation}%
where we use the following notation: 
\begin{equation*}
a_{n}=\int_{S_{n}}\left( 1+\left\vert y_{1}\right\vert \right) \widetilde{V}%
\left( y\right) dy,\ \ \ b_{n}=\left( \int_{Q_{n}}\widetilde{V}^{p}dy\right)
^{1/p},
\end{equation*}%
where%
\begin{equation*}
S_{n}=\left\{ 
\begin{array}{ll}
\left\{ y\in \widetilde{S}:2^{n-1}<y_{1}<2^{n}\right\} , & n>0, \\ 
\left\{ y\in \widetilde{S}:-1<y_{1}<1\right\} , & n=0, \\ 
\left\{ y\in \widetilde{S}:-2^{\left\vert n\right\vert
}<y_{1}<-2^{\left\vert n\right\vert -1}\right\} , & n<0.%
\end{array}%
\right.
\end{equation*}%
and%
\begin{equation*}
Q_{n}=\left\{ y\in \widetilde{S}:n<y_{1}<n+1\right\} .
\end{equation*}%
Consider also the rings $U_{n}$ and $W_{n}$ in $\mathbb{R}^{2}$ defined by (%
\ref{Un}) and (\ref{Wn}). Obviously, we have%
\begin{equation*}
\Psi \left( U_{n}\setminus L\right) =S_{n}\ \ \ \text{and \ \ }\Psi \left(
W_{n}\setminus L\right) =Q_{n}.
\end{equation*}%
Since $J_{\Psi }=\left\vert \Psi ^{\prime }\left( x\right) \right\vert ^{2}=%
\frac{1}{\left\vert x\right\vert ^{2}}$, where $x$ is treated as a complex
variable, we obtain by (\ref{VW}) \label{rem: more details} that 
\begin{eqnarray}
b_{n}^{p} &=&\int_{Q_{n}}\widetilde{V}^{p}\left( y\right)
dy=\int_{W_{n}}V^{p}\left( x\right) \left\vert J_{\Psi }\left( x\right)
\right\vert ^{1-p}dx  \notag \\
&=&\int_{W_{n}}V^{p}\left( x\right) \left\vert x\right\vert ^{2\left(
p-1\right) }dx=B_{n}^{p}.  \label{In+}
\end{eqnarray}%
Since for $y=\Psi \left( x\right) $ we have $y_{1}=\func{Re}\ln x=\ln
\left\vert x\right\vert $, it follows from (\ref{VW}) that 
\begin{equation*}
a_{n}=\int_{S_{n}}\widetilde{V}\left( y\right) \left( 1+\left\vert
y_{1}\right\vert \right) dy=\int_{U_{n}}V\left( x\right) \left( 1+\left\vert
\ln \left\vert x\right\vert \right\vert \right) dx=A_{n}.
\end{equation*}%
Combining together (\ref{Ng1}), (\ref{Ng2}), (\ref{Ng3}), we obtain (\ref%
{Ng0}).
\end{proof}

\begin{proof}[Proof of Corollary \protect\ref{Coral}]
If a stronger hypothesis%
\begin{equation*}
\sum_{n\in \mathbb{Z}}\sqrt{A_{n}\left( V\right) }+\sum_{n\in \mathbb{Z}%
}B_{n}\left( V\right) <\infty
\end{equation*}%
is satisfied then (\ref{NVA}) is an immediate consequence of (\ref{NegIn}).
To prove (\ref{NVA}) under the hypothesis (\ref{An+Bn}), we need an improved
version of (\ref{NegIn}). Let us come back to the proof of Theorem \ref%
{Tmainp} and use instead of (\ref{Ng3}) the estimate (\ref{NVSmain'}), that
is, 
\begin{equation*}
\func{Neg}(\widetilde{V},\widetilde{S})\leq 1+C\sum_{\left\{ n:a_{n}(%
\widetilde{V}^{\prime })>c\right\} }\sqrt{a_{n}(\widetilde{V}^{\prime })}%
+C\sum_{\left\{ n:b_{n}(\widetilde{V})>c\right\} }b_{n}(\widetilde{V}),
\end{equation*}%
where $\widetilde{V}^{\prime }$ is a modification of $\widetilde{V}$ that
vanishes on the rectangles $Q_{n}$ with $b_{n}(\widetilde{V})>c.$ Then we
obtain instead of (\ref{Ng0}) the following estimate:%
\begin{equation}
\func{Neg}\left( V,\mathbb{R}^{2}\right) \leq 1+C\sum_{\left\{ n\in \mathbb{Z%
}:A_{n}\left( V^{\prime }\right) >c\right\} }\sqrt{A_{n}\left( V^{\prime
}\right) }+C\sum_{\left\{ n\in \mathbb{Z}:B_{n}\left( V\right) >c\right\}
}B_{n}\left( V\right) ,  \label{NegIn'}
\end{equation}%
where $V^{\prime }$ is a modification of $V$ that vanishes on the annuli $%
W_{n}$ with $B_{n}\left( V\right) >c.$ As it was explained in Introduction, (%
\ref{NegIn'}) implies%
\begin{equation}
\func{Neg}\left( V,\mathbb{R}^{2}\right) \leq 1+C\int_{\mathbb{R}%
^{2}}V^{\prime }\left( x\right) \left( 1+\left\vert \ln \left\vert
x\right\vert \right\vert \right) dx+C\sum_{\left\{ n\in \mathbb{Z}%
:B_{n}\left( V\right) >c\right\} }B_{n}\left( V\right) .  \label{NV'}
\end{equation}%
Let us apply (\ref{NV'}) to the potential $\alpha V$ with $\alpha
\rightarrow \infty .$ Denote by $W\left( \alpha \right) $ the union of all
annuli $W_{n}$ with $B_{n}\left( \alpha V\right) \leq c$, so that $\left(
\alpha V\right) ^{\prime }=\alpha V\mathbf{1}_{W\left( \alpha \right) }$.
Then (\ref{NV'}) implies%
\begin{equation}
\func{Neg}\left( \alpha V,\mathbb{R}^{2}\right) \leq 1+C\alpha \int_{\mathbb{%
R}^{2}}V\left( x\right) \mathbf{1}_{W\left( \alpha \right) }\left(
1+\left\vert \ln \left\vert x\right\vert \right\vert \right) dx+C\alpha
\sum_{n\in \mathbb{Z}}B_{n}\left( V\right) .  \label{alV'}
\end{equation}%
For any $n$ with $B_{n}\left( V\right) >0$, the condition $B_{n}\left(
\alpha V\right) >c$ will be satisfied for large enough $\alpha $, so that
for such $\alpha $ the function $V\mathbf{1}_{W\left( \alpha \right) }$
vanishes on $W_{n}$. If $B_{n}\left( V\right) =0$ then $V=0$ on $W_{n}$ and,
hence, $V\mathbf{1}_{W\left( \alpha \right) }=0$ on $W_{n}$ again. We see
that $V\mathbf{1}_{W\left( \alpha \right) }\rightarrow 0$ $\mathrm{a.e.}$ as 
$\alpha \rightarrow \infty $, and by the dominated convergence theorem%
\begin{equation*}
\int_{\mathbb{R}^{2}}V\left( x\right) \mathbf{1}_{W\left( \alpha \right)
}\left( 1+\left\vert \ln \left\vert x\right\vert \right\vert \right)
dx\rightarrow 0.
\end{equation*}
Substituting into (\ref{alV'}) we obtain (\ref{NVA}).
\end{proof}

\begin{proof}[Proof of Corollary \protect\ref{CorW}]
Let us estimate the both terms in the right hand side of (\ref{Zn}) using
the H\"{o}lder inequality. For the first term we have%
\begin{equation*}
\int_{\mathbb{R}^{2}}V\left( x\right) \left( 1+\left\vert \ln \left\vert
x\right\vert \right\vert \right) dx\leq \left( \int_{\mathbb{R}^{2}}V^{p}%
\mathcal{W}dx\right) ^{1/p}\left( \int_{\mathbb{R}^{2}}\frac{\left(
1+\left\vert \ln \left\vert x\right\vert \right\vert \right) ^{p^{\prime }}}{%
\left( \mathcal{W}\left( \left\vert x\right\vert \right) \right) ^{\frac{%
p^{\prime }}{p}}}dx\right) ^{1/p^{\prime }}.
\end{equation*}%
The second integral can be computed in the polar coordinates and it is equal
to%
\begin{equation*}
\int_{\mathbb{R}^{2}}\frac{\left( 1+\left\vert \ln r\right\vert \right) ^{%
\frac{p}{p-1}}}{\mathcal{W}\left( r\right) ^{\frac{1}{p-1}}}2\pi rdr,
\end{equation*}%
which is finite by (\ref{W4}). Hence, we obtain that%
\begin{equation}
\int_{\mathbb{R}^{2}}V\left( x\right) \left( 1+\left\vert \ln \left\vert
x\right\vert \right\vert \right) dx\leq C\left( \int_{\mathbb{R}^{2}}V^{p}%
\mathcal{W}dx\right) ^{1/p}  \label{lnVp}
\end{equation}%
To estimate the second term in (\ref{Zn}), take any sequence $\left\{
l_{n}\right\} _{n\in \mathbb{Z}}$ of positive reals and write%
\begin{eqnarray*}
\sum_{n}B_{n} &=&\sum_{n}l_{n}^{-1/p}l_{n}^{1/p}\left(
\int_{W_{n}}V^{p}\left\vert x\right\vert ^{2\left( p-1\right) }dx\right)
^{1/p} \\
&\leq &\left( \sum_{n}l_{n}^{-\frac{1}{p-1}}\right) ^{1/p^{\prime }}\left(
\sum_{n}l_{n}\int_{W_{n}}V^{p}\left\vert x\right\vert ^{2\left( p-1\right)
}dx\right) ^{1/p}.
\end{eqnarray*}%
Choose here%
\begin{equation*}
l_{n}=\frac{\mathcal{W}\left( e^{n}\right) }{\left( e^{n+1}\right) ^{2\left(
p-1\right) }}
\end{equation*}%
so that, for $x\in W_{n}$,%
\begin{equation*}
l_{n}\left\vert x\right\vert ^{2\left( p-1\right) }\leq \mathcal{W}\left(
e^{n}\right) \leq \mathcal{W}\left( \left\vert x\right\vert \right)
\end{equation*}%
and%
\begin{equation*}
\sum_{n}l_{n}\int_{W_{n}}V^{p}\left( x\right) \left\vert x\right\vert
^{2\left( p-1\right) }dx\leq \int_{\mathbb{R}^{2}}V^{p}\left( x\right) 
\mathcal{W}\left( \left\vert x\right\vert \right) dx.
\end{equation*}%
On the other hand, we have%
\begin{equation*}
\sum_{n}l_{n}^{-\frac{1}{p-1}}=\sum_{n}\frac{e^{2\left( n+1\right) }}{%
\mathcal{W}\left( e^{n}\right) ^{\frac{1}{p-1}}}\simeq \int_{0}^{\infty }%
\frac{rdr}{\mathcal{W}\left( r\right) ^{\frac{1}{p-1}}}<\infty
\end{equation*}%
by (\ref{W4}). Hence,%
\begin{equation}
\sum_{n}B_{n}\leq C\left( \int_{\mathbb{R}^{2}}V^{p}\mathcal{W}dx\right)
^{1/p}.  \label{bnn}
\end{equation}%
Substituting (\ref{lnVp}) and (\ref{bnn}) into (\ref{Zn}), we obtain (\ref%
{NegW}).
\end{proof}

%


\addcontentsline{toc}{section}{References}
{\footnotesize
\begin{thebibliography}{9}

\def\Revista{Revista Matem\'atica Iberoamericana}
\def\IHP{Heat kernels and analysis on manifolds, graphs, and metric spaces}
\def\JDGsurveys{Surveys in Differential Geometry}
\def\GAFA{Geom. Funct. Anal.}
\def\Russian{(in Russian)\ }
\def\appear#1{to appear in {\sl {#1}}}
\def\au#1{{\bf{#1}}, }
\def\boo#1{{\rm {``#1''}, }}
\def\inboo#1{{\sl in:} {\rm ``#1'',}\ }
\def\eng{Engl. transl.:\ }
\def\jo#1{{\sl {#1}},\ }
\def\no#1{no.{#1},\ }
\def\pa#1{\, {#1}.}
\def\pab#1{\, {#1}.}
\def\pano#1{}
\def\pbh#1{{#1},\ }
\def\preprint{pre\-print }
\def\ser#1{{#1},\ }
\def\vser#1{#1}
\def\ti#1{{#1},\ }
\def\vo#1{{\bf {#1} }\ }
\def\ya#1{(#1)\ }
\def\yab#1{#1.}
\def\yano#1{}
\def\h#1{{\accent94 #1}}
\def\edt#1{ed. {#1},\ }
\def\ISBN#1{}
\def\comment#1{}
\def\MR#1#2{}
\def\CMP#1#2{}
\def\other#1{#1}
\def\libunibie#1{}
\def\onlineres#1{}
\def\onlinemy#1{}
\def\arxiv#1#2{arXiv:#1, #2.}


\bibitem[1]{Bennett}{\au{Bennett G.}
     \ti{Some elementary inequalities}
     \jo{Quart. J. Math. Oxford (2)}\vo{38}\ya{1987}\pa{401-425}
  }

\bibitem[2]{BirLap}{\au{Birman M.Sh., Laptev A.}
     \ti{The negative discrete spectrum of a two-dimensional Schr\"odinger operator}
     \jo{Comm. Pure Appl. Math.} \vo{49}\ya{1996}\no{9}\pa{967-997} 
  }
\bibitem[3]{BLSone}{\au{Birman M.Sh., Laptev A., Solomyak M.}
     \ti{On the eigenvalue behaviour for a class of differential operators on semiaxis}
     \jo{Math. Nachr.} \vo{195}\ya{1998}\pa{17-46} 
  }

\bibitem[4]{BirSolTr}{\au{Birman M.Sh., Solomyak M.Z.}
      \ti{Spectral asymptotics of nonsmooth elliptic operators. I, II.} 
      \jo{Trans. Moscow Math. Soc.}\vo{27}\ya{1972}\pa{1-52} 
      \other{and} \vo{28}\ya{1973}\pa{1-32}
  }
\bibitem[5]{CKMW}{\au{Chadan K., Khuri N.N., Martin A., Wu Tai Tsun}
    \ti{Bound states in one and two spatial dimensions}
    \jo{J. Math. Physics}\vo{44}\ya{2003}\no{2}\pa{406-422}
  }
\bibitem[6]{Cwikel}{\au{Cwikel W.}
    \ti{Weak type estimates for singuar values and the number of bound states of Schr\"odinger operators}
    \jo{Ann. Math.}\vo{106}\ya{1977}\pa{93--100}
  }      
\bibitem[7]{Fefferman}{\au{Fefferman Ch.L.}
     \ti{The uncertainty principle}
     \jo{Bull. Amer. Math. Soc.}\vo{9}\ya{1983}\no{2}\pa{129-206}
  }
\bibitem[8]{GrigW}{\au{Grigor'yan A.}
     \ti{Heat kernels on weighted manifolds and applications}
     \jo{Contemporary Mathematics}\vo{398}\ya{2006}\pa{93-191}
     \comment{jorg}\comment{wma}
  }
\bibitem[9]{GrigNetYau}{\au{Grigor'yan A., Netrusov Yu., Yau S.-T.}
    \ti{Eigenvalues of  elliptic operators and geometric applications}
    \inboo{Eigenvalues of Laplacians and other geometric operators}
    \ser{Surveys in Differential Geometry \vser{IX}}\ya{2004}\pa{147-218}
    \comment{eigsch}\comment{esch}\comment{GrigYauSch}
  }
\bibitem[10]{GrigYauHigh}{\au{Grigor'yan A., Yau S.-T.}
    \ti{Isoperimetric properties of higher eigenvalues of elliptic operator}
    \jo{Amer. J. Math}\vo{125}\ya{2003}\pa{893-940}
    \comment{heigen}
  }  
\bibitem[11]{KMW}{\au{Khuri N.N., Martin A., Wu Tai Tsun}
    \ti{Bound states in $n$ dimensions (especially $n=1$ and $n=2$)}
    \jo{Few-Body Systems}\vo{31}\ya{2002}\pa{83-89}
  }
\bibitem[12]{LapSol}{\au{Laptev A., Solomyak M.}
    \ti{On spectral estimates for two-dimensional Schr\"odinger operators}
    \preprint\arxiv{1201.3074v1}{2012}
  }
\bibitem[13]{LapSolrad}{\au{Laptev A., Solomyak M.}
    \ti{On the negative spectrum of the two-dimensional Schr\"odinger operator with radial potential}
    \preprint\arxiv{1108.1002v3}{2011}
  }
\bibitem[14]{LevSolom}{\au{Levin D., Solomyak M.}
     \ti{The Rozenblum-Lieb-Cwikel inequality for Markov generators}
     \jo{J. d'Analyse Math.}\vo{71}\ya{1997}\pa{173-193}
  }  
\bibitem[15]{LiYauE}{\au{Li P., Yau S.-T.}
    \ti{On the Schr\"odinger equation and the eigenvalue problem}
    \jo{Comm. Math. Phys.}\vo{88}\ya{1983}\pa{309--318}
  }  
\bibitem[16]{LiebE}{\au{Lieb E.H.}
     \ti{Bounds on the eigenvalues of the Laplace and Schr\"odinger operators}
     \jo{Bull. Amer. Math. Soc.}\vo{82}\ya{1976}\pa{751-753}
  }
\bibitem[17]{Lieb}{\au{Lieb E.H.}
     \ti{The number of bound states of one-body Schr\"odinger operators and the Weyl problem}
     \jo{Proc. Sym. Pure Math.}\vo{36}\ya{1980}\pa{241-252}
  }
\bibitem[18]{MelRoz}{\au{Melgaard M, Rozenblum G.V.}
     \ti{Spectral estimates for magnetic operators} 
     \jo{Math. Scand.} \vo{79}\ti{1996}\no{2}\pa{237-254} 
  }
\bibitem[19]{MolchVain}{\au{Molchanov S., Vainberg B.}
     \ti{On negative eigenvalues of low-dimensional Schr{\"o}dinger operators}
     \preprint\yab{2010}
  }
\bibitem[20]{NaimSol}{\au{Naimark K., Solomyak M.}
     \ti{Regular and pathological eigenvalue behavior for the equation $-\lambda u''=Vu$ on the semiaxis}
     \jo{J. Funct. Anal.}\vo{151}\ya{1997}\pa{504-530}
  }

\bibitem[21]{RozSolSmall}{\au{Rozenblum G., Solomyak M.}
    \ti{On spectral estimates for Schr\"odinger-type operators: the case of small local dimension}
    \jo{Functional Analysis and Its Applications}\vo{44}\ya{2010}\no{4}\pa{259-269}
  }
\bibitem[22]{Rozen}{\au{Rozenblum G.V.}
    \ti{The distribution of the discrete spectrum for singular differential operators}
    \jo{Dokl. Akad. Nauk SSSR}\vo{202}\ya{1972}\pa{1012-1015}
  }
\bibitem[23]{RozenVUZ}{\au{Rozenblum G.V.}
    \ti{Distribution of the discrete spectrum of singular differential operators}
    \jo{Soviet Math. (Iz. VUZ)}\vo{20}\ya{1976}\no{1}\pa{63-71}
  }
\bibitem[24]{Solom}{\au{Solomyak M.}
    \ti{Piecewise-polynomial approximation of functions from $H^{\ell}((0, 1)^d),\ 2\ell=d,$
        and applications to the spectral theory of the Schr{\"o}dinger operator}
    \jo{Israel J. Math.}\vo{86} \ya{1994}\pa{253-275}
  }

\end{thebibliography}
}
\end{document}